\documentclass[11pt]{article}
\usepackage{amsmath,amssymb,amsthm,bm,graphicx,color,epsfig,enumerate,caption}
\usepackage{mathtools}

\usepackage{tikz}
\usetikzlibrary{shapes,arrows,matrix,patterns,positioning}
\usepackage[norelsize,boxed,linesnumbered,vlined,algo2e]{algorithm2e} 
\usepackage{algorithm}
\usepackage{algorithmic}
\usepackage{amssymb}
\usepackage{epstopdf}
\usepackage{float}
\usepackage{hyperref}
\usepackage{booktabs}
\usepackage{multirow}
\usepackage{array}
\usepackage{bm}

\newtheorem{theorem}{Theorem}[section]
\newtheorem{lemma}[theorem]{Lemma}

\newtheorem{definition}[theorem]{Definition}

\newcommand{\R}{\mathbb{R}}
\newcommand{\N}{\mathcal{N}}
\newcommand{\T}{\mathcal{T}}
\newcommand{\LL}{\mathcal{L}^C}
\newcommand{\WD}{\mathcal{WD}}

\newcommand{\eps}{\varepsilon}

\newcommand{\defeq}{\vcentcolon=}
\newcommand{\E}{\mathrm{E}}

\newcommand{\Var}{\mathrm{Var}}

\newcommand{\SNR}{\text{SNR}}

\usepackage{mathtools}

\begin{document}

\title{Recursive Diffeomorphism-Based Regression for Shape Functions}

\author{Jieren Xu$^\dagger$, Haizhao Yang$^*$ and Ingrid Daubechies$^\dagger$\\
  \vspace{0.1in}\\
  $^\dagger$ Department of Mathematics, Duke University\\
  $^*$Department of Mathematics, National University of Singapore\\
}

\date{October, 2016; revised July 2017}
\maketitle

\begin{abstract}
  This paper proposes a recursive diffeomorphism-based regression method for the one-dimensional generalized mode decomposition problem that aims at extracting generalized modes $\alpha_k(t)s_k(2\pi N_k\phi_k(t))$ from their superposition $\sum_{k=1}^K \alpha_k(t)s_k(2\pi N_k\phi_k(t))$. We assume that the instantaneous information, e.g., $\alpha_k(t)$ and $N_k\phi_k(t)$, is determined by, e.g., a one-dimensional synchrosqueezed transform or some other methods. Our main contribution is to propose a novel approach based on diffeomorphisms and nonparametric regression to estimate wave shape functions $s_k(t)$. This leads to a framework for the generalized mode decomposition problem under a weak well-separation condition. Numerical examples of synthetic and real data are provided to demonstrate the successful application of our approach.

\end{abstract}



\section{Introduction}
\label{sec:intro}

The analysis of oscillatory data is a ubiquitous challenge, arising in a wide range of applications including but not limited to medicine (like ECG and EEG readings \cite{HauBio2}), physical science (e.g., gravitational waves \cite{Chirplet1}, atomic crystal images \cite{Crystal}), mechanical engineering (such as vibration measurements \cite{Eng2}), finance, geology (e.g., seismic data analysis \cite{GeoReview,SSCT}), art investigation \cite{Canvas}, and audio signals (including speech and music recordings \cite{musicShape}). Although different problems depend on different interpretation of the data measurement, it is common that one wants to extract certain time-varying features or conduct adaptive component analysis. For this purpose, a typical model is to assume that the oscillatory data $f(t)$ consists of a superposition of several (but typically reasonably few) oscillatory modes like
\begin{equation}
\label{P1}
f(t)=\sum_{k=1}^K \alpha_k(t) e^{2\pi i N_k \phi_k(t)},
\end{equation}
where $\alpha_k(t)$ is the instantaneous amplitude, $2\pi N_k \phi_k(t)$ is the instantaneous phase and $N_k\phi_k'(t)$ is the instantaneous frequency. Analyzing instantaneous properties (e.g., instantaneous frequencies, instantaneous amplitudes and instantaneous phases) and decomposing the signal $f(t)$ into several modes $\alpha_k(t) e^{2\pi iN_k \phi_k(t)}$ have been an important topic for over two decades. Many methods, have been proposed to address this mode decomposition problem, e.g. empirical mode decomposition methods \cite{Huang1998}, time-frequency reassignment methods \cite{Auger1995,Chassande-Mottin2003}, synchrosqueezed transforms \cite{Daubechies2011}, adaptive optimization \cite{VMD,Hou2012}, recursive filtering \cite{iterativeFilter2,iterativeFilter1}, and data-driven time-frequency decomposition \cite{Chui2016,EWT}.

\begin{figure}
  \begin{center}
    \begin{tabular}{cc}
      \includegraphics[height=2.4in]{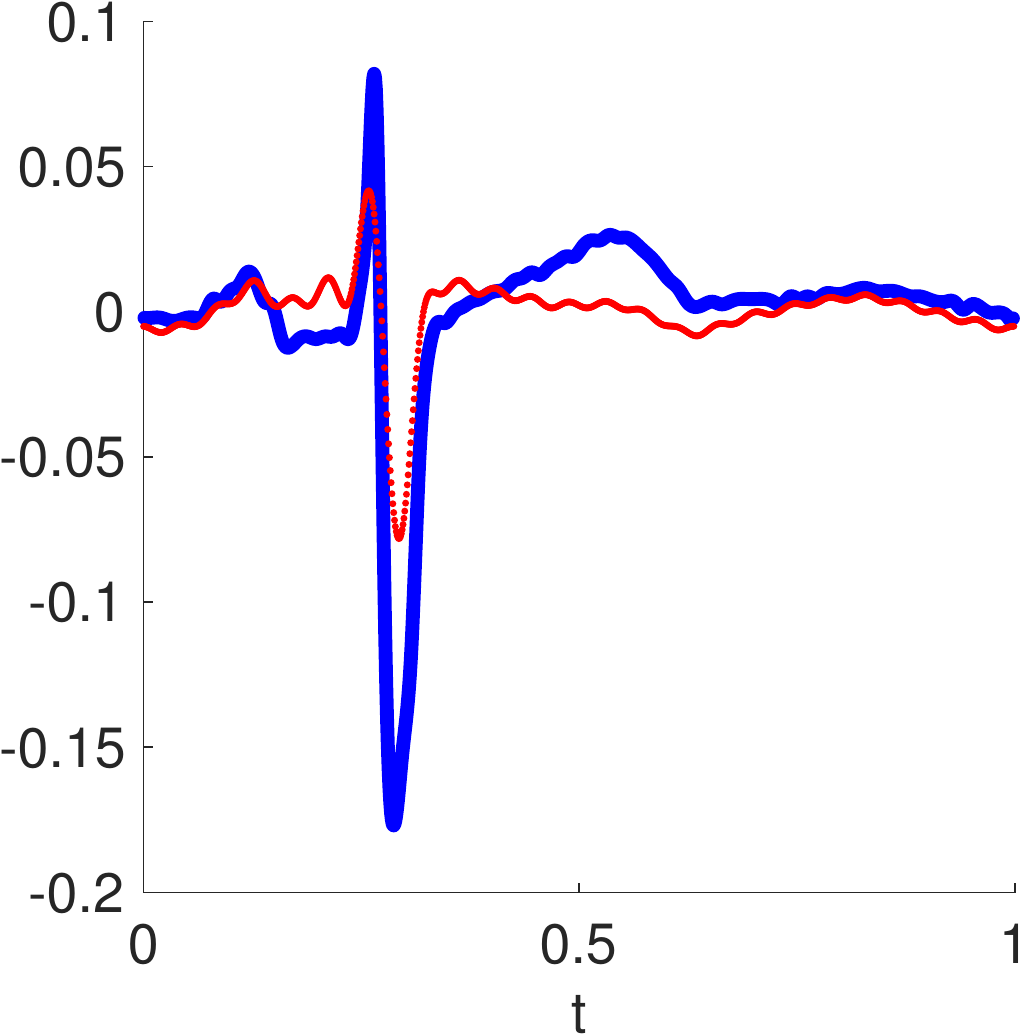} &
      \includegraphics[height=2.4in]{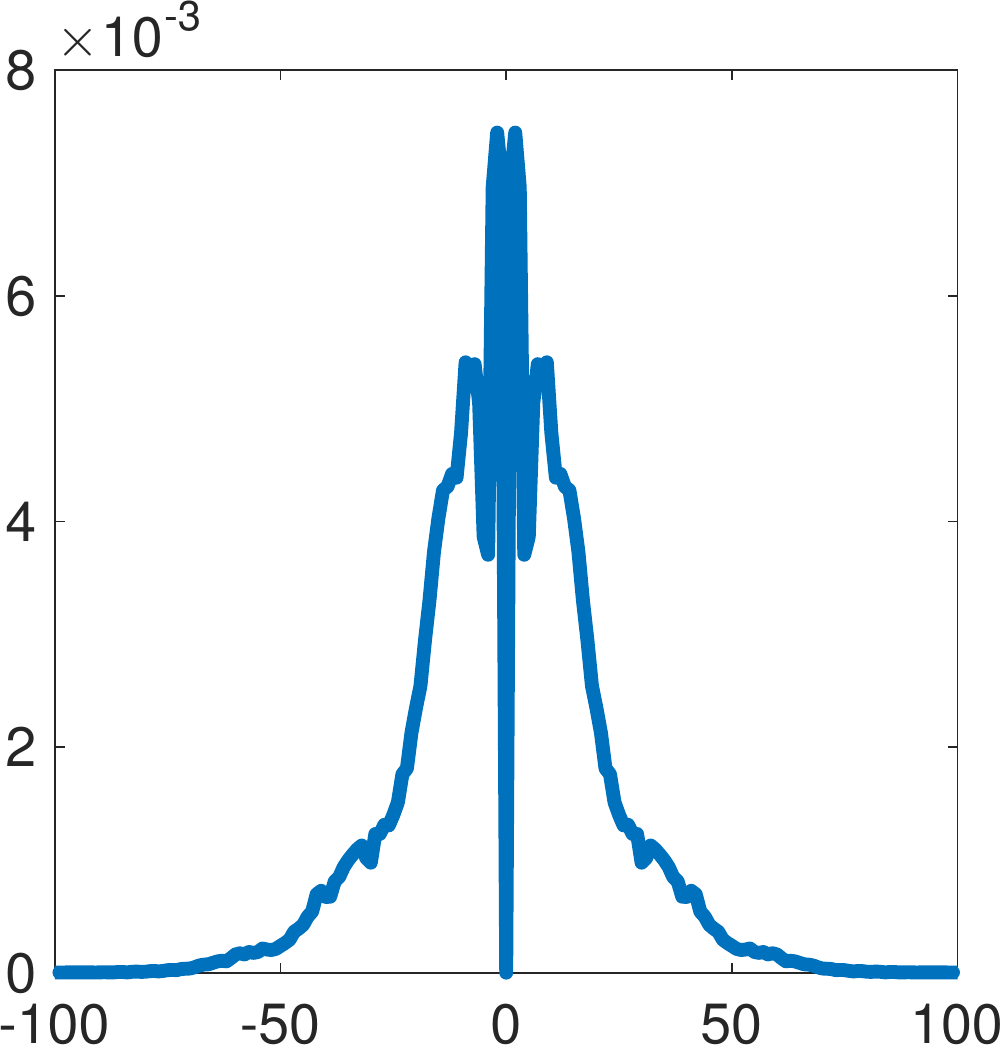}
    \end{tabular}
  \end{center}
  \caption{Left: A generalized shape function $s(t)$ of a real ECG signal (in blue) and its band-limited approximation $\sum_{|n|\leq 10} \widehat{s}(n) e^{2\pi i n t}$ (in red). Right: The Fourier power spectrum $|\widehat{s}(\xi)|$ of $s(t)$. As seen in the real ECG shape function in blue, there are three major peaks (called P, R, and T peaks from left to right, respectively) that are valuable in medical study. This figure illustrates that a high sampling rate is needed to identify these peaks accurately, and that a band-limited approximation to the shape function loses important information.}
  \label{fig:ECGshape}
\end{figure}

In spite of considerable success in modeling oscillatory data of the form \eqref{P1}, modes with sinusoidal oscillatory patterns like $\alpha_k(t) e^{2\pi iN_k \phi_k(t)}$ are not longer sufficiently adaptive to characterize complicated features in the data. This motivates the generalized mode decomposition problem of the form
\begin{equation}
\label{P2}
f(t)=\sum_{k=1}^K f_k(t) = \sum_{k=1}^K \alpha_k(t) s_k(2 \pi N_k \phi_k(t)),
\end{equation}
where $\{s_k(t)\}_{1\leq k\leq K}$ are $2\pi$-periodic generalized shape functions. 
For example, the oscillatory pattern in the electrocardiography (ECG) signal contains information of the electrical pathway inside the heart, the respiration, and the heart anatomy, which is embedded in a  generalized shape function as shown in Figure \ref{fig:ECGshape} (left). For another example, different kinds of timbre of different music instruments result from different wave shapes in music signals (see Figure \ref{fig:musicShape}). The Fourier expansion of generalized shape functions results  in
\begin{equation}
f(t)= \sum_{k=1}^K \alpha_k(t) s_k(2 \pi N_k \phi_k(t))=\sum_{k=1}^K \sum_{n=-\infty}^{\infty} \widehat{s_k}(n)\alpha_k(t) e^{2\pi i n N_k \phi_k(t)}.
\end{equation}
Hence, in another point of view, the generalized mode decomposition problem comes from the motivation that combining modes with similar oscillatory patterns of the form \eqref{P1} leads to a more adaptive and physically more meaningful decomposition of the form \eqref{P2}. When generalized shape functions are not band-limited, the generalized mode decomposition problem is challenging.

Although various methods have been proposed for the mode decomposition problem, the generalized mode decomposition problem is relatively recent and there are few solutions. Existing solutions assume that the instantaneous amplitudes $\alpha_k(t)$ and the fundamental instantaneous frequencies $N_k\phi_k'(t)$ can be estimated by the synchrosqueezed transform  \cite{1DSSWPT,ChuiLinWu2016} or the data-driven time-frequency analysis \cite{Hou2012}. With these instantaneous properties ready, the generalized shape function can be estimated by the diffeomorphism based spectral analysis (DSA) \cite{1DSSWPT}, the singular value decomposition (SVD) method \cite{Hou2016}, and the functional regression method \cite{ChuiLinWu2016} under different conditions as specified in these references. 
 This paper proposes a recursive diffeomorphism-based regression method (RDBR) as an alternative solution to the generalized mode decomposition problem. Before applying the RDBR method, the synchrosqueezed transform is applied to estimate instantaneous amplitudes $\{\alpha_k(t)\}_k$ and instantaneous frequencies $\{N_k\phi_k'(t)\}_k$. With these instantaneous properties ready, the RDBR method is able to estimate generalized shape functions $\{s_k(t)\}_k$ using a time-frequency unwarping technique (a diffeomorphism to be clarified later) and a nonparametric regression algorithm with theoretical guarantee. Numerical examples show that this novel method works in many different situations: it can identify and extract generalized modes with similar phase functions and a wide range of generalized shape functions; it can estimate generalized shape functions of a short signal with few periods, which potentially enables online computation for dynamic shapes changing in time.

\begin{figure}
  \begin{center}
    \begin{tabular}{cc}
      \includegraphics[height=1.0in]{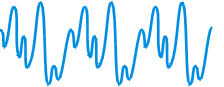} &
      \includegraphics[height=1.0in]{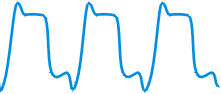}
    \end{tabular}
  \end{center}
  \caption{A few consecutive copies of generalized shape functions for the sound produced by playing a note on a musical instrument. Left: Violin. Right: Piano.}
  \label{fig:musicShape}
\end{figure}

The rest of this paper is organized as follows. In Section \ref{sec:SST}, the one-dimensional synchrosqueezed transform is briefly introduced with a simple example. In Section \ref{sec:RDBR}, the recursive diffeomorphism-based shape regression method is introduced and asymptotically analyzed.\footnote{Notations in the asymptotical analysis: we shall use the $O(\epsilon)$ notation, as well as the related notations  $\lesssim$ and $\gtrsim$; in particular, we write $F=O(\epsilon)G$ if there exists a constant $C$ (which we will not specify further) such that $|F|\leq C\epsilon |G|$; here $C$ may depend on some general parameters as detailed just before Theorem \ref{thm:main2}.} In Section \ref{sec:results}, some synthetic and real examples are provided to demonstrate the efficiency of the RDBR method. Finally, we conclude this paper in Section \ref{sec:conclusion}.

\section{Synchrosqueezed transform (SST)}
\label{sec:SST}

A powerful tool for the mode decomposition problem is the synchrosqueezed transform (SST). It consists of a linear time-frequency analysis tool and a nonlinear synchrosqueezing technique to obtain a sharpened time-frequency representation \cite{SSCT,SSSTFT,Daubechies2011,1DSSWPT,SSWPT,behera}. The SST is a reasonably robust algorithm \cite{Robustness,Daubechies20150193} with fast forward and inverse transforms based on the FFT. It has been applied to analyze oscillatory data in a wide range of real problems. Following \cite{1DSSWPT}, the one-dimensional synchrosqueezed wave packet transforms (SSWPT) is applied to estimate fundamental instantaneous properties before estimating generalized shapes. Hence, we will follow the notations in \cite{1DSSWPT} and briefly introduce the SSWPT with a concrete example 
\begin{equation}
\label{eqn:ex1}
f(t)=f_1(t)+f_2(t),
\end{equation}
where
\[
f_1(t) = \alpha_1(t)s_1(2\pi N_1\phi_1(t))= (1+0.05\sin(4\pi x))s_1\left(120\pi(x+0.01\sin(2\pi x))\right)
\] 

\[
f_2(t) = \alpha_2(t)s_2(2\pi N_2\phi_2(t))= (1+0.1\sin(2\pi x))s_2\left(180\pi(x+0.01\cos(2\pi x))\right),
\] 
$s_1(t)$ and $s_2(t)$ are generalized shape functions defined in $[0,1]$ as shown in Figure \ref{fig:2}. The SSWPT is applied to recover $\alpha_i(t)$, $i=1$, $2$, and $N_i\phi_i(t)$, $i=1$, $2$ from $f(t)$. For detailed implementation, the reader is referred to \cite{1DSSWPT}.

\begin{figure}
  \begin{center}
    \begin{tabular}{cccc}
      \includegraphics[height=1.2in]{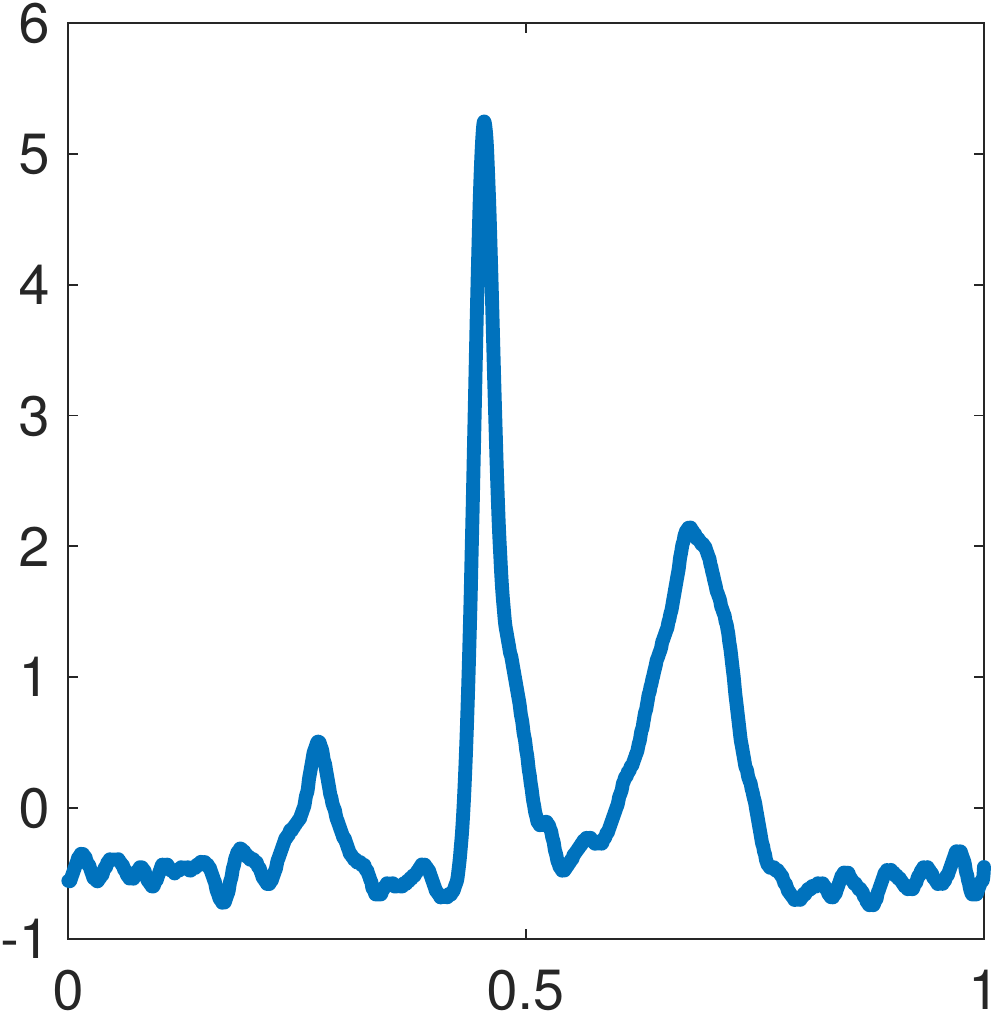} &
      \includegraphics[height=1.2in]{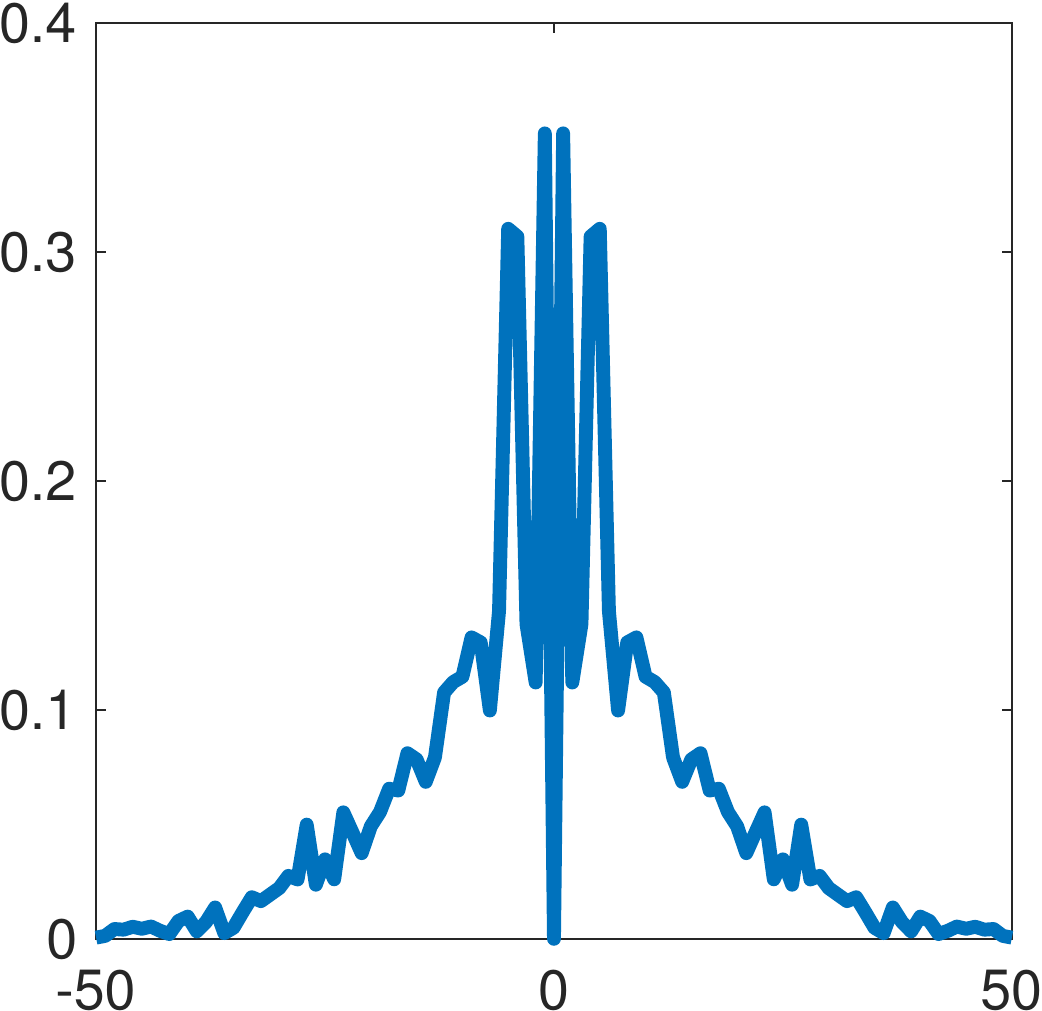} &
      \includegraphics[height=1.2in]{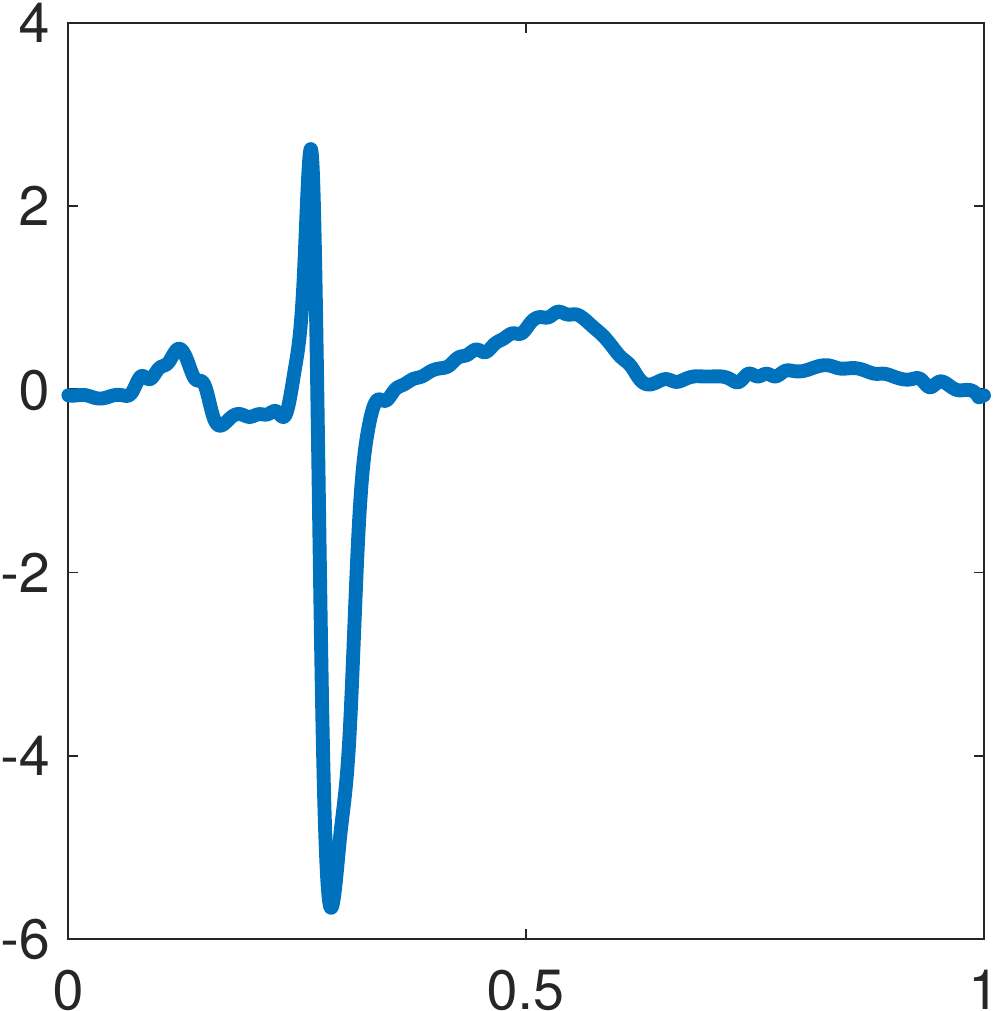} &
      \includegraphics[height=1.2in]{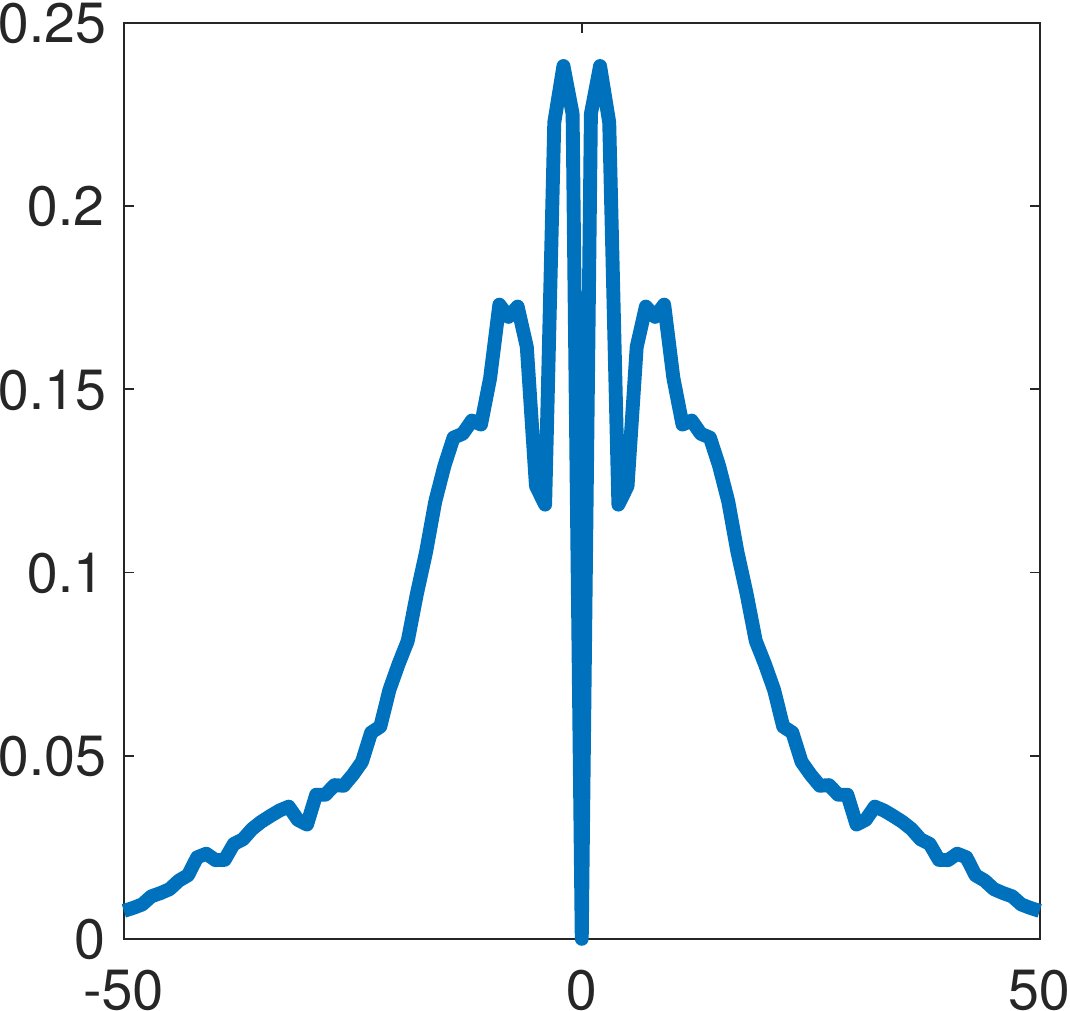}\\
      (a) & (b) & (c) & (d)
    \end{tabular}
  \end{center}
  \caption{(a) and (c): two generalized shape functions $s_1(2\pi t)$ and $s_2(2\pi t)$ of real ECG signals; (b) and (d): the Fourier spectrum of the generalized shape functions in (a) and (c), respectively.}
  \label{fig:2}
\end{figure}

Let $w(t)$ be a mother wave packet in the Schwartz class. The Fourier transform
$\widehat{w}(\xi)$ is assumed to be a real-valued, non-negative, smooth function
with a support equal to $(-d,d)$ with $d\leq 1$. Using $w(t)$, a family of wave packets can be constructed through scaling, modulation, and translation controlled by a geometric parameter $s$.

\begin{definition}
  \label{def:WA2d}
  Given the mother wave packet $w(t)$ and the parameter $s\in(1/2,1)$,
  the family of wave packets $\{w_{a b}(t): |a|\geq 1,b\in \R\}$
  is defined as
  \[
  w_{a b}(t)=|a|^{s/2} w(|a|^s(t-b)) e^{2\pi i (t-b)a},
  \]
  or equivalently, in the Fourier domain as
  \[
  \widehat{w_{ab}}(\xi) = |a|^{-s/2} e^{-2\pi i b\xi}
  \widehat{w}(|a|^{-s}(\xi-a)).
  \]
\end{definition}

If $s$ were equal to $1$ or $1/2$, these
functions would be qualitatively similar to the standard wavelets or the wave atoms \cite{Demanet2007}, respectively. Allowing one more degree of freedom $s$ makes wave packets more adaptive to the given data. 

\begin{definition}
  \label{def:WAT}
  The one-dimensional wave packet transform of a function $f(t)$ is a function
  \begin{align}
    W_f(a,b) 
    &= \langle w_{a b},f\rangle =  \int \overline{w_{a b}(t)}f(t)dt \label{eq:WAT} 
  \end{align}
  for $|a|\geq 1,b\in \R$.
\end{definition}

\begin{definition} 
  \label{def:IF}
  Instantaneous frequency information function:

  Let $f\in L^\infty(\R)$. The instantaneous frequency information function of $f$ is defined by
  \begin{equation}
    v_f(a,b)=
    \begin{cases}
      \frac{ \partial_b W_f(a,b) }{ 2\pi i W_f(a,b)},
      & \text{for }|W_f(a,b)|>0;\\
      \infty, 
      & otherwise,
    \end{cases}
    \label{E1single}
  \end{equation}
  where $\partial_b W_f(a,b)$ is the partial derivative of $W_f(a,b)$ with respect to $b$.
\end{definition}

As we shall see in Theorem \ref{thm:main2}, for a class of oscillatory functions $f(t) = \alpha(t) e^{2\pi i N\phi(t)}$, $v_f(a,b)\approx  N\phi'(b)$ independently of $a$ as long as $W_f(a,b)\neq 0$. If we reassign the coefficients $W_f(a,b)$ from the original position $(a,b)$ to a new position $(v_f(a,b),b)$, then we would obtain a sharpened time-frequency representation of $f(t)$. This motivates the definition of the synchrosqueezed energy distribution as follows.

\begin{definition}
  Given $f(t)$, $W_f(a,b)$, and $v_f(a,b)$, the synchrosqueezed energy
  distribution $T_f(v, b)$ is defined by
  \begin{equation}
    T_f(v,b) = \int_{\R} |W_f(a,b)|^2 \delta(\Re{v_f(a,b)}-v) da \label{eq:SED}
  \end{equation}
  for $v,b\in \R$, where $\Re$ means the real part of a complex number.
\end{definition} 

For a multi-component signal $f(t)=\sum_{k=1}^K \alpha_k(t) e^{2\pi i N_k \phi_k(t)}$, the synchrosqueezed energy of each component will concentrate around its corresponding instantaneous frequency $N_k\phi_k'(b)$, i.e. the supports of $T_f(v,b)$ are essentially narrow bands around the curves $(b,N_k\phi_k'(b))$ in the two-dimensional time-frequency domain, (see Figure \ref{fig:3} (left) for an example of $T_f(v,b)$). Hence, the SSWPT can provide information about their instantaneous frequencies. In the presence of generalized shape functions defined below, the spectral information becomes more complicated (see Figure \ref{fig:3} (right) for an example).

\begin{definition} Generalized shape functions:
\label{def:GSF}
The generalized shape function class ${\cal S}_M$ consists of $2\pi$-periodic functions $s(t)$ in the Wiener Algebra with a unit $L^2([-\pi,\pi])$-norm and a $L^\infty$-norm bounded by $M$ satisfying the following spectral conditions:
\begin{enumerate}
\item The Fourier series of $s(t)$ is uniformly convergent;
\item $\sum_{n=-\infty}^{\infty}|\widehat{s}(n)|\leq M$ and $\widehat{s}(0)=0$;
\item Let $\Lambda$ be the set of integers $\{|n|: \widehat{s}(n)\neq 0\}$. The greatest common divisor $\gcd(s)$ of all the elements in $\Lambda$ is $1$.
\end{enumerate}
\end{definition}

\begin{definition}
  \label{def:GIMTF}
  A function $f(t)=\alpha(t)s(2\pi N \phi(t))$ is a generalized intrinsic mode type
  function (GIMT) of type $(M,N)$, if $s(t)\in {\cal S}_M$ and $\alpha(t)$ and $\phi(t)$ satisfy the conditions below.
\begin{align*}
    \alpha(t)\in C^\infty, \quad |\alpha'|\leq M, \quad 1/M \leq \alpha\leq M \\
    \phi(t)\in C^\infty,  \quad  1/M \leq | \phi'|\leq M, \quad |\phi''|\leq M.
   \end{align*}
\end{definition}

\begin{definition}
  \label{def:GSWSIMC}
  A function $f(t)$ is a well-separated generalized superposition of type
  $(M,N,K,s)$, if
  \[
  f(t)=\sum_{k=1}^K f_k(t),
  \] 
  where each $f_k(t)=\alpha_k(t)s_k(2\pi N_k \phi_k(t))$ is a GIMT of type $(M,N_k)$ such that $N_k\geq N$ and the phase functions satisfy the
  separation condition: for any pair $(a,b)$, there exists at most one pair $(n,k)$ such that $\widehat{s_k}(n)\neq 0$ and that
\[
|a|^{-s}|a-nN_k\phi_k'(b)|< d.
\]
We denote by $GF(M,N,K,s)$ the set of all such functions.
\end{definition}

If $f(t)$ is a well-separated generalized superposition of type
  $(M,N,K,s)$, the synchrosqueezed energy distribution $T_f(v,b)$ has well-separated supports, each of which concentrates around one instantaneous frequency $nN_k\phi_k'(b)$. Fortunately, although $f(t)$ is not a well-separated generalized superposition, its Fourier expansion components $\widehat{s}_k(n)\alpha_k(t)e^{2\pi i nN_k\phi_k(t)}$ might still be well-separated in the low-frequency domain (see Figure \ref{fig:3} (left) for an example). Hence, some instantaneous frequency $nN_k\phi_k'(b)$ can be estimated from the ridge (or average) of its corresponding support, and its corresponding component $\widehat{s_k}(n)\alpha_k(t) e^{2\pi i n N_k \phi_k(t)}$ can be recovered by an inverse SST restricted to the corresponding support. Since in practice, and as illustrated by Figure \ref{fig:3}, even if low-frequency components are well-separated, high-frequency components might still be mixed up, well-separated generalized superposition is thus very rare; this motivates the definition of a more reasonable situation below.

\begin{definition}
 A function $f(t)$ is a weak well-separated generalized superposition of type
  $(M,N,K,s)$ if
  \[
  f(t)=\sum_{k=1}^K f_k(t)
  \] 
  where each $f_k(t)=\alpha_k(t)s_k(2\pi N_k \phi_k(t))$ is a GIMT of type $(M,N_k)$ such that $N_k\geq N$ and the phase functions satisfy the following weak well-separation conditions. 
\begin{enumerate}
\item Suppose 
\[
Z_{nk}=\left\{(a,b):|a-nN_k\phi_k'(b)|\leq d|a|^{s}\right\}.
\]
For each $k$, there exists $n_k$ such that $\widehat{s_k}(n_k)\neq 0$ and $Z_{n_kk}\cap Z_{nj}=\emptyset$ for all pairs $(n,j)\neq (n_k,k)$ and $\widehat{s_j}(n)\neq 0$.
\item $\exists K_0<\infty$ such that $\forall a\in\R$ and $\forall b\in \R$ there exists at most $K_0$ pairs of $(n,k)$ such that $(a,b)\in Z_{nk}$.
\end{enumerate}
We denote by $wGF(M,N,K_0,K,s)$ the set of all such functions.
\end{definition}

The weak well-separation condition essentially requires that each generalized mode has at
least one Fourier expansion component $\widehat{s}_k(n)\alpha_k(t)e^{2\pi i nN_k\phi_k(t)}$ well-separated in the time-frequency domain. This enables the
SSWPT to estimate the fundamental instantaneous amplitude and frequency of each generalized mode. The following theorem proved in \cite{1DSSWPT} supports this intuition in more detail. Recall that, when we write $O(\cdot)$, $\lesssim$, or $\gtrsim$, the implicit constants may depend on $M$, $K$, $K_0$, and no other parameters.

\begin{theorem}
  \label{thm:main2}
  For a function $f(t)$ and $\eps>0$, we define
  \[
  R_{\eps} = \{(a,b): |W_f(a,b)|\geq |a|^{-s/2}\sqrt \eps\}
  \]
  and 
  \[
  Z_{n,k} = \{(a,b): |a-nN_k\phi_k'(b)|\leq d|a|^s \}
  \]
  for $1\le k\le K$ and $|n|\geq 1$. For fixed $M$, $K_0$, $K$ and $\forall \eps>0$, there
  exists a constant $N_0(M,K_0,K,s,\eps)>0$ such that
  $\forall N>N_0$ and $f(t)\in wGF(M,N,K_0,K,s)$ the following statements
  hold.
  \begin{enumerate}[(i)]
  \item For each $j$, there exists $n_j$ such that $\widehat{s_j}(n_j)\neq 0$ and $Z_{n_jj}\cap Z_{nk}=\emptyset$ for all pairs $(n,k)\neq (n_j,j)$ and $\widehat{s_k}(n)\neq 0$;
  \item For any $(a,b) \in R_{\eps} \cap Z_{n_j,j}$, 
    \[
    \frac{|v_f(a,b)-n_jN_j\phi_j'(b)|}{ |n_jN_j \phi_j'(b)|}\lesssim\sqrt \eps.
    \]
\item For each $j$, let 
\[
l_{n_j}(b)=\min \left\{a:(a,b)\in R_\epsilon\cap Z_{n_jj}\right\},\quad u_{n_j}(b)=\max\left\{a:(a,b)\in R_\epsilon\cup Z_{n_jj}\right\}.
\]
Suppose $v_f(a,b)\neq \infty$. If $a\leq l_{n_j}(b)$, then $v_f(a,b)\leq l_{n_j}(b)(1+O(\sqrt{\epsilon}))$. If $a\geq u_{n_j}(b)$, then $v_f(a,b)\geq u_{n_j}(b)(1-O(\sqrt{\epsilon}))$.
  \end{enumerate}
\end{theorem}

\begin{figure}
  \begin{center}
    \begin{tabular}{cc}
      \includegraphics[height=2.4in]{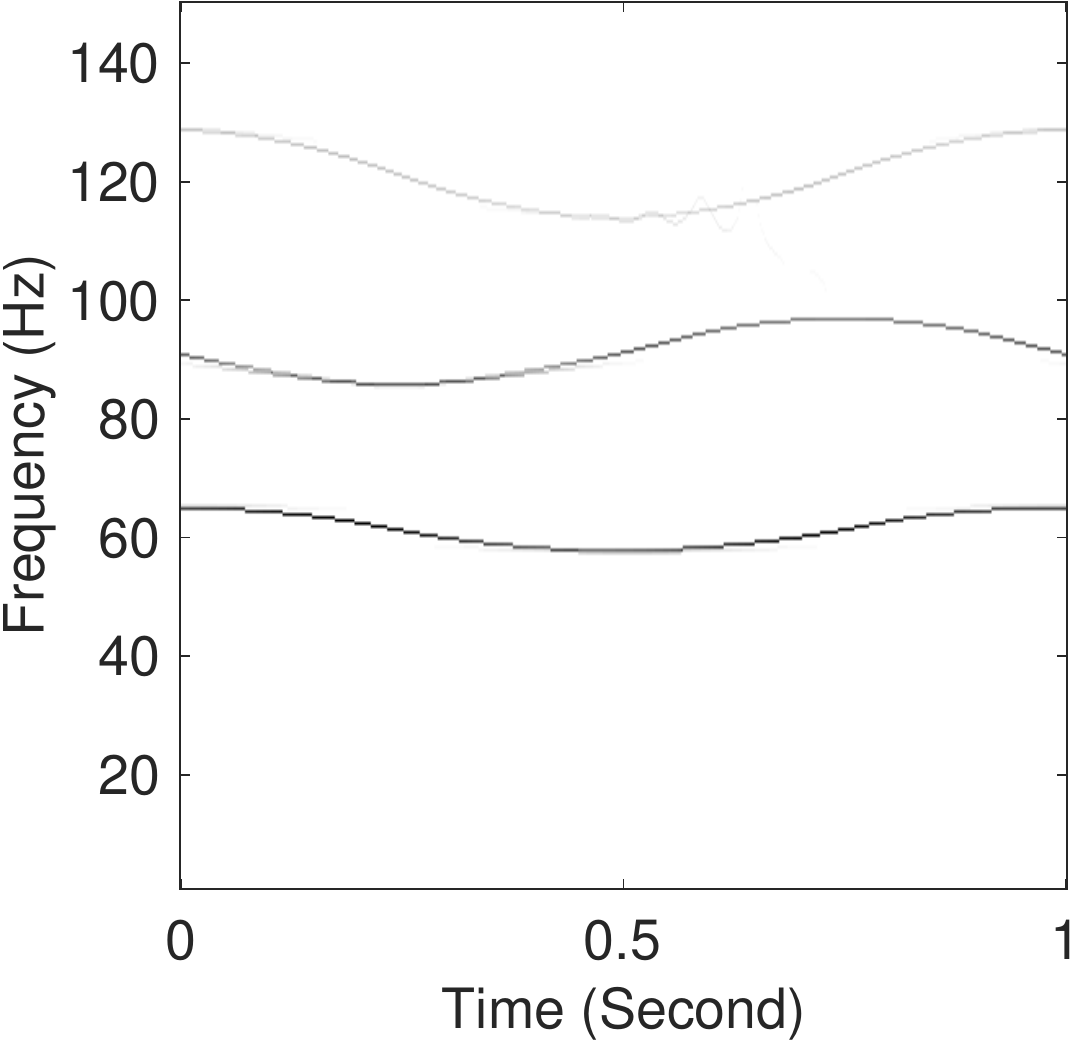} &
      \includegraphics[height=2.4in]{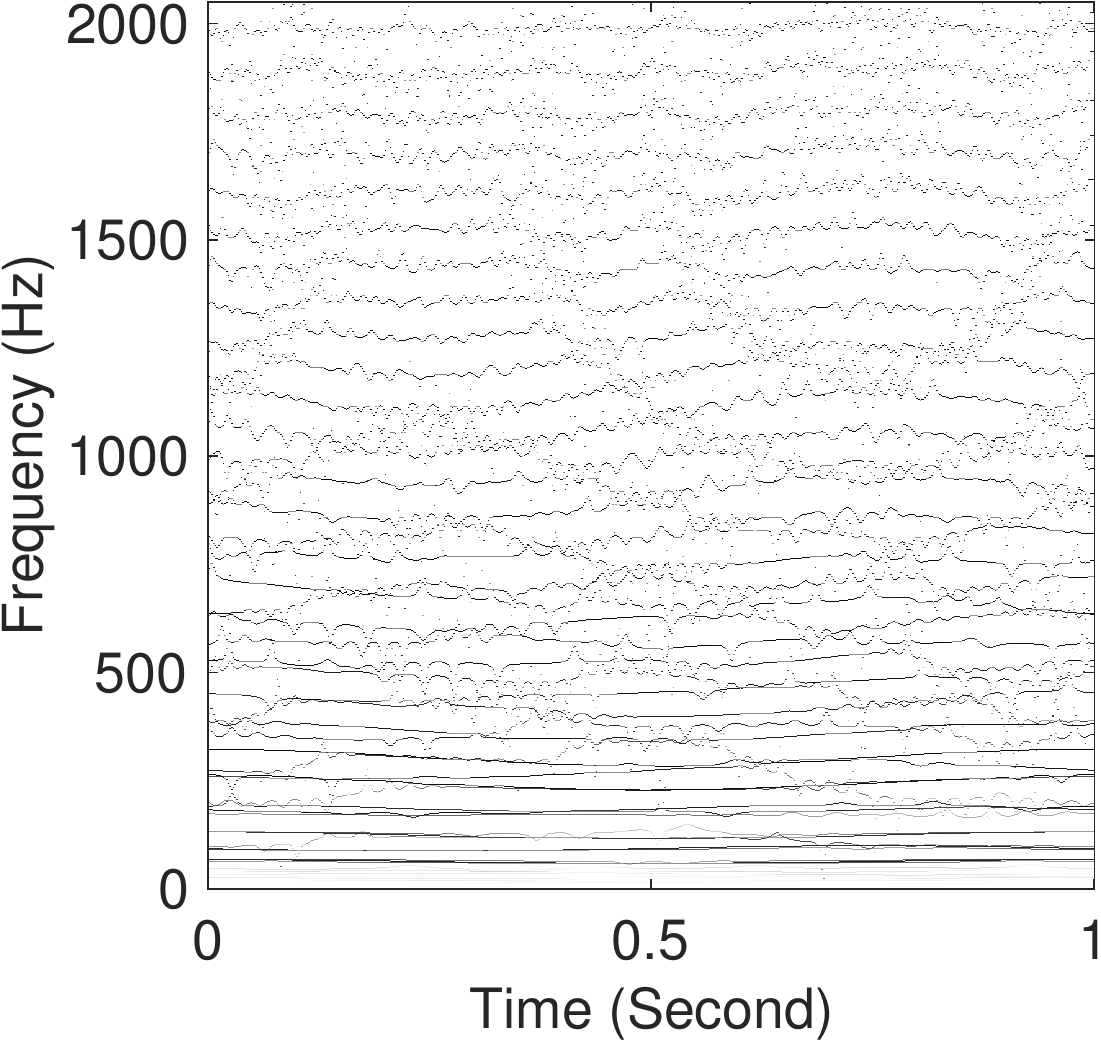} 
    \end{tabular}
  \end{center}
  \caption{Both figures show the synchrosqueezed energy distribution of the same signal in Equation \ref{eqn:ex1} . Left: in the low-frequency domain, the synchrosqueezed energy distribution shows a few well-separated oscillatory components. The instantaneous frequencies of these components can be directly read off from this distribution. Right: the synchrosqueezed energy distribution in the whole time-frequency domain. Instantaneous frequencies of the Fourier expansion terms in $\{ \widehat{s_k}(n)\alpha_k(t) e^{2\pi i n N_k \phi_k(t)}\}_{k,n}$  cross over together.}
  \label{fig:3}
\end{figure}

Theorem \ref{thm:main2} shows that the supports of the synchrosqueezed energy distribution $T_f(v,b)$ are essentially narrow bands around the curves $(b,n_kN_k\phi_k'(b))$ in the two-dimensional time-frequency domain, (see Figure \ref{fig:3} (left) for an example of $T_f(v,b)$). Hence, we can estimate the instantaneous frequencies by tracking these curves. These curves naturally belong to $K$ groups, each of which corresponds to the multiple of a fundamental instantaneous frequency $N_k\phi'_k(t)$. Following the curve classification idea in Algorithm 3.7 and Theorem 3.9 in \cite{1DSSWPT}, we are able to classify these curves and extract the fundamental instantaneous frequencies $\{N_k\phi'_k(t)\}_{1\leq k\leq K}$, which give the fundamental instantaneous phases $\{N_k\phi_k(t)\}_{1\leq k\leq K}$. By applying the inverse synchrosqueezed transform to the support of $T_f(v,b)$ corresponding to the instantaneous frequency $n_k N_k\phi'_k(t)$, we can reconstruct $\widehat{s}_k(n_k)\alpha_k(t)e^{2\pi i n_k N_k\phi_k(t)}$, the magnitude of which gives the estimation of the fundamental instantaneous amplitudes $\alpha_k(t)$ up to a constant prefactor. We refer the reader to \cite{1DSSWPT} for more detail and assume that these estimates are available afterward. There are some other alternative methods available to estimate fundamental instantaneous properties, e.g. a recent paper \cite{ceptrum} based on the short-time ceptrum transform.

\section{Recursive diffeomorphism-based shape regression (RDBR)}
\label{sec:RDBR}

\subsection{Algorithm description}
\label{sub:aglo}

In this section, we introduce the recursive diffeomorphism-based regression (RDBR), which we shall use to estimate generalized shape functions $s_k(t)$ in a superposition of the form
\begin{equation*}
f(t)= \sum_{k=1}^K \alpha_k(t) s_k(2 \pi N_k \phi_k(t)),
\end{equation*}
assuming that the fundamental instantaneous phases $\{N_k\phi_k(t)\}_{k=1}^K$ and the instantaneous amplitudes $\{\alpha_k(t)\}_{k=1}^K$ are known a priori. Suppose the signal $f(t)$ is sampled randomly in the domain $[0,1]$. In particular, we have $L$ points of measurement $\{f(t_\ell)\}_{\ell=1,\dots,L}$ with $L$ independent and identically distributed (i.i.d.) grid points $\{t_\ell\}_{\ell=1,\dots,L}$ with a uniform distribution in $[0,1]$. Usually, the grid is uniform in $[0,1)$, which is a special case in our assumption. Numerical examples are provided in Section \ref{sub:iid} to support this assumption.

Notice that the smooth function $p_k(t)=N_k\phi_k(t)$ has the interpretation of a warping in each generalized mode via a diffeomorphism $p_k:\R\to\R$. Hence, we can define the inverse-warping data by
\begin{eqnarray*}
h_k(v)& =& \frac{f\circ p_k^{-1}(v)}{\alpha_k\circ p_k^{-1}(v)}\\
&=&s_k(2\pi v)+ \sum_{j\neq k} \frac{\alpha_j\circ p_k^{-1}(v)}{\alpha_k\circ p_k^{-1}(v)} s_j(2\pi   p_j\circ p_k^{-1}(v))\\
&\defeq & s_k(2\pi v)+ \kappa_k(v),
\end{eqnarray*}
where $v = p_k(t)$. Correspondingly, we have a set of measurements of $h_k(v)$ sampled on  $\{h_k(v_\ell)\}_{\ell=1,\dots,L}$ with $v_\ell=p_k(t_\ell)$.

The observation that $s_k(2\pi v)$ is a periodic function with a period $1$ motivates the following folding map $\tau$ that folds the two-dimensional point set $\{(v_\ell,h_k(v_\ell))\}_{\ell=1,\dots,L}$  together
\begin{eqnarray*}
\tau: \ \      \left(v_\ell, h_k(v_\ell) \right)   \mapsto    \left(\text{mod}(v_\ell,1),  h_k(v_\ell) \right).
\end{eqnarray*}
If there was only one generalized mode, then the point set $\{\tau(v_\ell,s_k(2\pi v_\ell))\}_{\ell=1,\dots,L}\subset \mathbb{R}^2$ is a two-dimensional point set located at the curve $(v,s_k(2\pi v))\subset \mathbb{R}^2$ given by the generalized shape function $s_k(2\pi v)$ with $v\in [0,1)$. Figure \ref{fig:4} (left) visualizes one example of this point set in the case of one mode. This could also be understood in the following way. Let $X_k$ be an independent random variable in $[0,1)$ and $Y_k$ be the response random variable in $\mathbb{R}$. Consider $(x_\ell,y_\ell)=\tau(v_\ell,s_k(2\pi v_\ell))$ for $\ell =1,\dots,L$ as $L$ i.i.d. samples of the random vector $(X_k,Y_k)$. Then we know $s_k$ is the regression function satisfying $Y_k=s_k(2\pi X_k)$. Hence, the solution of the following regression problem gives the generalized shape function
\begin{equation}
\label{eqn:re}
s_k = s^R_k\defeq \underset{s:\mathbb{R}\rightarrow \mathbb{R}}{\arg\min}\quad \E\{\left| s(2\pi X_k)-Y_k\right|^2\},
\end{equation}
where the superscript $^R$ means the ground truth regression function. 
Let $\bar{s}_k$ denote the numerical solution of the above regression problem, then the approximation $\bar{s}_k\approx s_k$ is precise once $L$ is sufficiently large, since the variance $\sigma^2\defeq\Var\{Y_k|X_k=x\}=0$ for all $x\in[0,1)$. Once the generalized shape function $s_k$ has been estimated, the estimated generalized mode $\alpha_k(t)\bar{s}_k(2\pi N_k\phi_k(t))$ is an immediate result.

The assumption that $(x_\ell,y_\ell)=\tau(v_\ell,s_k(2\pi v_\ell))$ for $\ell =1,\dots,L$ are $L$ i.i.d. samples of the random vector $(X_k,Y_k)$ comes from the fact that grid points $\{t_\ell\}_{\ell=1,\dots,L}$ are i.i.d.. In practice, these grid points are determinate and usually uniform. Notice that the warping and folding maps behave essentially like pseudorandom number generators with a nonlinear congruential function $mod(p_k(t),1)$ for $k=1,\dots,K$. Hence, even if in the case of determinate grid points, the point set $(x_\ell,y_\ell)=\tau(v_\ell,s_k(2\pi v_\ell))$ for $\ell =1,\dots,L$ has similar statistical properties like a set of i.i.d. samples. It might be interesting to analyze this assumption further but we would only assume it in this paper and focus on the regression problem. This assumption will be validated by numerical examples in Section \ref{sub:iid}.

\begin{figure}
  \begin{center}
    \begin{tabular}{ccc}
      \includegraphics[height=1.8in]{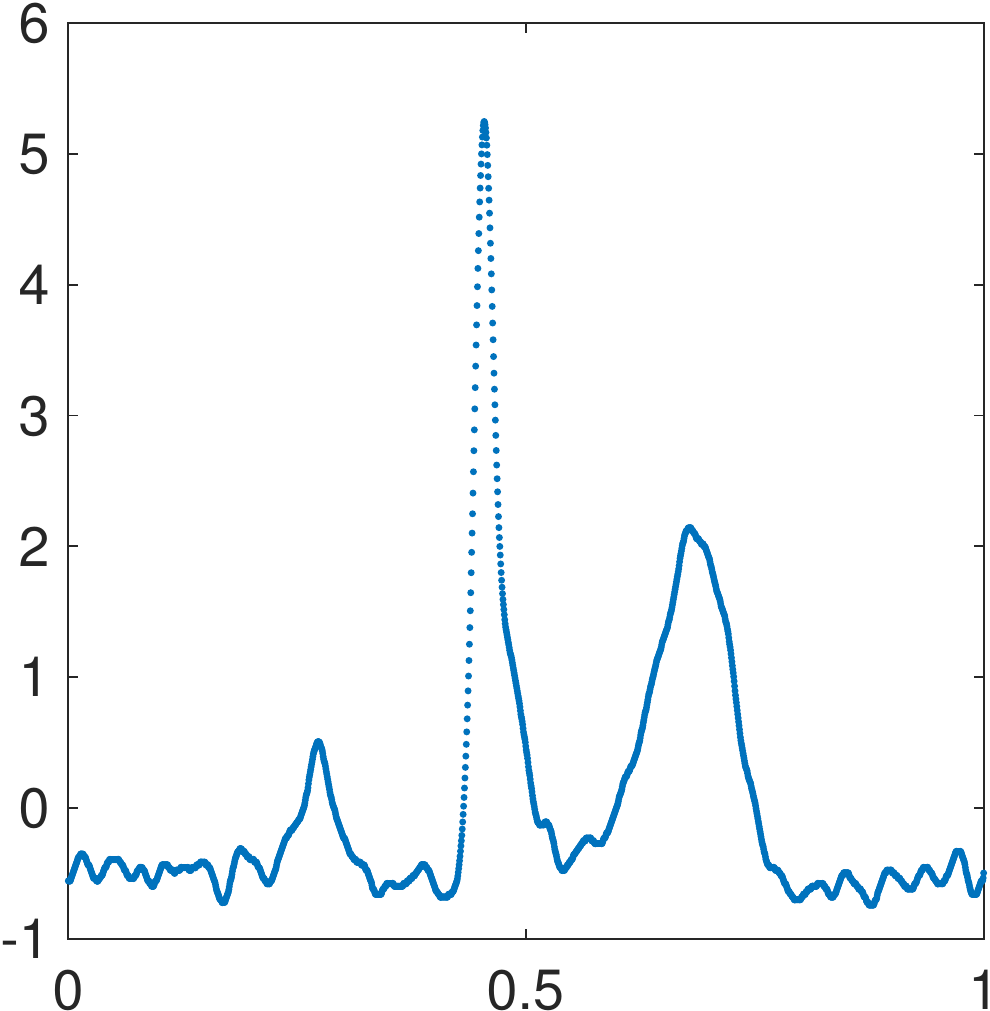} &
      \includegraphics[height=1.8in]{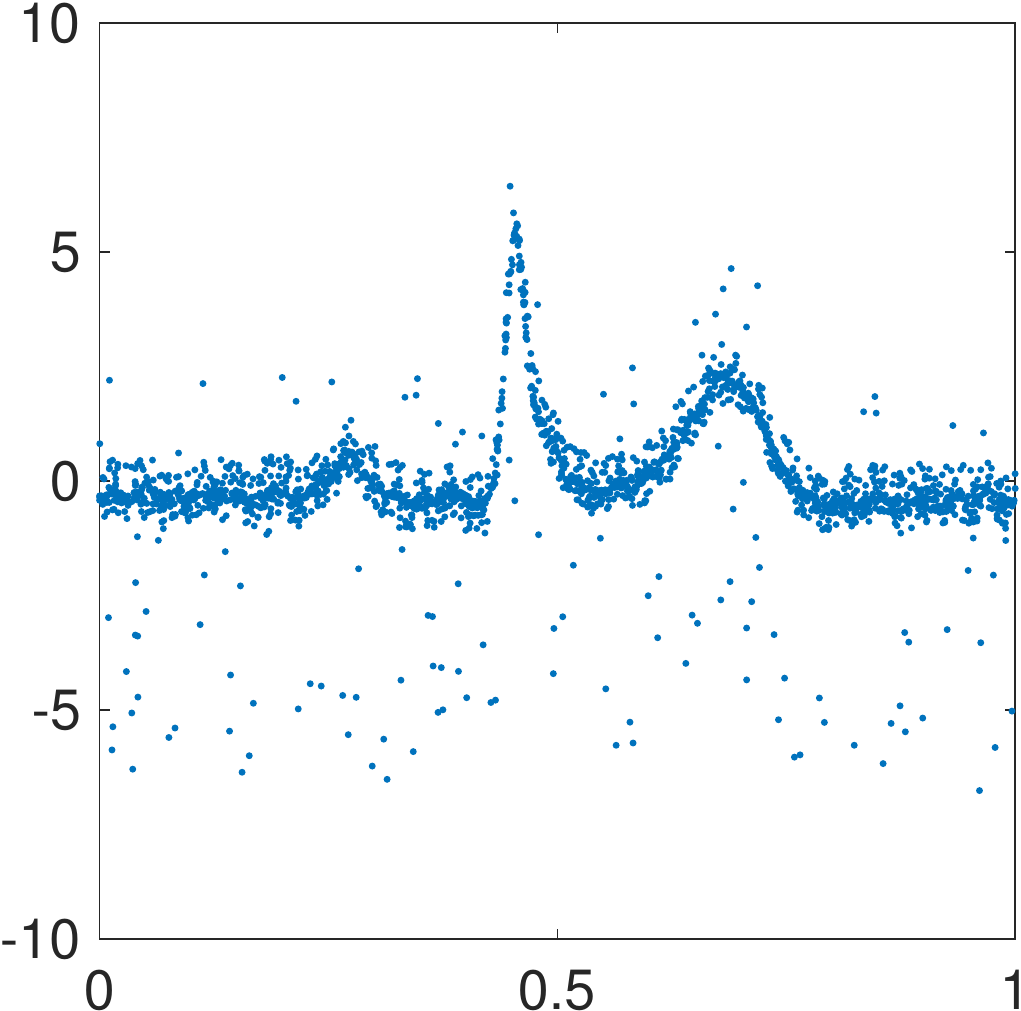} &
      \includegraphics[height=1.8in]{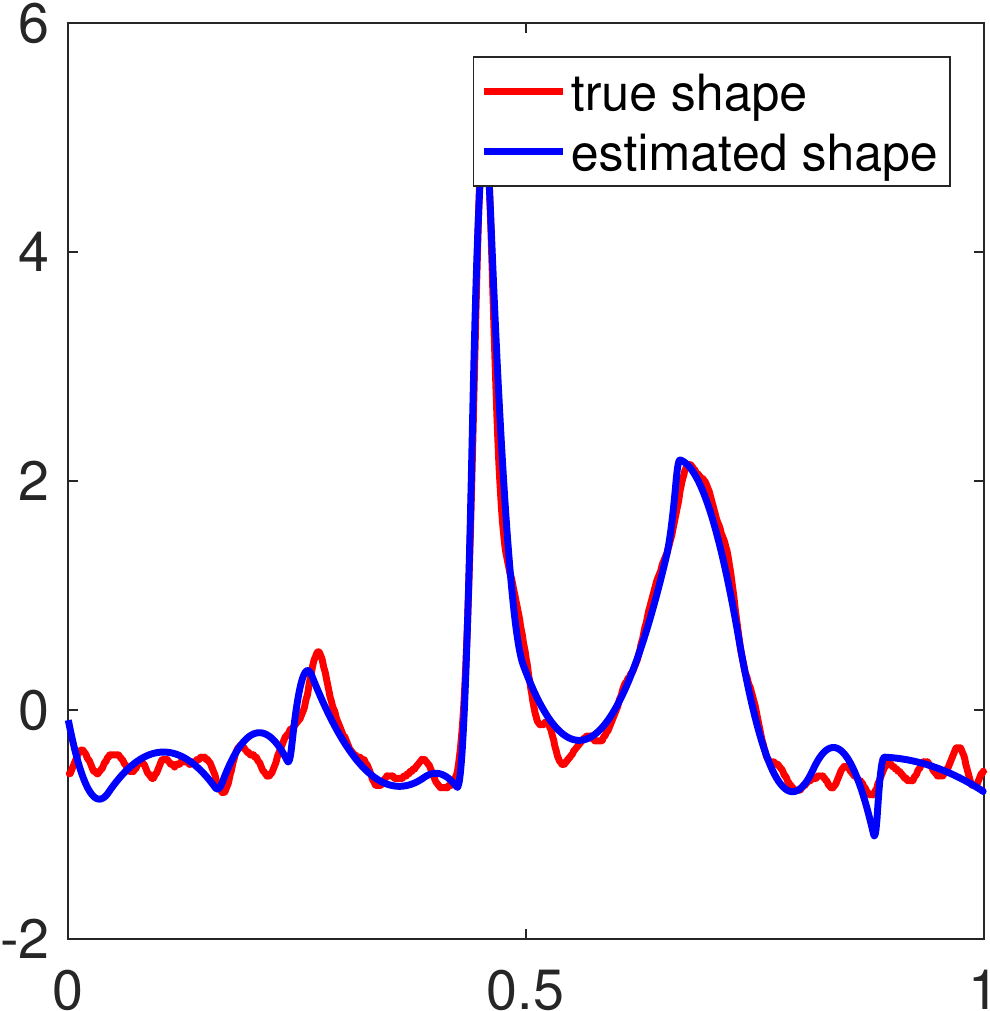} 
    \end{tabular}
  \end{center}
  \caption{Left: the point set $\{\tau(v_\ell,s_k(2\pi v_\ell))\}_{\ell=1,\dots,L}\subset \mathbb{R}^2$ in the case of one mode $f(t)=f_1(t)$, where $f_1(t)$ is given in \eqref{eqn:ex1}. Middle:  the point set $\{\tau(v_\ell,s_k(2\pi v_\ell))\}_{\ell=1,\dots,L}\subset \mathbb{R}^2$ in the case of two modes $f(t)=f_1(t)+f_2(t)$ defined in \eqref{eqn:ex1}. Right: the true shape function $s_1(2\pi t)$ and the estimated one $\bar{s}_1(2\pi t)$ by regression once from the point set in the middle figure.}
  \label{fig:4}
\end{figure}

 However, in the presence of multiple modes, the trace $\tau(v_\ell,\kappa_k(v_\ell))$ of $\kappa_k(v)$ in the $\ell$th component stets like a noise perturbation in the regression problem (see Figure \ref{fig:4} middle for an illustration). In this case, $(x_\ell,y_\ell)=\tau(v_\ell,s_k(2\pi v_\ell)+\kappa_k(v_\ell))$ for $\ell =1,\dots,L$ can be considered as $L$ i.i.d. samples of a random vector $(X_k,Y_k)$ with a noise perturbation in $Y_k$. Hence, the regression in \eqref{eqn:re} only results in a rough estimate $\bar{s}_k$ as shown in Figure \ref{fig:4} right. To be more precise, let $s^E_k(2\pi x)\defeq \E\{Y_k-s_k(2\pi X_k)|X_k=x\}$ be the residual shape function, then $\bar{s}_k\approx s^R_k\defeq s_k+s^E_k$. Hence, the residual error of the mode decomposition 
 \[
 r(t)=f(t)- \sum_{k=1}^K \alpha_k(t) \bar{s}_k(2 \pi N_k \phi_k(t))\approx-\sum_{k=1}^K \alpha_k(t) s^E_k(2 \pi N_k \phi_k(t))
 \]
  might be large, if the residual shape functions $\{s^E_k\}_k$ are not zero. This motivates the design of a recursive scheme that repeats the same decomposition procedure to decompose the residual $r(t)$ until the residual is small as follows.

\begin{algorithm2e}[H]
\label{alg:pcg}
\caption{Recursive diffeomorphism-based regression (RDBR).}
Input: $L$ points of i.i.d. measurement $\{f(t_\ell)\}_{\ell=1,\dots,L}$ with $t_\ell \in [0,1]$, instantaneous phases $\{p_k\}_{k=1,\dots,K}$, instantaneous amplitudes $\{\alpha_k\}_{k=1,\dots,K}$, an accuracy parameter $\epsilon<1$, and the maximum iteration number $J$.

Output: generalized shape function estimates  $\{\tilde{s}_k\}_{k=1,\dots,K}$.

Initialize: let $r^{(0)}=f$, $\epsilon_1=\epsilon_2=1$, $\epsilon_0=2$, the iteration number $j=0$, $\bar{s}^{(0)}_k=0$,  and $\tilde{s}_k=0$ for all $k=1,\dots,K$.

\While{$j<J$, $\epsilon_1> \epsilon$, $\epsilon_2>\epsilon$, and $|\epsilon_1-\epsilon_0|>\epsilon$}{ For all $k$, $1\leq k\leq K$, define \[h^{(j)}_k= \frac{r^{(j)}\circ p_k^{-1}}{\alpha_k\circ p_k^{-1}},\]  and we know it is sampled on grid points $v_\ell=p_k(t_\ell)$;

Observe that $\left\{\tau(v_\ell,h^{(j)}_k(v_\ell))\right\}_{\ell=1,\dots,L}$ behaves like a sequence of i.i.d. samples of a certain random vector $(X_k,Y^{(j)}_k)$ with $X_k\in [0,1)$;
  
For all $k$, $1\leq k\leq K$, solve the distribution-free regression problem
\begin{equation}
\label{eqn:reg}
\bar{s}^{(j+1)}_k \approx s^{R,(j+1)}_k=\underset{s:\mathbb{R}\rightarrow \mathbb{R}}{\arg\min}\quad \E\{\left| s(2\pi X_k)-Y^{(j)}_k\right|^2\},
\end{equation}
where $\bar{s}^{(j+1)}_k$ denotes the numerical solution approximating the ground truth solution $s^{R,(j+1)}_k$;

Update $\bar{s}_k^{(j+1)}=\bar{s}_k^{(j+1)}-\frac{1}{2\pi}\int_0^{2\pi} \bar{s}_k^{(j+1)}(t)dt$ for all $k$;

Update $\tilde{s}_k=\tilde{s}_k+\bar{s}_k^{(j+1)}$ for all $k$;

Update $r^{(j+1)}=r^{(j)}-\sum_{k=1}^K \alpha_k(t) \bar{s}^{(j+1)}_k(2 \pi p_k(t))$;

Update $\epsilon_0=\epsilon_1$, $\epsilon_1=\|r^{(j+1)}\|_{L^2}$, $\epsilon_2=\max_{k}\{\|\bar{s}_k^{(j+1)}\|_{L^2}\}$, 

Set $j = j + 1$. 
}
\end{algorithm2e}

Remark that Line $8$ in Algorithm \ref{alg:pcg} is essential because it is a crucial condition for the convergence of the recursive scheme as we shall see in Lemma \ref{lem:SP}. It ensures that all generalized shape functions in each iteration has a zero mean. In practice, it is sufficient to estimate instantaneous amplitude functions $\{\alpha_k(t)\}$ up to an unknown prefactor, (e.g., $\{\widehat{s}_k(1)\alpha_k(t)\}$ can be estimated by the synchrosqueezed transform but $\{\widehat{s}_k(1)\}$ are not available), because the unknown prefactor has been absorbed in the  shape function estimation $\tilde{s}_k(t)$, which will approximate $s_k(t)/\widehat{s}_k(1)$.  Hence, when we reconstruct the $k$th mode, the unknown prefactor is cancelled out as follows
\[
\widehat{s}_k(1)\alpha_k(t)\frac{1}{\widehat{s}_k(1)}s_k(2\pi p_k(t))=\alpha_k(t)s_k(2\pi p_k(t)).
\]
Similarly, it is sufficient that a phase function $p_k(t)$ is available up to a constant, e.g., $p_k(t)-p_k(0)$ can be estimated by the synchrosqueezed transform but $p_k(0)$ is unknown. A shift in the estimation of a phase function leads to a shift in the estimation of its corresponding shape function. In the end, the shift is cancelled out when a mode is reconstructed.

Although Algorithm \ref{alg:pcg} relies on the exact instantaneous amplitudes and phases, numerical examples in Section \ref{sec:results} shows that the algorithm is not sensitive to the input amplitudes and phases as long as the folding procedure in Line $6$ is able to fold the periodic part in $h^{(j)}_k$ together.

\subsection{Convergence analysis}
\label{sub:conv}

In this section, an asymptotic analysis on the convergence of the recursive diffeomorphism-based regression (RDBR) method is provided. The partition-based regression method (or partitioning estimate) in Chapter 4 of \cite{regressionBook} will be used. The first theorem concerns about the convergence rate of the estimated regression function to the ground truth regression function as the number of samples tends to infinity. The second theorem clarifies the conditions that guarantee the convergence of the RDBR. The last theorem shows that the RDBR method is robust against noise. Before presenting these theorems, some notations and definitions are introduced below. 

Given a small step side $h\ll 1$, the time domain $[0,1]$ is uniformly partitioned into $N^h = \frac{1}{h}$ (assumed to be an integer) parts $\{[t^h_n,t^h_{n+1})\}_{n=0,\dots,N^h-1}$, where $t^h_n=nh$. The partition-based regression method with this uniform partition is applied to analyze the RDBR. Suppose $(x_\ell,y_\ell)_{\ell=1,\dots,L}$ are $L$ i.i.d. samples of a random vector $(X,Y)$ with a ground truth regression function denoted as $s^R$. Let $s^P_L$ denote the estimated regression function by the partition-based regression method with $L$ samples. Following the definition in Chapter 4 of \cite{regressionBook}, we have
\[
s^R(x)\approx s^P_L(x) \defeq\frac{\sum_{\ell=1}^L\mathcal{X}_{[t^h_n,t^h_{n+1})} (x_\ell) y_\ell }{\sum_{\ell=1}^L\mathcal{X}_{[t^h_n,t^h_{n+1})}(x_\ell) },
\]
when $x\in [t^h_n,t^h_{n+1})$, where $\mathcal{X}_{[t^h_n,t^h_{n+1})}(x)$ is an indicator function of $[t^h_n,t^h_{n+1})$. The following theorem given in Chapter 4 in \cite{regressionBook} estimates the $L_2$ risk of the approximation $s^P_L\approx s^R$ as follows.

\begin{theorem}
\label{thm:reg}
For the uniform partition with a step side $h$ in $[0,1)$ as defined just above, assume that 
\[
\Var(Y|X=x)\leq \sigma^2,\quad x\in\mathbb{R},
\]
\[
|s^R(x)-s^R(z)|\leq C|x-z|,\quad x,z\in\mathbb{R},
\]
$X$ has a compact support $[0,1)$, and there are $L$ i.i.d. samples of $(X,Y)$. Then the partition-based regression method provides an estimated regression function $s^P_L$ to approximate the ground truth regression function $s^R$, where
\begin{equation*}
s^R = \underset{s:\mathbb{R}\rightarrow \mathbb{R}}{\arg\min}\quad \E\{\left| s(2\pi X)-Y\right|^2\},
\end{equation*}
with an $L^2$ risk bounded by
\[
\E\|s^P_L-s^R\|^2\leq c_0\frac{\sigma^2+\|s^R\|^2_{L^\infty}}{Lh}+C^2h^2,
\]
where $c_0$ is a constant independent of the number of samples $L$, the regression function $s$,  the step side $h$, and the Lipschitz continuity constant $C$.
\end{theorem}

Other regression methods would also be suitable for the analysis of the RDBR and might lead to better convergence rate than the one in Theorem \ref{thm:reg}. For the sake of simplicity, we only focus on the analysis based on Theorem \ref{thm:reg}. To simplify notations, $s^P$ will be used instead of $s^P_L$ and $s\in \LL$ means that $s$ is a Lipschitz continuous function with a constant $C$ later on. 

Denote the set of sampling grid points $\{t_\ell\}_{\ell=1,\dots,L}$ in Algorithm \ref{alg:pcg} as $\T$. To estimate the regression function using the partition-based regression method, $\T$ is divided into several subsets as follows. For $i,j=1,\dots,K$, $i\neq j$, $m,n=0,\dots,N^h-1$, let 
\[
\T^{ij}_h(m,n)=\left\{ t\in \T: \mod(p_i(t),1)\in[t^h_m,t^h_m+h),\mod(p_j(t),1)\in[t^h_n,t^h_n+h)\right\},
\]
and
\[
\T^{i}_h(m)=\left\{ t\in \T: \mod(p_i(t),1)\in[t^h_m,t^h_m+h)\right\},
\]
then $\T=\cup_{m=0}^{N^h-1} \T^{i}_h(m)=\cup_{m=0}^{N^h-1}\cup_{n=0}^{N^h-1} \T^{ij}_h(m,n)$. 
Let 
\begin{equation}
\label{eqn:D}
D^{ij}_h(m,n)\quad \text{ and }\quad D^{i}_h(m)
\end{equation}
 denote the number of points in  $\T^{ij}_h(m,n)$ and $\T^{i}_h(m)$, respectively. 

\begin{definition}
  \label{def:wd}
  Suppose $f_k(t)=\alpha_k(t)s_k(2\pi N_k \phi_k(t))$ is a GIMT of type $(M,N_k)$ for $k=1,\dots,K$. Then the collection of phase functions $\{p_k(t)\}_{1\leq k\leq K}$ is said to be $(M,N,K,h,\beta,\gamma)$-well-differentiated and denoted as $\{p_k(t)\}_{1\leq k\leq K}\in\WD(M,N,K,h,\beta,\gamma)$, if the following conditions are satisfied:
\begin{enumerate}
\item $N_k\geq N$ for $k=1,\dots,K$;
\item $  \gamma\defeq  \underset{m,n,i\neq j }{\min}D^{ij}_h(m,n)$ satisfies $\gamma>0$ 
, where $D^{ij}_h(m,n)$ (and $D^{i}_h(m)$  below) is defined in \eqref{eqn:D};
\item Let \[\beta_{i,j} \defeq \left(\sum_{m=0}^{N^h-1} \frac{1}{D^{i}_h(m) }\left( \sum_{n=0}^{N^h-1} ( D^{ij}_h(m,n)-\gamma )^2 \right)\right)^{1/2}\] for all $i\neq j$, then $\beta\defeq \max\{\beta_{i,j}:i\neq j\}$ satisfies $M^2(K-1)\beta<1$.
\end{enumerate}
\end{definition}


In the above definition, $\gamma$ quantifies the dissimilarity between phase functions. The larger $\gamma$ is, the more dissimilarity phase functions have. If two phases are very similar, there might be some nearly empty sets $\T^{ij}_h(m,n)$ and hence $\gamma$ is small.  If $\gamma$ is larger, the numbers $\{D^{ij}_h(m,n)\}_{m,n}$ are closer and $\beta$ would be smaller. To guarantee a large $\gamma$, $N$ and $L$ should be sufficiently large. With these notations defined, it is ready to present the main analysis of the RDBR. 

Let's recall that in each iteration of Algorithm \ref{alg:pcg}, if we denote the target shape function as $s_k^{(j)} $ then the given data is 
\begin{equation}
\label{eqn:r}
r^{(j)}=\sum_{k=1}^K \alpha_k(t) s^{(j)}_k(2 \pi p_k(t)).
\end{equation}
In the regression problem 
\begin{eqnarray}
s_k^{R,(j+1)} &=& \underset{s:\mathbb{R}\rightarrow \mathbb{R}}{\arg\min}\quad \E\{\left| s(2\pi X_k)-Y_k^{(j)} \right|^2\}\\
&=&  \underset{s:\mathbb{R}\rightarrow \mathbb{R}}{\arg\min}\quad \E\{\left| s(2\pi X_k)-(Y_k^{(j)} -s_k^{(j)}(2\pi X_k))\right|^2\}-s_k^{(j)},
\label{eqn:regthm}
\end{eqnarray}
we have 
\[
s_k^{R,(j+1)} = s_k^{(j)} + s^{E,(j)}_k,
\]
where
\begin{equation}
\label{eqn:regthm1}
s^{E,(j)}_k(2\pi x)\defeq \E\{Y_k^{(j)} -s_k^{(j)}(2\pi X_k)|X_k=x\}\neq 0
\end{equation}
due to the perturbation caused by other modes. In the next iteration, the target shape function $s^{(j+1)}_k= - s^{E,(j)}_k$. Hence, the key convergence analysis is to show that $s^{E,(j)}_k$ decays as $j\rightarrow \infty$. 

Remark that the partition-based regression method is only used to provide a decay rate. In the analysis, we assume that all regression problems are solved exactly, i.e., $\bar{s}^{(j+1)}_k=s^{R,(j+1)}_k$ in Equation \eqref{eqn:reg}. 

In what follows, we assume that an accuracy parameter $\epsilon$ is fixed. Furthermore, suppose the given $K$ GIMT's $f_k(t)=\alpha_k(t)s_k(2\pi N_k \phi_k(t))$, $k=1,\dots,K$, have phases in $\WD(M,N,K,h,\beta,\gamma)$, all generalized shape functions and amplitude functions are in the space $\LL$. Under these conditions, all regression functions $s^{(j)}_k\in\LL$ and have bounded $L^\infty$ norm depending only on $M$ and $K$. By Line $8$ in Algorithm \ref{alg:pcg}, we have the nice and key condition that $\int_0^1 s^{(j)}_k(2\pi t)dt=0$ at each iteration for all $k$ and $j$. Note that $\Var(Y_k^{(j)}|X_k=x)$ is bounded by a constant depending only on $M$ and $K$ as well. For the fixed $\epsilon$ and $C$, there exists $h_0(\epsilon,M,K,C)$ such that $C^2h^2<\epsilon^2$ if $0<h<h_0$. By the abuse of notation, $O(\epsilon)$ is used instead of $Ch$ later. By Theorem \ref{thm:reg}, for the fixed $\epsilon$, $M$, $K$, $C$, and $h$, there exists $L_0(\epsilon,M,K,C,h)$ such that the $L^2$ error of the partition-based regression is bounded by $\epsilon^2$. In what follows, $h$ is smaller than $h_0$, $L$ is larger than $L_0$, and hence all estimated regression functions approximate the ground truth regression function with an $L^2$ error of order $\epsilon$. Under these conditions and assumptions,  $s^{E,(j)}_k$ is shown to decay to $O(\epsilon)$ as $j\rightarrow \infty$ and the decay rate will be estimated.

\begin{lemma}
\label{lem:SP}
Under the conditions listed in the paragraph immediately preceding this lemma, for the given $\epsilon$,the estimated regression function $s^{P,(j)}_k$ of the regression problem in \eqref{eqn:regthm} by the partition-based regression method satisfies
\[
\|s^{P,(j)}_k\|_{L^2}\leq O(\epsilon)+M^2(K-1)\beta  \underset{1\leq k\leq K}{\max} \|s_k^{(j)}\|_{L^2}
\]
for all $k=1,\dots,K$ and $j$.
\end{lemma}

We would like to emphasize that $s^{P,(j)}_k$ is only used in the analysis and it is not computed in Algorithm \ref{alg:pcg}. 

\begin{proof}
First, we start with the case when $K=2$ and $\alpha_k(t)=1$ for all $t$ and $k$.

Recall that $p_k(t)$ can be considered as a diffeomorphism from $\R$ to $\R$ transforming data in the $t$ domain to the $p_k(t)$ domain. We have introduced the inverse-warping data
\begin{eqnarray*}
h_k^{(j)}(v)& =& r^{(j)}\circ p_k^{-1}(v)\\
&=&s_k^{(j)}(2\pi v)+ \sum_{j\neq k}  s_j^{(j)}(2\pi   p_j\circ p_k^{-1}(v))\\
&\defeq & s_k^{(j)}(2\pi v)+ \kappa_k^{(j)}(v),
\end{eqnarray*}
where $v = p_k(t)$. After the folding map
\begin{eqnarray*}
\tau: \ \      \left(v, h_k(v) \right)   \mapsto    \left(\text{mod}(v,1),  h_k^{(j)}(v) \right),
\end{eqnarray*}
we have $(x_\ell,y_\ell)=\tau(v_\ell,s_k^{(j)}(2\pi v_\ell)+\kappa_k^{(j)}(v_\ell))$ for $\ell =1,\dots,L$ as $L$ i.i.d. samples of a random vector $(X_k,Y_k^{(j)})$, where $X_k\in[0,1]$. We can assume the target shape functions $s^{(j)}_k$ for all $k$ at the $j$th step are known in the analysis, although they are not known in practice. The partition-based regression method is applied (not necessary to know the distribution of the random vector $(X_k,Y_k^{(j)})$) to solve the following regression problem approximately
\[
 \underset{s:\mathbb{R}\rightarrow \mathbb{R}}{\arg\min}\quad \E\{\left| s(2\pi X_k)-(Y_k^{(j)} -s_k^{(j)}(2\pi X_k))\right|^2\},
 \]
 and the solution is denoted as $s^{P,(j)}_k$. Recall notations in Definition \ref{def:wd}. By the partition-based regression method, when $x\in[t^h_m,t^h_m+h)$,
 \[
 s^{P,(j)}_k(x)=\frac{\sum_{n=0}^{N^h-1}\left( s^{(j)}_i(2\pi t^h_n)+O(\epsilon)\right)D^{ki}_h(m,n)}{D^k_h(m)},
 \]
 where $i=1$ if $k=2$ or $i=2$ if $k=1$, $O(\epsilon)$ comes from the approximation of the $\LL$ function $s_i$ using the values on grid points $t^h_n$. Hence,
  \begin{eqnarray*}
 |s^{P,(j)}_k(x)|&\leq &O(\epsilon)+\frac{\sum_{n=0}^{N^h-1}s^{(j)}_i(2\pi t^h_n) D^{ki}_h(m,n)}{D^k_h(m)}\\
 &=&O(\epsilon)+\frac{\sum_{n=0}^{N^h-1}s^{(j)}_i(2\pi t^h_n) \left(D^{ki}_h(m,n)-\gamma\right)}{D^k_h(m)}+ \frac{ \gamma\sum_{n=0}^{N^h-1}s^{(j)}_i(2\pi t^h_n)}{D^k_h(m)}. 
 \end{eqnarray*}
 Since $s^{(j)}_i\in\LL$ and $\int_0^1s^{(j)}_i(2\pi t)dt=0$, $|\sum_{n=0}^{N^h-1}s^{(j)}_i(2\pi t^h_n)|\leq C$. Note that $D^k_h(m)\geq N^h\gamma=\gamma/h$. Hence,
 \[
 \left| \frac{ \gamma\sum_{n=0}^{N^h-1}s^{(j)}_i(2\pi t^h_n)}{D^k_h(m)}\right|\leq O(\epsilon),
 \]
 and
 \[
  |s^{P,(j)}_k(x)|\leq O(\epsilon)+\frac{\sum_{n=0}^{N^h-1}s^{(j)}_i(2\pi t^h_n) \left(D^{ki}_h(m,n)-\gamma\right)}{D^k_h(m)}.
 \]
 By the triangle inequality, 
 \begin{eqnarray*}
 \|s^{P,(j)}_k\|_{L^2}&\leq& O(\epsilon) +\left( \sum_{m=0}^{N^h-1} \left( \frac{\sum_{n=0}^{N^h-1}s^{(j)}_i(2\pi t^h_n) \left(D^{ki}_h(m,n)-\gamma\right)}{D^k_h(m)} \right)^2 h \right)^{1/2}\\
 &\leq & O(\epsilon) +\left( \sum_{m=0}^{N^h-1}\left( \sum_{n=0}^{N^h-1}\left(s^{(j)}_i(2\pi t^h_n) \right)^2\right) \left(\sum_{n=0}^{N^h-1}\left(  \frac{ D^{ki}_h(m,n)-\gamma}{D^k_h(m)} \right)^2\right) h \right)^{1/2}\\
 &=&O(\epsilon) +\left( \sum_{n=0}^{N^h-1}\left(s^{(j)}_i(2\pi t^h_n) \right)^2h \right) ^{1/2}\left( \sum_{m=0}^{N^h-1}\left(\sum_{n=0}^{N^h-1}\left(  \frac{ D^{ki}_h(m,n)-\gamma}{D^k_h(m)} \right)^2\right)  \right)^{1/2},
 \end{eqnarray*}
 where the second inequality comes from H{\"o}lder's inequality. Since $s^{(j)}_i\in\LL$, 
 \[
 \left( \sum_{n=0}^{N^h-1}\left(s^{(j)}_i(2\pi t^h_n) \right)^2h \right) ^{1/2}=\left(\|s^{(j)}_i\|^2_{L^2}+O(\epsilon)\right)^{1/2}=\|s^{(j)}_i\|_{L^2}+O(\epsilon).
 \]
 Since phase functions are in $\WD(M,N,h,\beta,\gamma)$, 
 \[
 \left( \sum_{m=0}^{N^h-1}\left(\sum_{n=0}^{N^h-1}\left(  \frac{ D^{ki}_h(m,n)-\gamma}{D^k_h(m)} \right)^2\right)  \right)^{1/2}\leq \beta<1.
 \]
 Hence,
 \[
  \|s^{P,(j)}_k\|_{L^2}\leq O(\epsilon)+\beta \|s^{(j)}_i\|^2_{L^2}\leq O(\epsilon)+ \beta \underset{1\leq k\leq K}{\max} \|s_k^{(j)}\|_{L^2}.
 \]
 
To care the general case, we need to extend the argument to $K>2$ and non-constant $\alpha_k$. We shall do this in two steps: first $K>2$ but $\alpha_k\equiv 1$ for all $k$, and then, finally, $K>2$ and varying $\alpha_k$. Rather than repeating the earlier argument in full detail, adapted to these more general situations, we indicate simply, for both steps, what extra estimates need to be taken into account. This may not give the sharpest estimate, but this is not a concern for now.
 
 Next, we prove the case when $K>2$ and $\alpha_k(t)=1$ for all $t$ and $k$. Similarly, by the definition of the partition-based regression and the triangle inequality, we have
 \[
  |s^{P,(j)}_k(x)|\leq O(K\epsilon)+ \sum_{i\neq k} \lambda_i \frac{\sum_{n=0}^{N^h-1}s^{(j)}_i(2\pi t^h_n) \left(D^{ki}_h(m,n)-\gamma\right)}{D^k_h(m)}.
 \]
 Hence, by the triangle inequality and the H{\"o}lder inequality again, it holds that
  \begin{eqnarray*}
 \|s^{P,(j)}_k\|_{L^2}&\leq& O(K\epsilon) +  \sum_{i\neq k}  \left( \sum_{m=0}^{N^h-1} \left( \frac{\sum_{n=0}^{N^h-1}s^{(j)}_i(2\pi t^h_n) \left(D^{ki}_h(m,n)-\gamma\right)}{D^k_h(m)} \right)^2 h \right)^{1/2}\\
 &\leq &O(K\epsilon) + \sum_{i\neq k}  \left( \sum_{n=0}^{N^h-1}\left(s^{(j)}_i(2\pi t^h_n) \right)^2h \right) ^{1/2}\left( \sum_{m=0}^{N^h-1}\left(\sum_{n=0}^{N^h-1}\left(  \frac{ D^{ki}_h(m,n)-\gamma}{D^k_h(m)} \right)^2\right)  \right)^{1/2}\\
  &\leq &O(K\epsilon) + \sum_{i\neq k}\beta \|s^{(j)}_i\|_{L^2}\\
  &\leq & O(K\epsilon) + (K-1) \beta  \underset{1\leq k\leq K}{\max} \|s_k^{(j)}\|_{L^2}\\
  &=& O(\epsilon) +(K-1) \beta  \underset{1\leq k\leq K}{\max} \|s_k^{(j)}\|_{L^2}.
 \end{eqnarray*}
 
 Finally, we prove the case when amplitude functions are smooth functions but not a constant $1$. If the instantaneous frequencies are sufficiently large, depending on $\epsilon$, $M$, $K$, and $C$, amplitude functions are nearly constant up to an approximation error of order $\epsilon$. The time domain $[0,1]$ is divided into sufficiently small intervals such that amplitude functions are nearly constant inside each interval. Accordingly, the samples $(x_\ell,y_\ell)=\tau(v_\ell,s_k(2\pi v_\ell)+\kappa_k(v_\ell))$ for $\ell =1,\dots,L$ of the random vector $(X_k,Y_k^{(j)})$ is divided into groups and the partition-based regression method is applied to estimate the regression function for each group. This is similar to data splitting in nonparametric regression. The bound of $ |s^{P,(j)}_k(x)|$ is a weighted average of the bound given by each group, and the weight comes the number of points in each group over the total number of samples. Note that $\|\alpha_k\|_{L^\infty}\leq M$. By repeating the analysis above, it is simple to show 
 \[
\|s^{P,(j)}_k\|_{L^2}\leq O(\epsilon)+M^2(K-1)\beta  \underset{1\leq k\leq K}{\max} \|s_k^{(j)}\|_{L^2},
\]
where $M^2$ comes from 
\[
 \frac{\alpha_i\circ p_k^{-1}(v)}{\alpha_k\circ p_k^{-1}(v)}
\]
in $\kappa_k(v)$ after warping.
\end{proof}

\begin{lemma}
\label{lem:SE}
Under the conditions in Lemma \ref{lem:SP}, $s^{E,(j)}_k$ in  \eqref{eqn:regthm1} satisfies
\[
\|s^{E,(j)}_k\|_{L^2}\leq O(\epsilon)+M^2(K-1)\beta \underset{1\leq k\leq K}{\max} \|s_k^{(j)}\|_{L^2}
\]
for all $k=1,\dots,K$ and $j$.
\end{lemma}

\begin{proof}
Let $s^{P,(j)}_k$ be the estimated regression function constructed in Lemma \ref{lem:SP}, then
\begin{eqnarray*}
\|s^{E,(j)}_k\|_{L^2}&\leq &\|s^{E,(j)}_k-s^{P,(j)}_k\|_{L^2}+\|s^{P,(j)}_k\|_{L^2}\\
&\leq &O(\epsilon)+M^2(K-1)\beta \underset{1\leq k\leq K}{\max} \|s_k^{(j)}\|_{L^2},
\end{eqnarray*}
by Theorem \ref{thm:reg} and Lemma \ref{lem:SP}.
\end{proof}


\begin{theorem}
\label{thm:conv} (Convergence of the RDBR) 
Under the conditions in Lemma \ref{lem:SP}, we have  
\[
\underset{1\leq k\leq K}{\max} \|s^{E,(j)}_k\|_{L^2}\leq O(c\epsilon+\left(M^2(K-1)\beta\right)^{j+1}),
\]
and
\[
\|r^{(j)}\|_{L^2}\leq O(c\epsilon+\left(M^2(K-1)\beta\right)^{j}),
\]
where $c=\frac{1}{1-M^2(K-1)\beta}$ is a finite number, $s^{E,(j)}_k$  is defined in Equation \eqref{eqn:regthm1} and $r^{(j)}$ is defined in Equation \eqref{eqn:r}.
\end{theorem}

\begin{proof}
This theorem is an immediate result of Lemma \ref{lem:SE} by induction. Let 
\[
c_1 = M^2(K-1)\beta,
\]
and
\[
c_2(j) = \sum_{n=0}^{j} c_1^n,
\]
then $c_2(j)<\lim_{j\rightarrow \infty}c_2(j)=c$ for all $j$ and $c$ is a finite number because $0<M^2(K-1)\beta<1$ by the assumption. In the initial step when $j=0$, $s^{(j)}_k=s_k$, $\|s_k\|_{L^\infty}\leq M$, and $\|\alpha_k\|_{L^\infty}\leq M$. Hence, by Lemma \ref{lem:SE},
\begin{eqnarray*}
\underset{1\leq k\leq K}{\max} \|s^{E,(0)}_k\|_{L^2}&\leq &O(\epsilon)+M^2(K-1)\beta  \underset{1\leq k\leq K}{\max} \|s_k^{(0)}\|_{L^2}\\
&\leq& O(\epsilon)+M^2(K-1)\beta\sqrt{2\pi}M\\
&=&O(\epsilon c_2(0)+M^2(K-1)\beta),
\end{eqnarray*}
\[
\|r^{(0)}\|_{L^2}\leq \sum_{k=1}^K \| \alpha_k(t) s_k(2 \pi N_k \phi_k(t)) \|_{L^2}\leq KM^2=O(1)
\]
and the conclusion holds. When $j\neq 0$, we have
\begin{equation}
\label{eqn:ind1}
s_k^{(j)} = - s^{E,(j-1)}_k,
\end{equation}
where 
\begin{equation}
\label{eqn:ind2}
\underset{1\leq k\leq K}{\max} \|s^{E,(j-1)}_k\|_{L^2}\leq O(\epsilon c_2(j-1)+\left(M^2(K-1)\beta\right)^{j}).
\end{equation}
By Lemma \ref{lem:SE}, Equation \eqref{eqn:ind1} and \eqref{eqn:ind2},
\begin{eqnarray*}
\|s^{E,(j)}_k\|_{L^2}& \leq & O(\epsilon)+M^2(K-1)\beta \underset{1\leq k\leq K}{\max} \|s_k^{(j)}\|_{L^2}\\
& \leq & O(\epsilon c_2(j)+\left(M^2(K-1)\beta\right)^{(j+1)})
\end{eqnarray*}
for all $k$. Hence, 
\begin{equation*}
\underset{1\leq k\leq K}{\max} \|s^{E,(j)}_k\|_{L^2}\leq O(\epsilon c_2(j)+\left(M^2(K-1)\beta\right)^{j+1}).
\end{equation*}

By \eqref{eqn:ind1} and \eqref{eqn:ind2},
\begin{eqnarray*}
\|r^{(j)}\|_{L^2}&=&\|\sum_{k=1}^K \alpha_k(t) s^{(j)}_k(2 \pi p_k(t))\|_{L^2}\\
&\leq& \sum_{k=1}^K M \| s^{(j)}_k(2 \pi p_k(t))\|_{L^2}\\
&\leq &O(\epsilon c_2(j-1)+\left(M^2(K-1)\beta\right)^{j}).
\end{eqnarray*}
Note that for all $j$ we have $0<c_2(j)<\lim_{j\rightarrow\infty}c_2(j)=c$. Hence, this theorem holds for $j\neq 0$.
\end{proof}

Theorem \ref{thm:conv} shows that the regression function in each step of Algorithm \ref{alg:pcg} decays, if $M^2(K-1)\beta<1$, in the $L^2$ sense up to a fixed accuracy parameter as the iteration number becomes large. Hence, the recovered shape function $\tilde{s}_k$ converges and the residual decays up to a fixed accuracy parameter, if $M^2(K-1)\beta<1$. When the iteration number is sufficiently large, the accuracy of the RDBR in Theorem \ref{thm:conv} is as good as a single step of regression in Theorem \ref{thm:reg}.

\begin{theorem}
\label{thm:ns} (Robustness of the RDBR) 
Let $f_k(t)=\alpha_k(t)s_k(2\pi N_k \phi_k(t))$, $k=1,\dots,K$, be $K$ GIMT's and $f(t)=\sum_{k=1}^Kf_k(t)+n(t)$, where $n(t)$ is a random noise with a bounded variance $\sigma^2$. Under the other conditions introduced in Theorem \ref{thm:conv}, for the given $\epsilon$,  $\exists L_0(\epsilon,M,K,C,h,\sigma)$, if $L>L_0$, then 
\[
\underset{1\leq k\leq K}{\max} \|s^{E,(j)}_k\|_{L^2}\leq O(c\epsilon+\left(M^2(K-1)\beta\right)^{j+1}),
\]
and
\[
\|r^{(j)}\|_{L^2}\leq O(c\epsilon+\left(M^2(K-1)\beta\right)^{j}), 
\]
where $c=\frac{1}{1-M^2(K-1)\beta}$ is a finite number, $s^{E,(j)}_k$  is defined in Equation \eqref{eqn:regthm1} and $r^{(j)}$ is defined in Equation \eqref{eqn:r}.
\end{theorem}

\begin{proof}
The proof basically follows the one in Theorem \ref{thm:conv}. The difference is that the number of samples in $(X_k,Y_k^{(j)})$ should be large enough, depending on $\sigma^2$, such that Lemma \ref{lem:SP} is still true.
\end{proof}

Theorem \ref{thm:ns} shows that as soon as the number of sampling points $L$ is large enough, the noise effect will be negligible and the RDBR method can still identify generalized shape functions accurately.

\section{Numerical examples}
\label{sec:results}
In this section, some numerical examples of synthetic and real data are provided to demonstrate the proposed properties of the RDBR. We apply the least squares spline regression method with free knots in \cite{Regression} to solve all the regression problems in this paper. The implementation of this method is available online\footnote{Available at https://www.mathworks.com/matlabcentral/fileexchange/25872-free-knot-spline-approximation.}. In all synthetic examples, we assume the fundamental instantaneous phases and amplitudes are known and only focus on verifying the theory of the RDBR in Section \ref{sec:RDBR}. In real examples, we apply the one-dimensional synchrosqueezed wave packet transform (SSWPT) to estimate instantaneous phases and amplitudes as inputs of the RDBR.  The implementation of the SSWPT is also publicly available in SynLab\footnote{Available at https://github.com/HaizhaoYang/SynLab.}. The code for the RDBR will be available online shortly\footnote{Will be available at https://github.com/HaizhaoYang/DeCom.}.

Before presenting results, we would like to summarize the main parameters in the above packages and in Algorithm \ref{alg:pcg}. In the spline regression with free knots, main parameters are
\begin{itemize}
\item $nk$:  the number of free knots;
\item $krf$: the knot removal factor, a number quantifying how likely a free knot would be removed;
\item $ord$: the highest degree of spline polynomials.
\end{itemize}
In SynLab, main parameters are
\begin{itemize}
\item $s$: a geometric scaling parameter;
\item $rad$: the support size of the mother wave packet in the Fourier domain;
\item $red$: a redundancy parameter, the number of frames in the wave packet transform;
\item $\epsilon_{sst}$: a threshold for the wave packet coefficients.
\end{itemize}
In Algorithm \ref{alg:pcg}, main parameters are
\begin{itemize}
\item $mIter$: the maximum number of iterations allowed;
\item $\epsilon$: the accuracy parameter.
\end{itemize}
For the purpose of convenience, the synthetic data is defined in $[0,1]$ and sampled on a uniform grid. All these parameters in different examples are summarized in Table \ref{tab:1}.

\begin{table}[htp]
\centering
\begin{tabular}{rcccccccccc}
\toprule
  figure &  $nk$ & $krf$ & $ord$ & $s$
                             & $rad$ & $red$ & $\epsilon_{sst}$ & $mIter$ & $\epsilon$ & $L$ \\
\toprule
 6 & 20 & 1.01 & 3 & -- & -- & -- & -- & 4000 & 1e-6 & $2^{16}$ \\
 7 (left) & 20 & 1.01 & 3 & -- & -- & -- & -- & 9 & 1e-13 & $2^{12}$\\
 7 (middle) & 20 & 1.01 & 3 & -- & -- & -- & -- & 200 & 1e-13 & -- \\
 7 (right) & 20 & 1.01 & 3 & -- & -- & -- & -- & -- & -- & -- \\
 8 (clean) & 20 & 1.01 & 3 & -- & -- & -- & -- & 200 & 1e-6 & $2^{12}$ \\
 8 (noisy) & 20 & 1.0001 & 3 & -- & -- & -- & -- & 200 & 1e-10 & $2^{16}$ \\
 9 (clean) & 20 & 1.01 & 2 & -- & -- & -- & -- & 200 & 1e-6 & $2^{12}$ \\
 9 (noisy) & 20 & 1.01 & 2 & -- & -- & -- & -- & 200 & 1e-6 & $2^{16}$ \\
 11 (clean) & 20 & 1.01 & 3 & -- & -- & -- & -- & 200 & 1e-6 & $2^{12}$ \\
 11 (noisy) & 20 & 1.001 & 3 & -- & -- & -- & -- & 200 & 1e-6 & $2^{16}$ \\
 12-14 (clean) & 20 & 1.001 & 3 & 0.66 & 1 & 8 & 1e-3 & 200 & 1e-10 & $2^{14}$ \\
 12-14 (noisy) & 20 & 1.001 & 3 & 0.66 & 1 & 8 & 1e-3 & 200 & 1e-10 & $2^{16}$ \\
 15 (noisy) & 20 & 1.01 & 3 & 0.66 & 1 & 20 & 1e-3 & 200 & 1e-10 & 1600 \\
\bottomrule
\end{tabular}
\caption{Parameters in the spline regression, SynLab, and Algorithm \ref{alg:pcg}. The notation ``--" means the corresponding parameter is not used or its value can be found in the description of its corresponding example.}
\label{tab:1}
\end{table}

\begin{figure}
  \begin{center}
    \begin{tabular}{ccc}
      \includegraphics[height=1.8in]{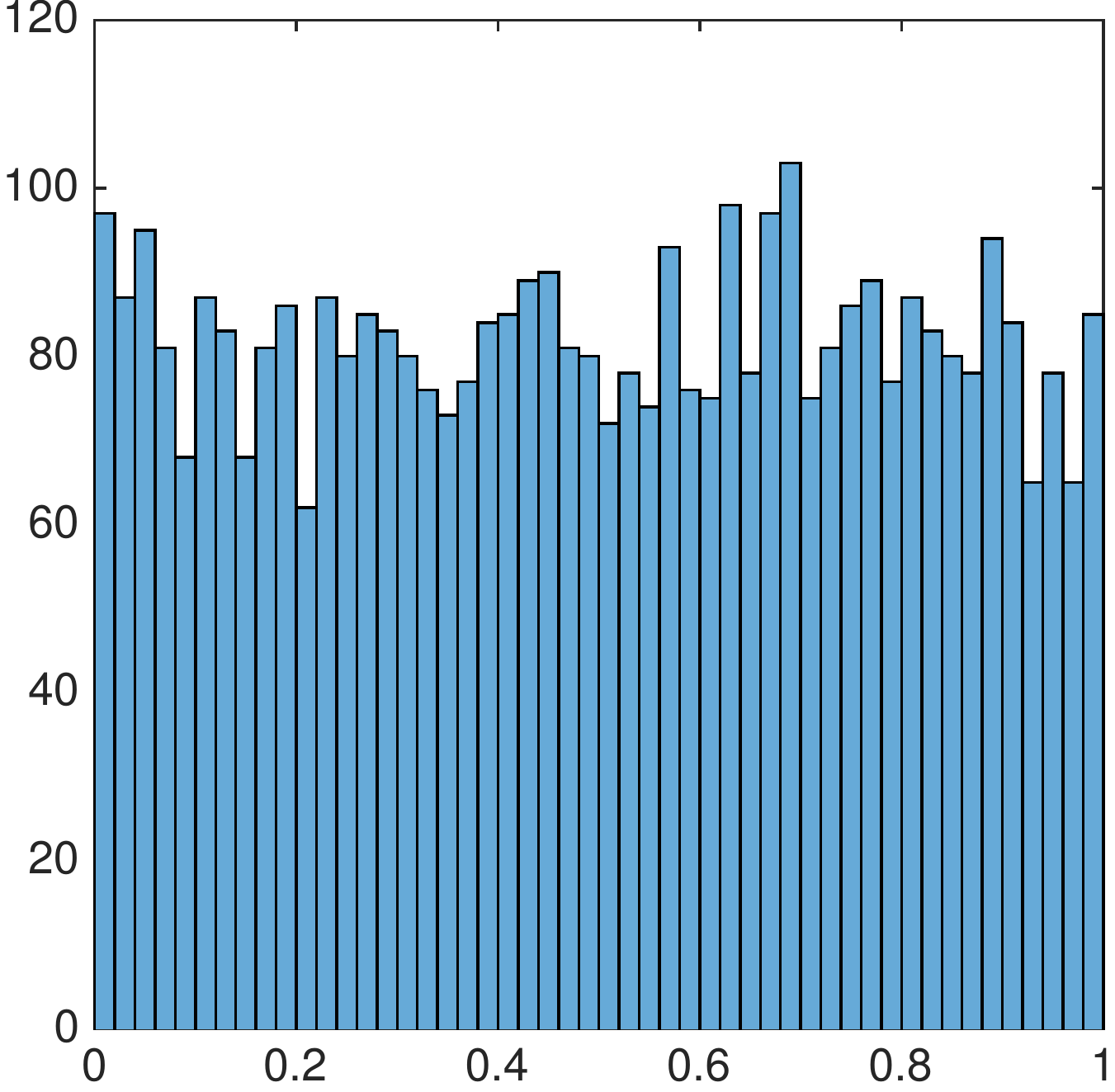} &
      \includegraphics[height=1.8in]{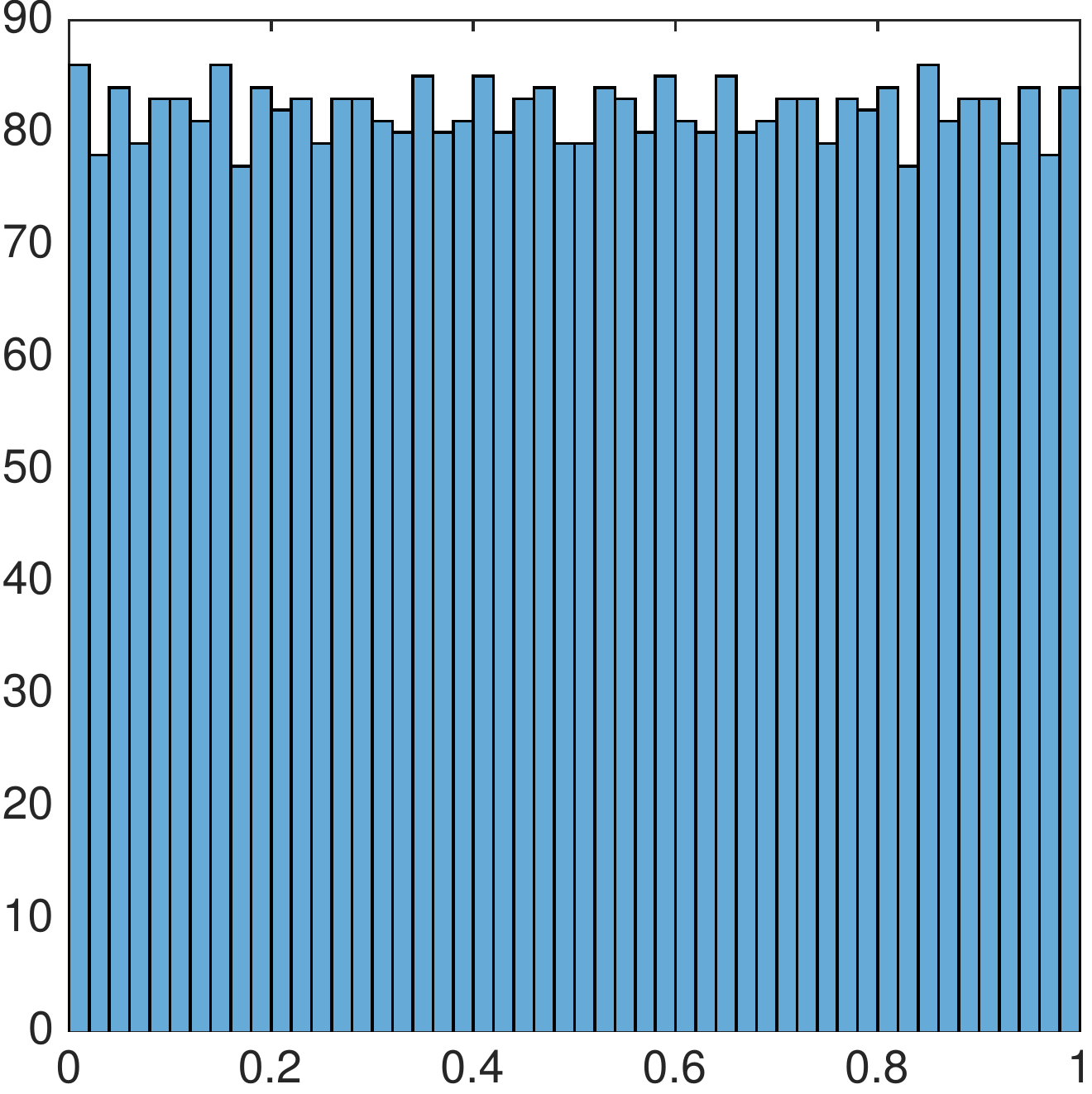} &
      \includegraphics[height=1.8in]{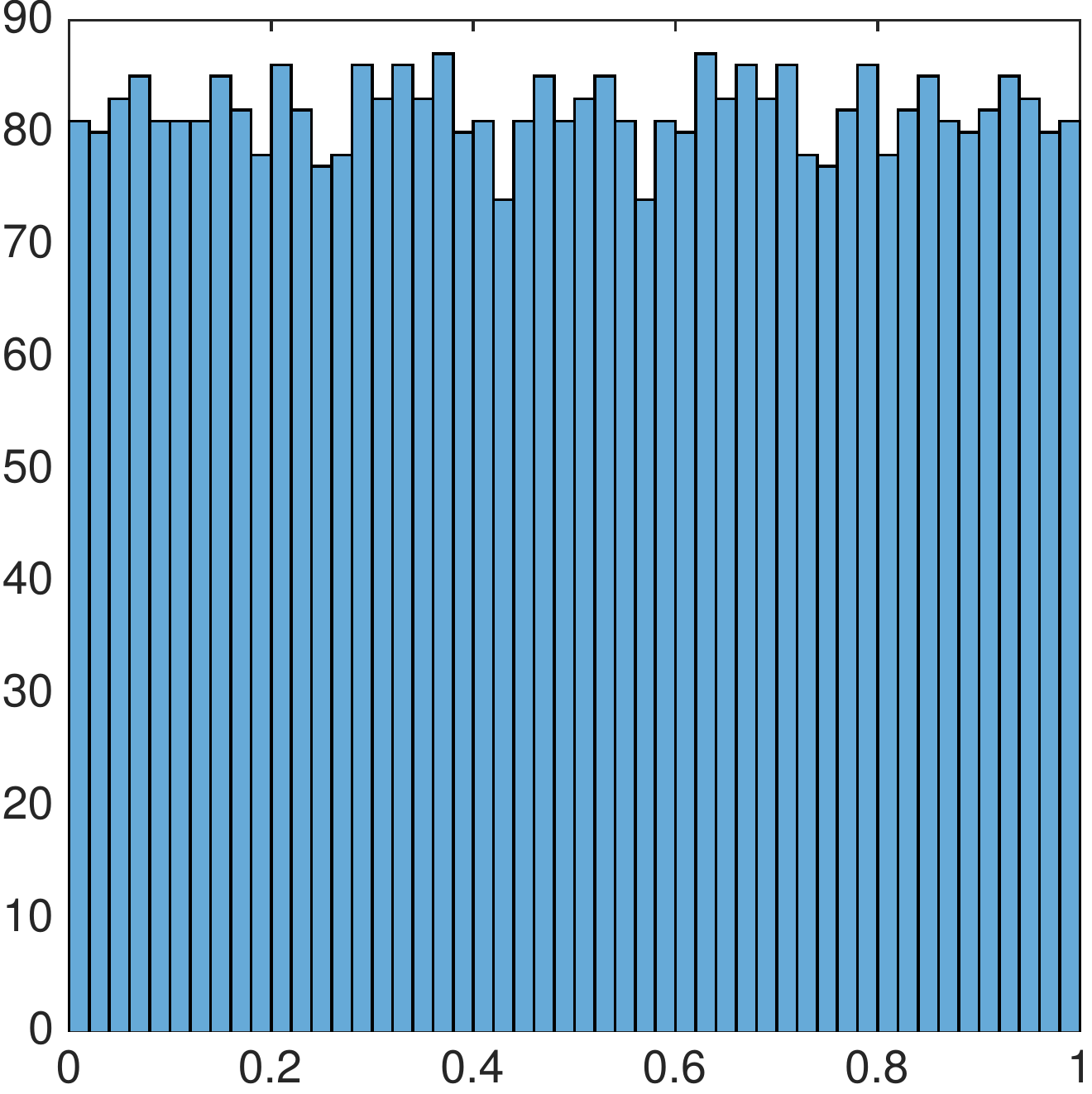} \\
      \includegraphics[height=1.8in]{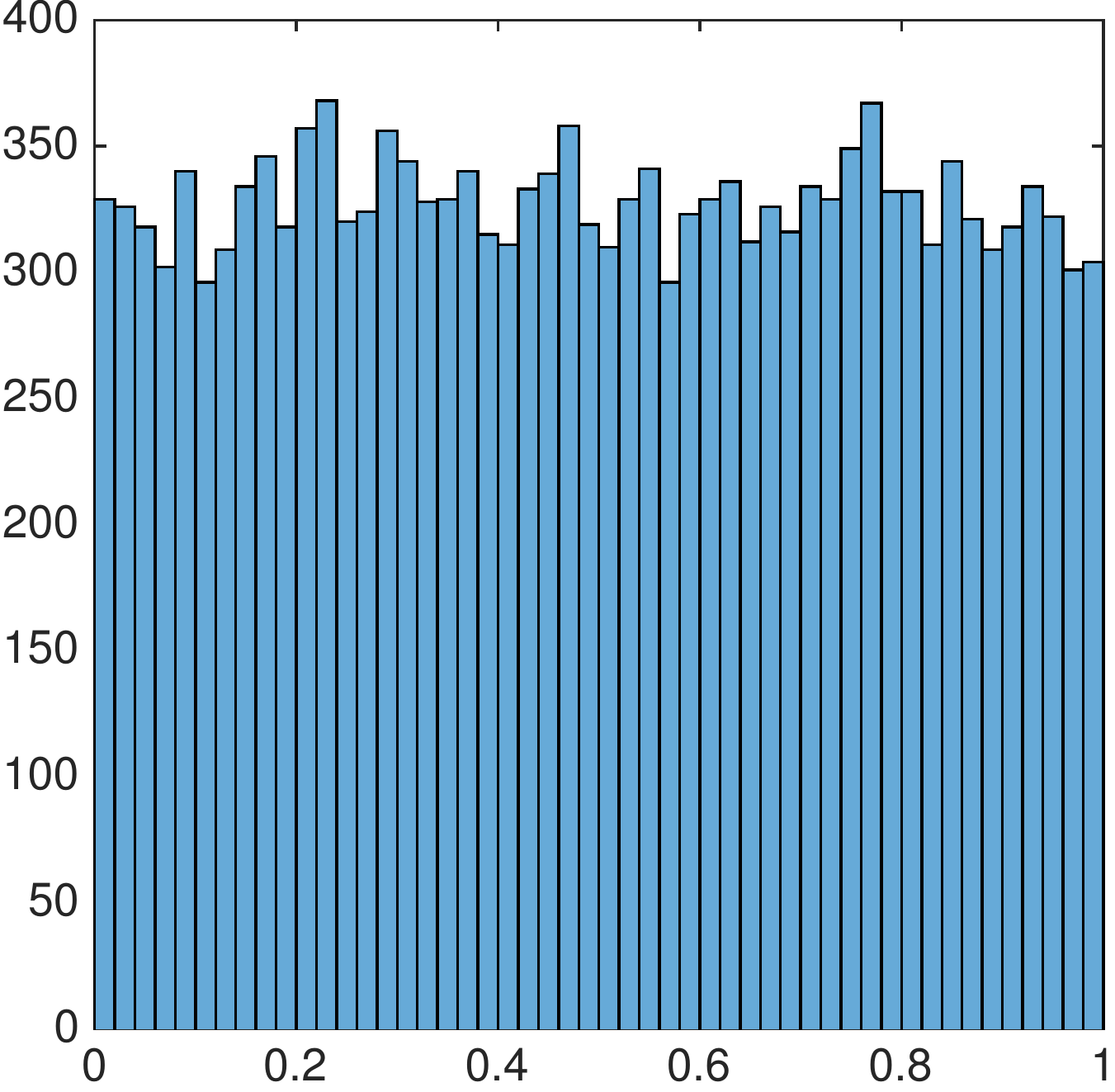} &
      \includegraphics[height=1.8in]{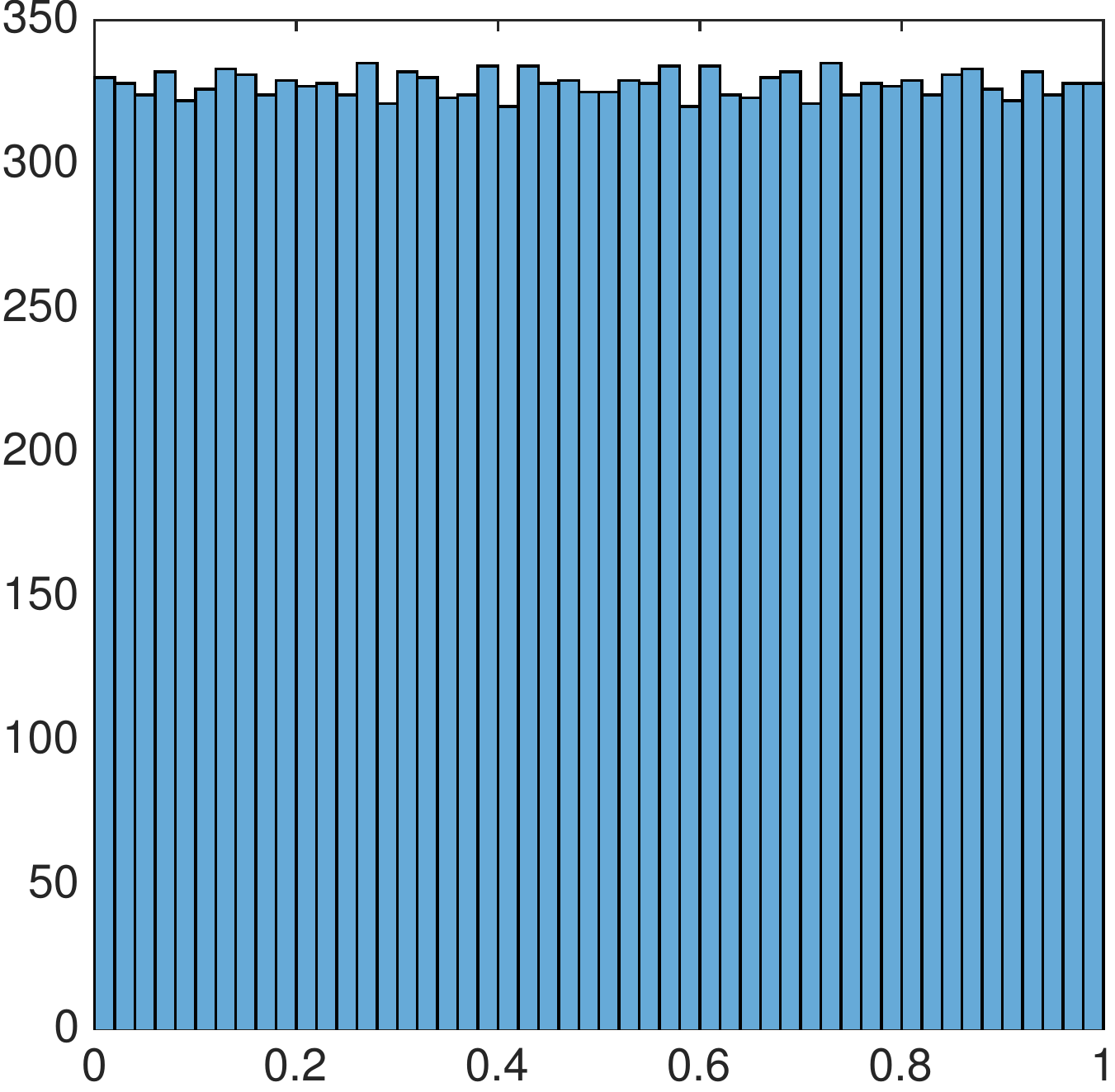} &
      \includegraphics[height=1.8in]{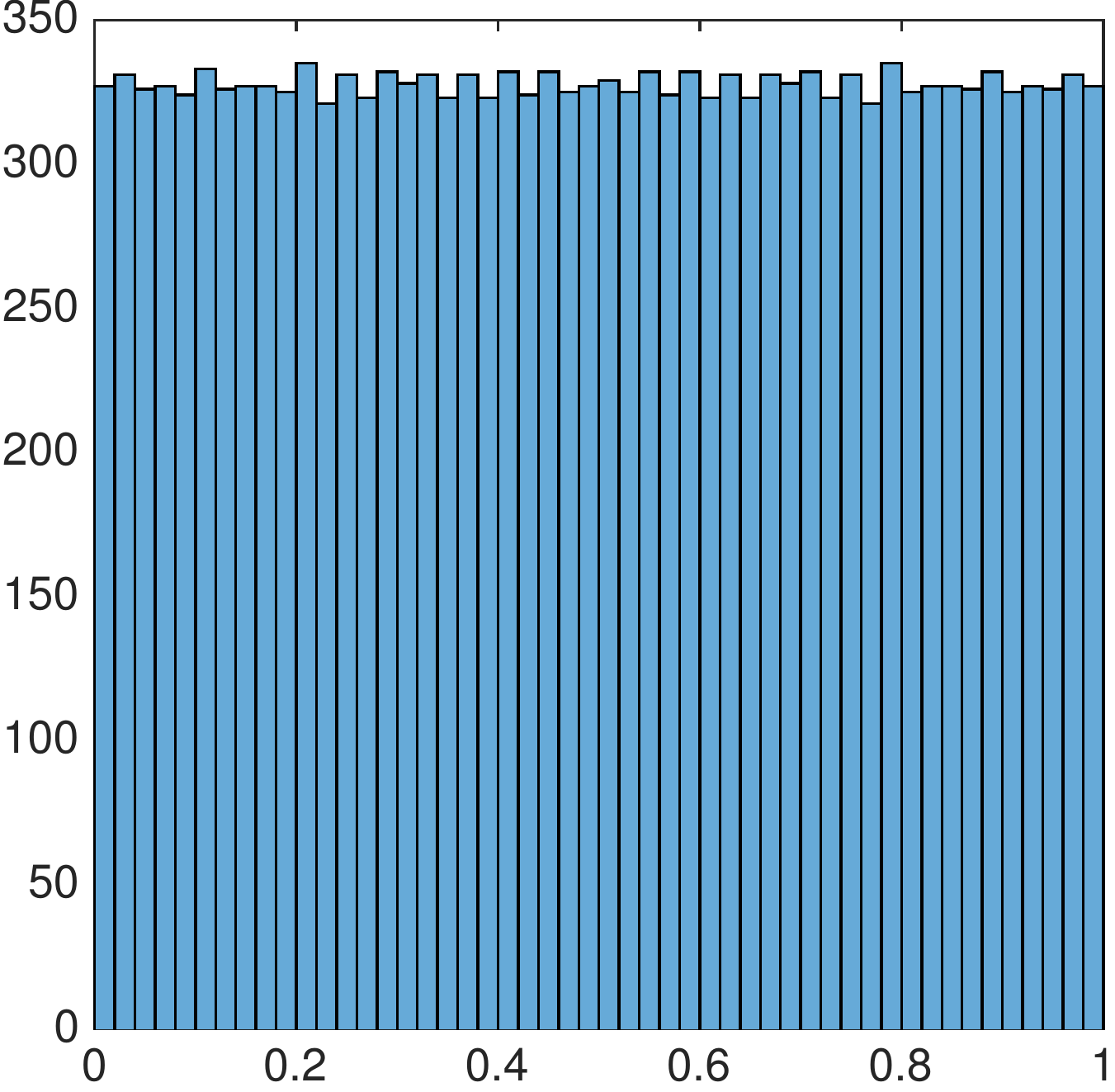} \\
      \includegraphics[height=1.8in]{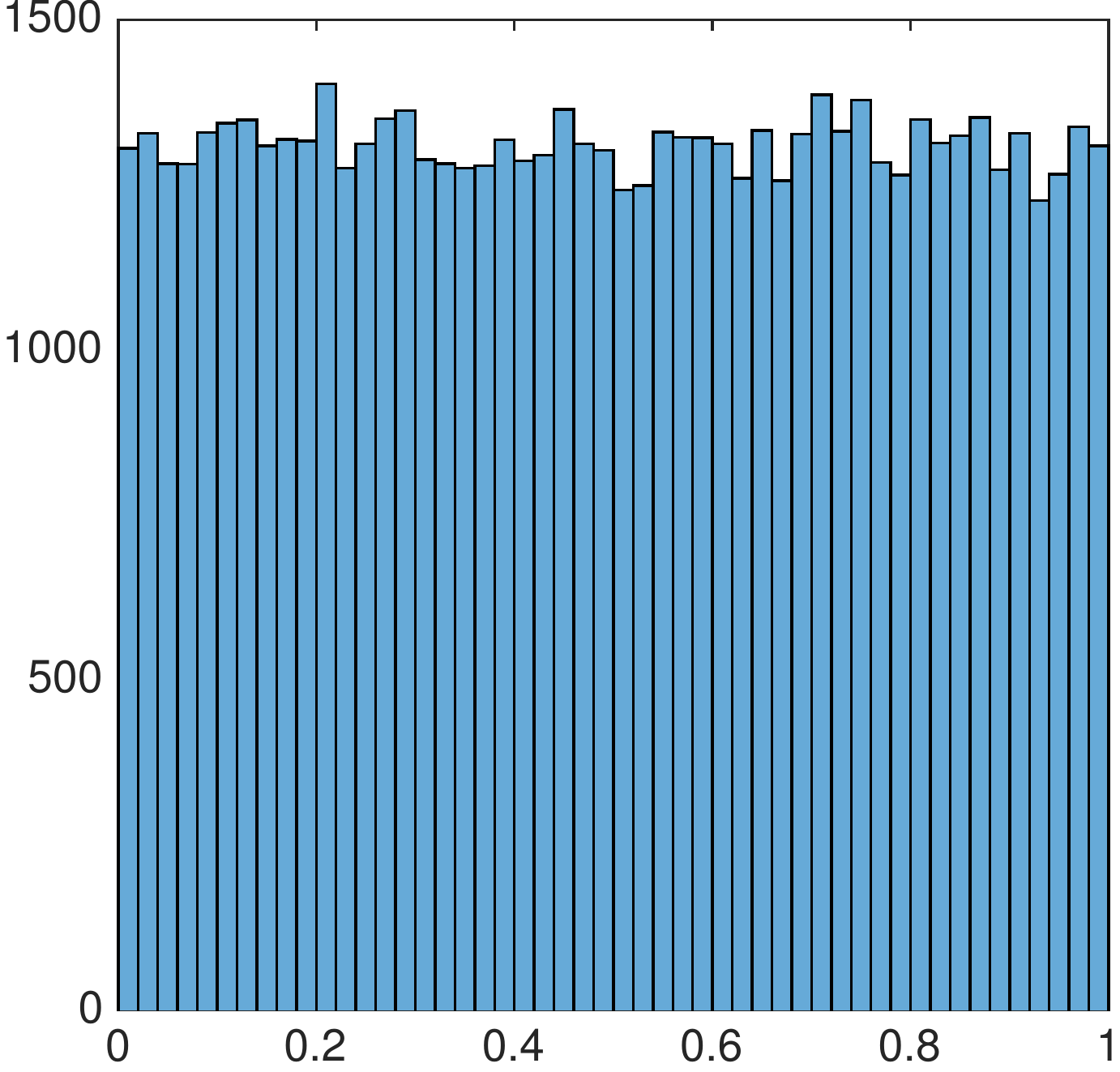} &
      \includegraphics[height=1.8in]{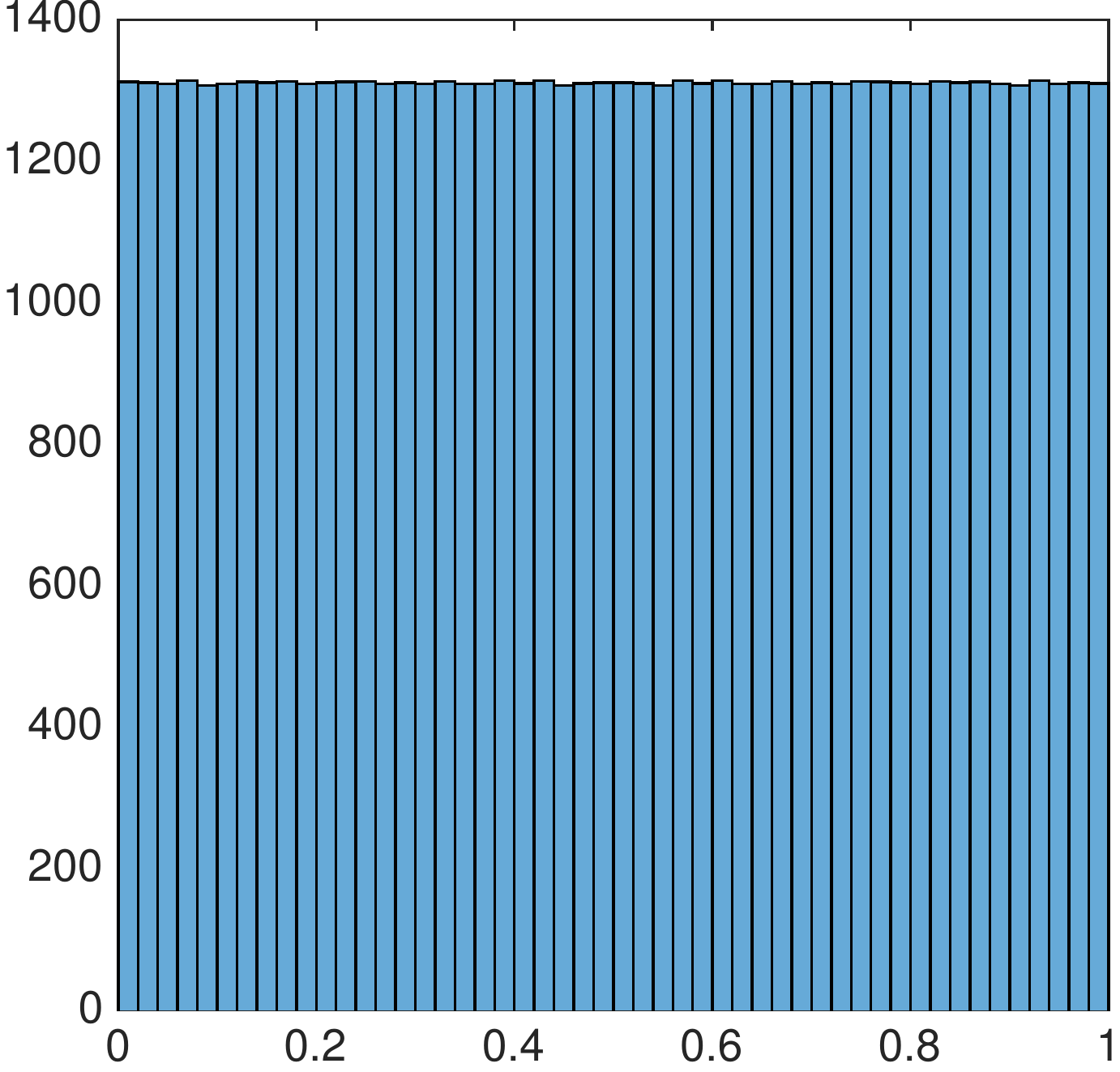} &
      \includegraphics[height=1.8in]{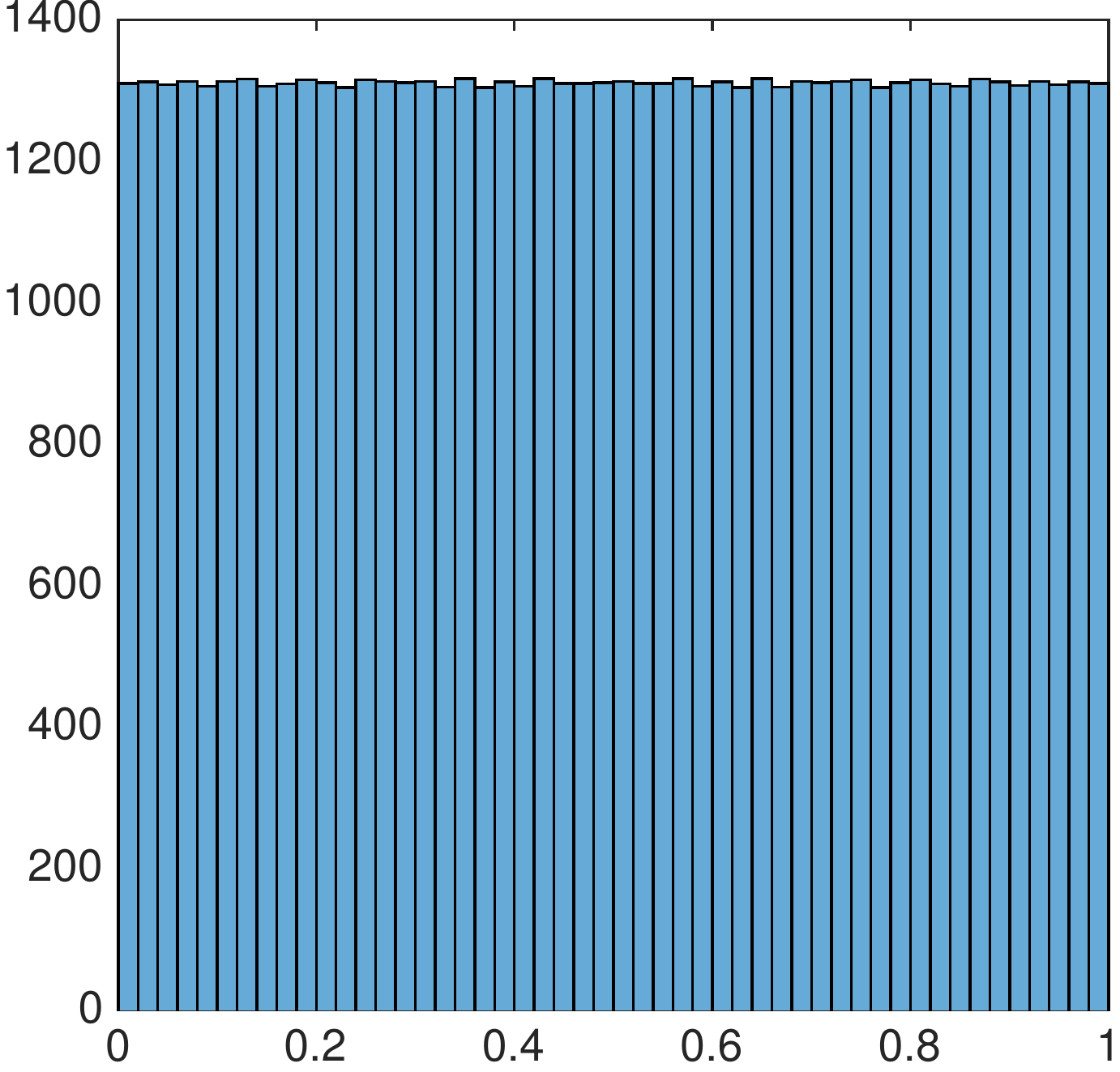} 
    \end{tabular}
  \end{center}
  \caption{Histograms of the point sets $S_0$ (left column), $S_1$ (middle column), and $S_2$ (right column) with a uniform bin size $0.02$. Here $S_1$ and $S_2$ are defined in Equation \ref{eqn:S1} and Equation \ref{eqn:S2}, respectively, and $S_0$ is defined right below Equation \ref{eqn:S2}. From top to bottom, the number of samples in $S_k$ (for all $k=0$, $1$, and $2$) is $L=2^{12}$, $2^{14}$, and $2^{16}$, respectively. These histograms show that points in $S_1$ and $S_2$ are approximately uniformly distributed in $[0,1)$.}
  \label{fig:5}
\end{figure}

In the noisy synthetic examples, Gaussian random noise with a distribution $\N(\sigma^2,0)$ is used. We follow the definition of the signal-to-noise ration ($\SNR$) in \cite{1DSSWPT}:
\begin{equation}
\label{eqn:SNR}
\SNR [dB]=\min \left\{10\log_{10}\left(\frac{\|f_i\|_{L^2}}{\sigma^2}\right),1\leq i\leq K\right\},
\end{equation}
where $\{f_i\}_{i=1}^K$ are the generalized modes contained in the signal $f(t)$ to be analyzed.

\subsection{Numerical distribution of sampling points}
\label{sub:iid}
In this section, we provide numerical examples to support the assumption that, the collection of samples after warping and folding behaves essentially like a collection of i.i.d. samples, in the analysis of the RDBR in Section \ref{sec:RDBR}. 

Let us revisit the example in \eqref{eqn:ex1} and choose its instantaneous phase functions 
\[
p_1(t) = 60(t+0.01\sin(2\pi t))
\]
and 
\[
p_2(t) = 90(t+0.01\cos(2\pi t))
\]
to define warping and folding maps
\begin{eqnarray*}
\tau_1: \ \     t   \mapsto    \text{mod}(p_1(t),1),
\end{eqnarray*}
and
\begin{eqnarray*}
\tau_2: \ \     t   \mapsto    \text{mod}(p_2(t),1).
\end{eqnarray*}
Let $\T$ denote the collection of $L$ uniform grid points $\{t_n\}_{n=0,\dots,L-1}$ in $[0,1]$ with a step side $1/L$ and $t_n = n/L$. The warping and folding map $\tau_1$ (or $\tau_2$) acts like a pseudorandom number generator mapping $\T$ to a collection of i.i.d. samples of a random variable $X$ with a uniform distribution in $[0,1]$, i.e., sample points in 
\begin{equation}
\label{eqn:S1}
S_1=\tau_1(\T)
\end{equation}
and 
\begin{equation}
\label{eqn:S2}
S_2=\tau_2(\T)
\end{equation}
are approximately uniformly distributed in $[0,1]$. To support this claim numerically, we randomly sample $L$ points from a random variable with a uniform distribution in $[0,1]$ and denote this set of samples as $S_0$. The distribution of $S_k$ for $k=0$, $1$, and $2$, when $L=2^{12}$, $2^{14}$, and $2^{16}$, are summarized in Figure \ref{fig:5}. Figure \ref{fig:5} shows that points in $S_1$ and $S_2$ have a more uniform distribution than those in $S_0$. Hence, $S_1$ and $S_2$  would lead to better results in the regression than $S_0$.

\subsection{Convergence of the RDBR}

In this section, numerical examples are provided to verify the convergence analysis in Section \ref{sec:RDBR}. In the analysis, for a fixed accuracy parameter $\epsilon$, we have shown that as long as the fundamental instantaneous frequencies are sufficiently high and the number of samples is large enough, the RDBR is able to estimate shape functions from a class of superpositions of generalized modes. The residual error in the iterative scheme linearly converges to a quantity of order $\epsilon$. Since it is difficult to specify the relation of the rate of convergence and other parameters explicitly in the analysis, we provide numerical examples to study this rate quantitatively.

\begin{figure}[ht!]
  \begin{center}
    \begin{tabular}{ccc}
      \includegraphics[height=1.6in]{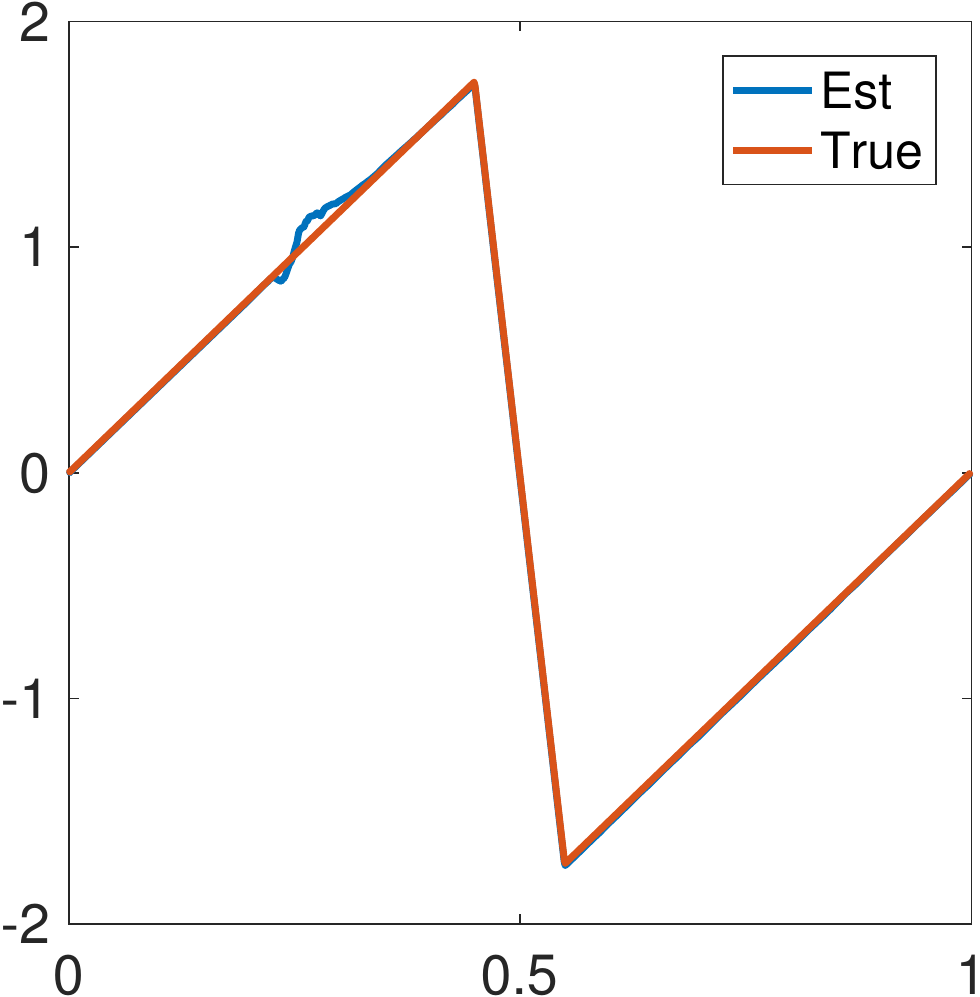}  &
   \includegraphics[height=1.6in]{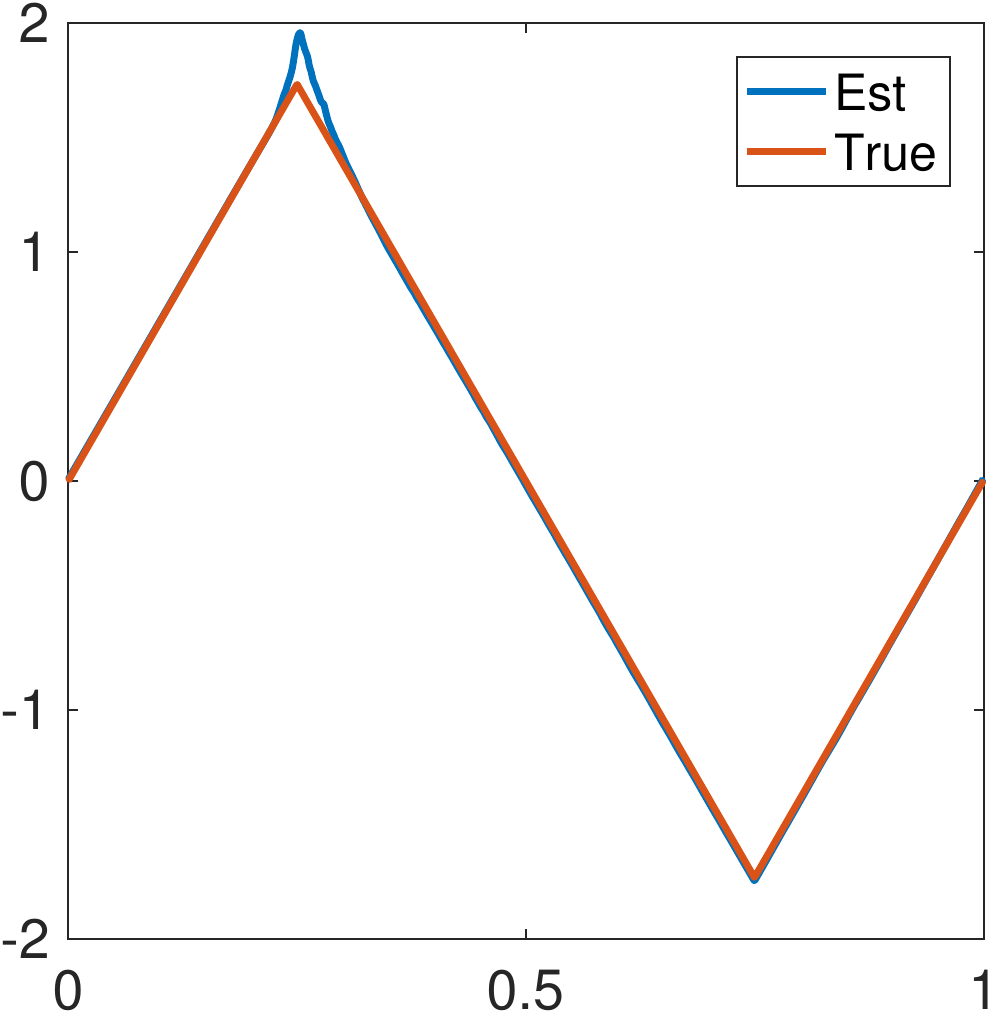}  &
   \includegraphics[height=1.65in]{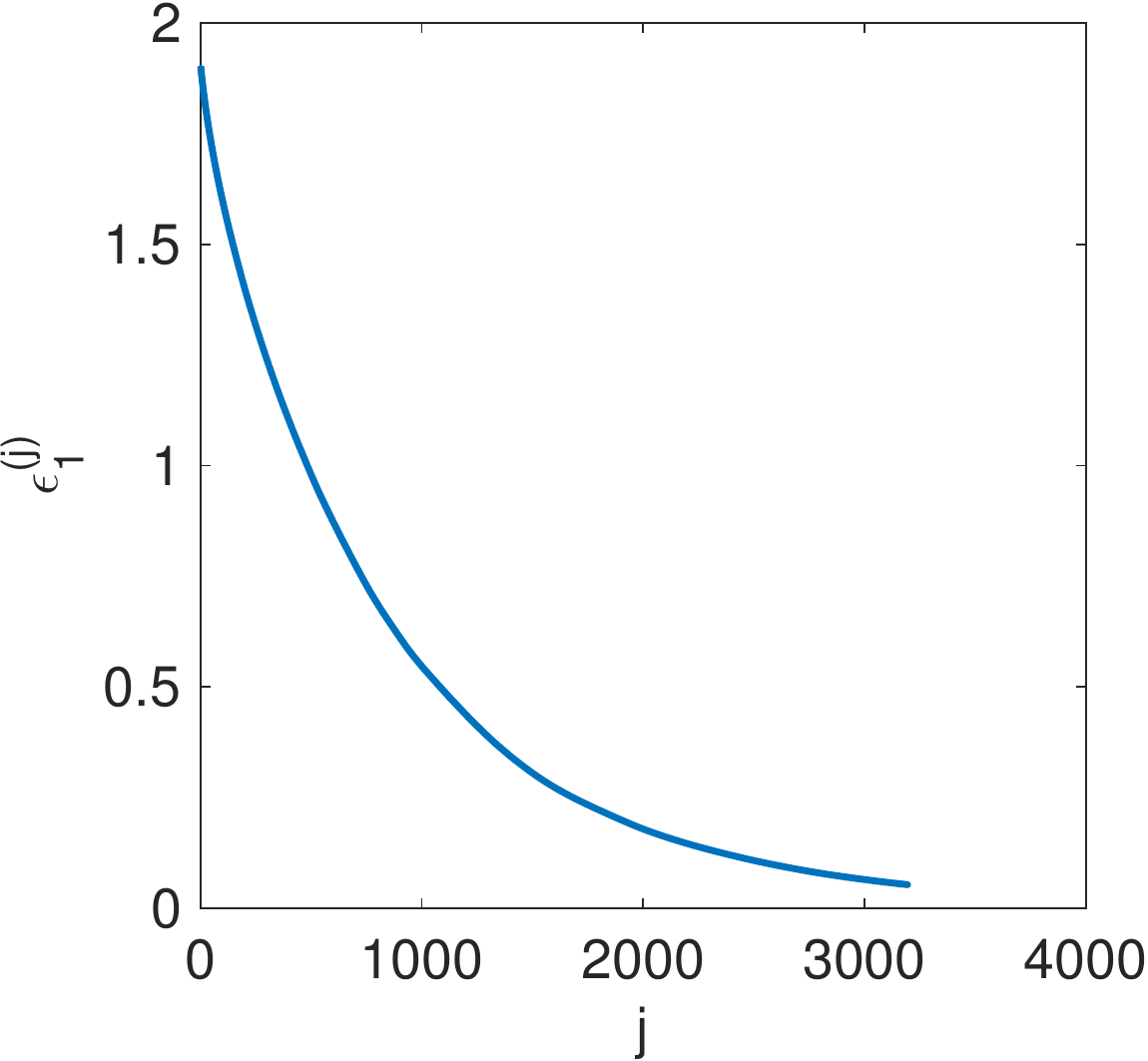}  
    \end{tabular}
  \end{center}
  \caption{Numerical results of the signal in \eqref{eqn:ex2} when $N=2$. Left: the ground truth shape function $s_1$ and its estimation by the RDBR. Middle: the ground truth shape function $s_2$ and its estimation by the RDBR. Right: the $L^2$-norm of residual $r^{(j)}$ in the $j$th iteration, i.e., $\epsilon_1$ in Algorithm \ref{alg:pcg}. } 
\label{fig:61}
\end{figure}

In all examples in this section, we consider a simple case when the signal has two components with piecewise linear and continuous generalized shapes. This makes it easier to verify the convergence analysis. For example, we consider signals of the form
\begin{equation}
\label{eqn:ex2}
f(t) = f_1(t)+f_2(t),
\end{equation}
where
\[
f_1(t) = \alpha_1(t)s_1(2\pi N\phi_1(t))= (1+0.05\sin(4\pi x))s_1\left(2\pi N(x+0.006\sin(2\pi x))\right)
\] 
and
\[
f_2(t) = \alpha_2(t)s_2(2\pi N\phi_2(t))= (1+0.05\cos(2\pi x))s_1\left(2\pi N(x+0.006\cos(2\pi x))\right),
\] 
$s_1(t)$ and $s_2(t)$ are generalized shape functions defined in $[0,1]$ as shown in Figure \ref{fig:61}. 

In the first example of this section, we show that the RDBR still converges even if the fundamental instantaneous frequencies are very low, i.e., the signal only contains a few periods of oscillation.  Figure \ref{fig:61} shows the numerical results of a signal when $N=2$ in \eqref{eqn:ex2} and $L=2^{16}$ samples on a uniform grid in $[0,1]$. This is a challenging case when there are approximately two periods in each mode. Although we cannot prove a linear convergence in this case, Figure \ref{fig:61} (right) shows that the RDBR converges with a sublinear convergence rate. With a sufficiently large iteration number, the RDBR is able to identify shape functions with a reasonably good accuracy as shown in Figure \ref{fig:61} (left and middle). The capability of handling low-frequency modes is attractive to people working on real-time data, in which the shape function might change in time and such changes are interesting to track; in this case, only the information within a few consecutive periods can be used for shape function extraction (see  \cite{ceptrum} for more references).

\begin{figure}[ht!]
  \begin{center}
    \begin{tabular}{ccc}
      \includegraphics[height=1.6in]{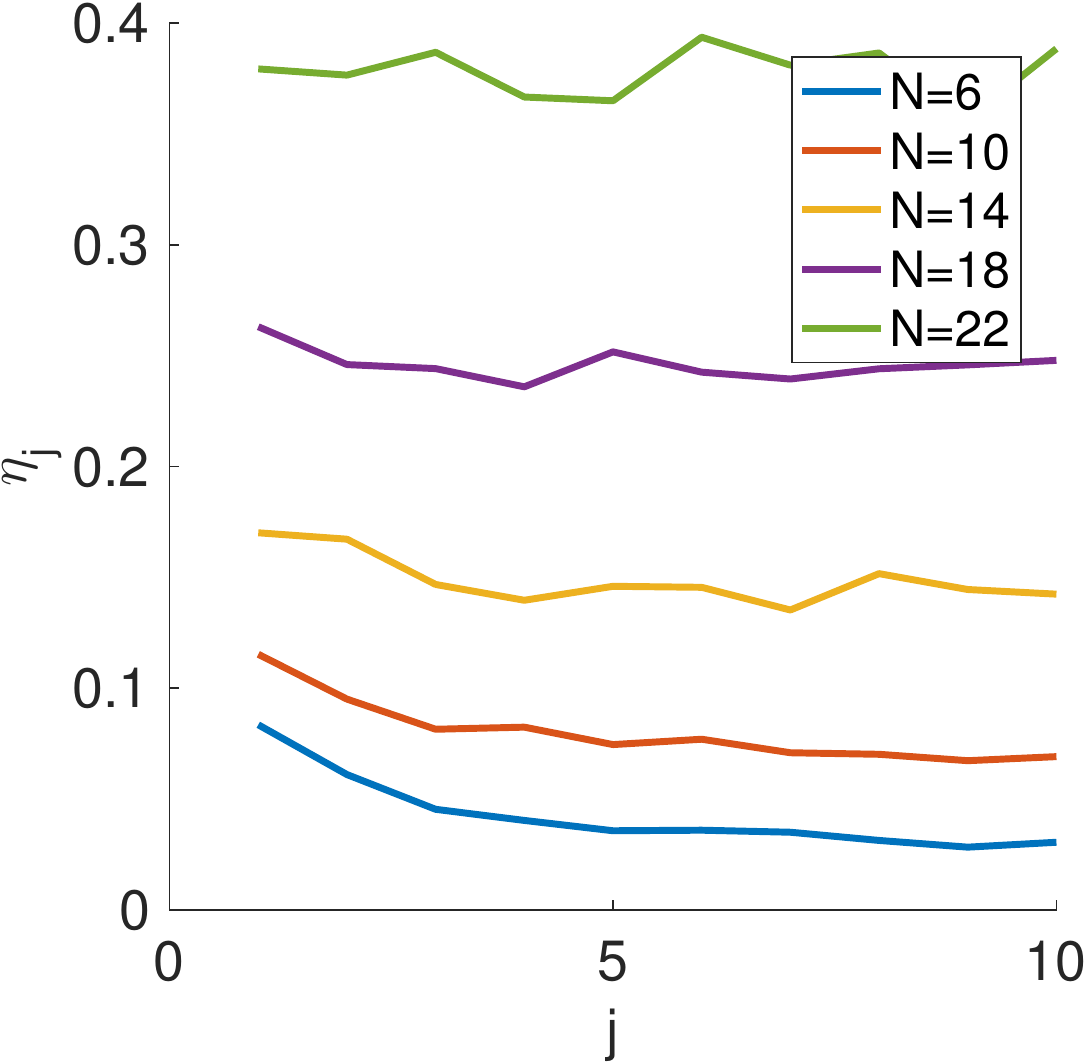}  &
      \includegraphics[height=1.6in]{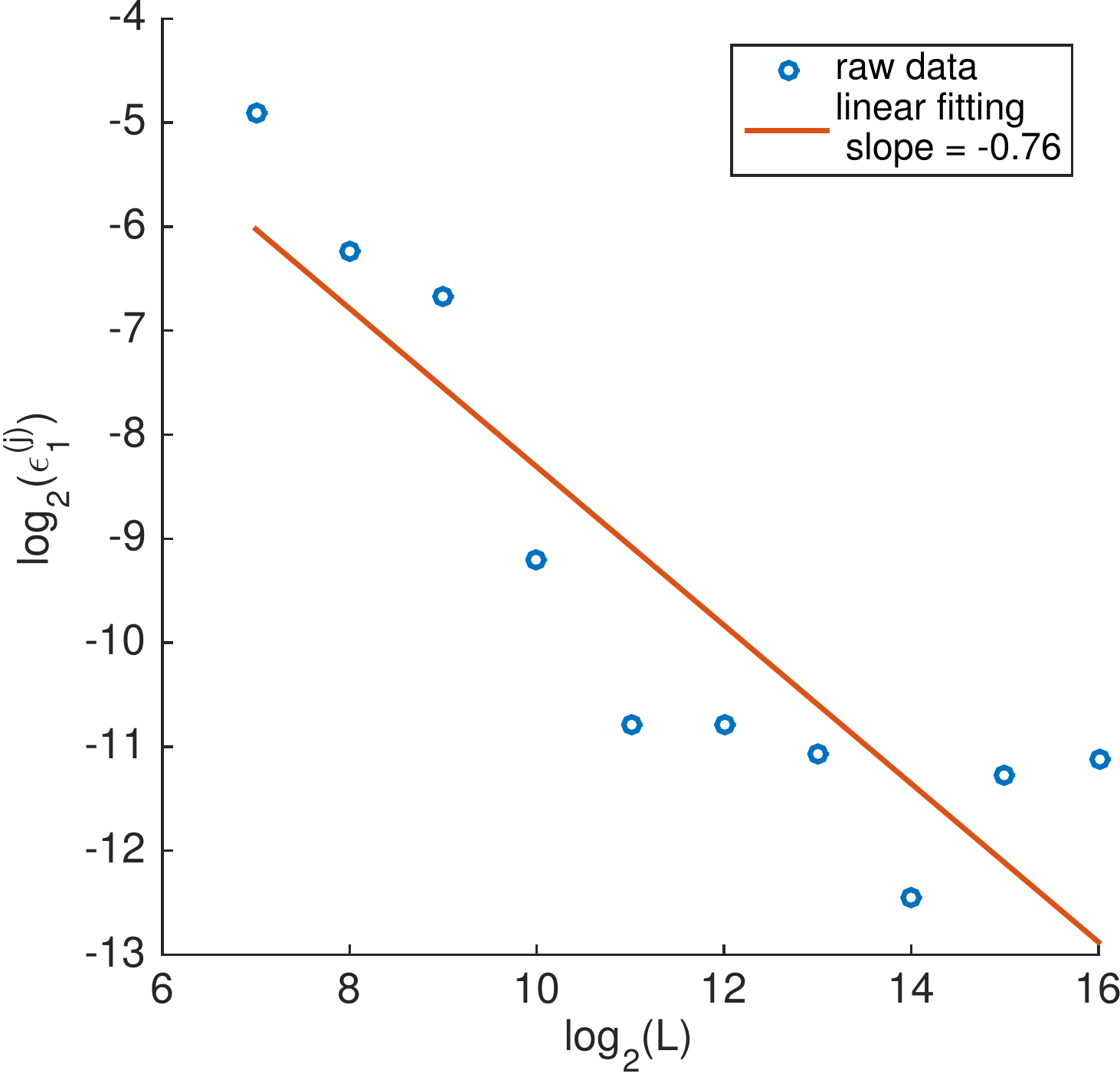}  &
      \includegraphics[height=1.6in]{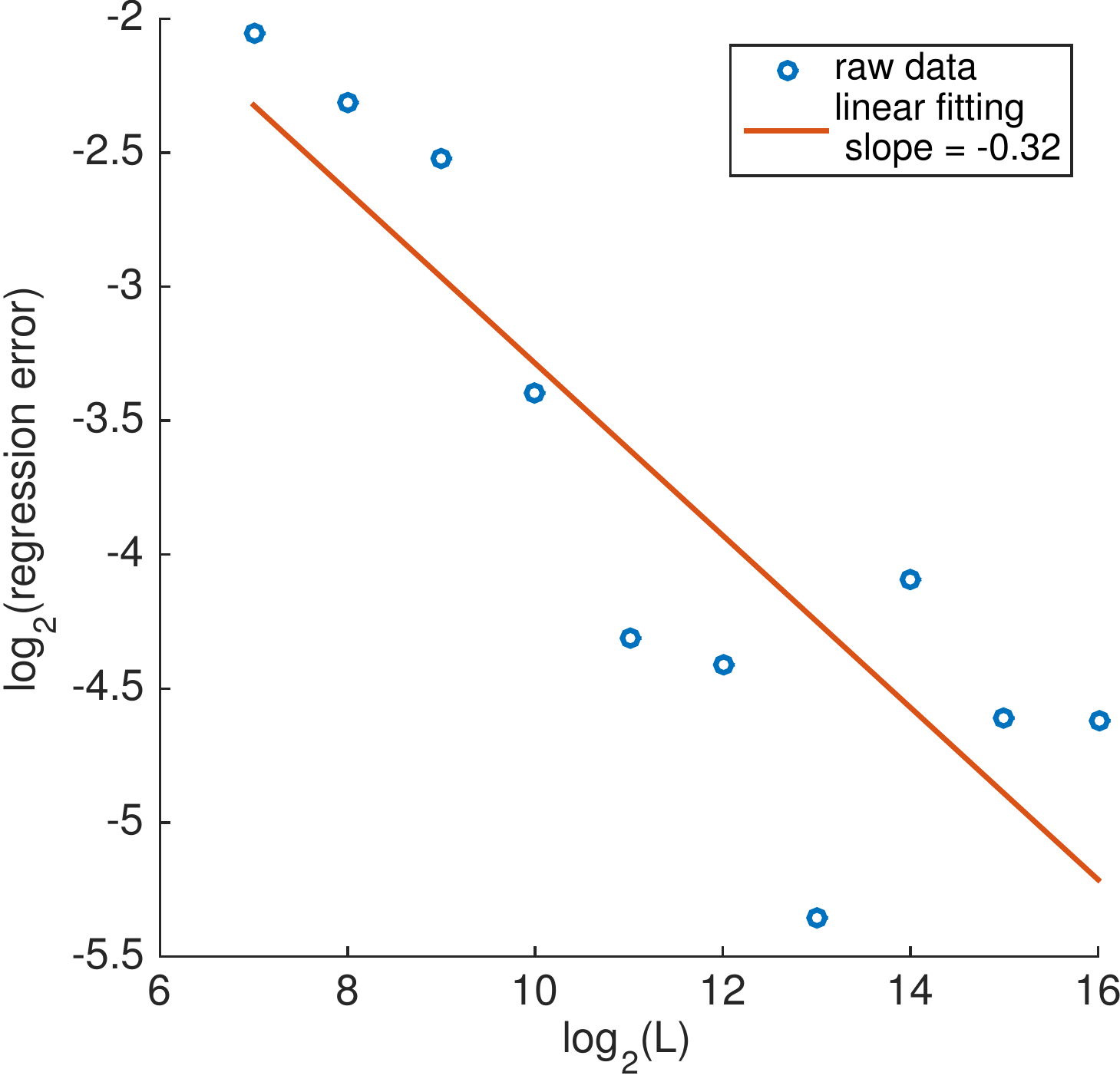}  
    \end{tabular}
  \end{center}
  \caption{Left: Estimated convergence rates $\beta$ in different iteration steps when different values are assigned to $N$ in \eqref{eqn:ex2}. Middle: the relation of the final residual norm $\epsilon_1^{(j)}$ (after the RDBR has been terminated) and the number of samples $L$. Right: the relation of the regression error in the $L^2$-norm and the number of samples $L$.}
\label{fig:7}
\end{figure} 

In the second example in this section, we fix the number of samples $L=2^{12}$, vary the parameter $N$ in \eqref{eqn:ex2}, and estimate the convergence rate numerically. By Theorem \ref{thm:conv} (adapted to the example in this section), the residual norm $\epsilon_1$ in Algorithm \ref{alg:pcg} converges to $O(\epsilon)$ as follows
\[
\epsilon_1^{(j)}=O(\epsilon)+\beta^jO(1).
\]
Hence, if we define a sequence $\{\mu_j\}$ by
\[
\mu_{j} =\log( |\epsilon_1^{(j-1)}-\epsilon_1^{(j)}|).
\]
and a sequence $\{\eta_j\}$ by
\[
\eta_j = \mu_j-\mu_{j+1},
\]
then $\eta_j$ approximately quantifies the convergence in the $j$th iteration, and should be nearly a constant close to $-\log(\beta)$. Figure \ref{fig:7} (left) visualizes the sequences $\{\eta_j\}$ generated from different signals with various $N$'s. It shows that when the fundamental frequency $N$ is sufficiently large, $\{\eta_j\}$ are approximately a constant for all $j$ and hence the convergence is linear; when $N$ is small, the RDBR converges sublinearly since $\eta_j>0$ for all $j$ and $\{\eta_j\}$ decays as $j$ becomes large. Remark that the convergence analysis is valid up to an $O(\epsilon)$ accuracy, i.e., once the residual is reduced to $O(\epsilon)$, it might not be reduced any further and, in the worst case, it might even increase again due to the numerical error in the spline regression with free knots. Hence, we only show results in the first few iterations in Figure \ref{fig:7} (left) to verify the convergence rate. Actually in the next example will illustrate the effect of the error in the spline regression on the accuracy of the RDBR.

In the last example of this section, we fixed $N=100$, only vary the number of samples $L=2^m$ with $m=7,8,\dots,12$, and compare the accuracy of the RDBR and the spline regression with free knots. To obtain results with an accuracy as high as possible, we let $maxIter=200$ and $\epsilon=1e-13$. Figure \ref{fig:7} (middle) shows that the final residual norm $\epsilon_1$ after the RDBR essentially decays in $L$ with the exception of the two largest values of $L$.

To understand the two exceptions observed in Figure \ref{fig:7} (middle), let us check the effect of the error in the spline regression on the accuracy of the RDBR. Recall that the final residual norm after the RDBR depends on the accuracy of the regression in Lemma \ref{lem:SP} by the analysis in Theorem \ref{thm:conv}. Although Theorem \ref{thm:reg} considers only the partition-based regression, a similar conclusion holds for the spline regression (see Chapter 14 in \cite{regressionBook}). Hence, if the number of samples $L$ increases to infinity, the regression error is reduced to a small constant depends on other parameters, i.e., the solution of the regression problem cannot be improved any more by increasing $L$. Therefore, the accuracy of the RDBR might not be improved by increasing $L$ if $L$ has been large enough. This explains the two exceptions in Figure \ref{fig:7} (middle).

To further verify the explanation in the last paragraph, we check the accuracy of the spline regression when $L$ is increased. Recall that, in each iteration of the RDBR, the estimation of each shape function is perturbed by other modes. For the example considered here, in the first iteration, when we try to estimate $s_1$ by regression, another mode $f_2$ acts like a noise perturbation with a nonzero mean and a bounded variance determined by the amplitude function and the shape function in $f_2$. By the formula of $f_2$, the largest amount perturbed is approximately $0.5$. Hence, we use a toy example
\[
Y=s_1(2\pi X)+ns,
\]
where  $X$ is a random variable with a uniform distribution in $[0,1]$, and $ns$ is a random variable with a zero mean and a uniform distribution in $[-0.5,0.5]$. $L$ samples of the random vector $(X,Y)$ are generated independently. The spline regression with free knots is applied to estimate $s_1(2\pi x)$ from these samples and its $L^2$ regression error defined in Theorem \ref{thm:reg} is recorded. Figure \ref{fig:7} (right) shows the regression errors with different $L$'s. As we can see, the regression error cannot be further reduced by increasing $L$ once $L$ has reached almost the same (large) critical value as in Figure \ref{fig:7} (middle). This agrees with the explanation in the last paragraph. Remark that the accuracy in Figure \ref{fig:7} (middle) is much higher than the one in Figure \ref{fig:7} (right), and the decay rate is larger as well, indicating the possibility that the accuracy of the recursive scheme in the RDBR exceeds that of a single regression.

\subsection{Synthetic examples}
In this section, we apply the RDBR to examples with known instantaneous properties in various generalized mode decomposition problems. To make it easier to compare the RDBR with other methods, we follow the examples in the paper of the diffeomorphism-based spectral analysis (DSA) method in \cite{1DSSWPT}, since \cite{1DSSWPT} provides various examples with different shape functions. We will apply the RDBR to the examples in Figure 16, 17, and 18 in \cite{1DSSWPT}, without replicating the results of the DSA here. For more details about setting up these examples, the reader is referred to \cite{1DSSWPT}. As for the parameters in the RDBR, please see Table \ref{tab:1}.

The first example in this section, corresponding to the example in Figure 17 and 18 in \cite{1DSSWPT}, contains two generalized modes generated with two ECG shape functions shown in Figure \ref{fig:11}. Figure \ref{fig:11} (left two graphs) shows the recovered shape functions, as compared with the ground truth shape functions, when the synthetic data is clean. Figure \ref{fig:11} (right two graphs) shows the recovered shape functions when the signal-to-noise ratio ($\SNR$ as defined in Equation \eqref{eqn:SNR}) is $-3\text{ dB}$. Note that the $\SNR$ in the example in Figure 17 and 18 in \cite{1DSSWPT} is $0\text{ dB}$, which is larger than the one in our example. Comparing Figure 17 in \cite{1DSSWPT} and Figure \ref{fig:11}, we see that the RDBR is more accurate than the DSA.

\begin{figure}[ht!]
  \begin{center}
    \begin{tabular}{c}
      \includegraphics[height=1.3in]{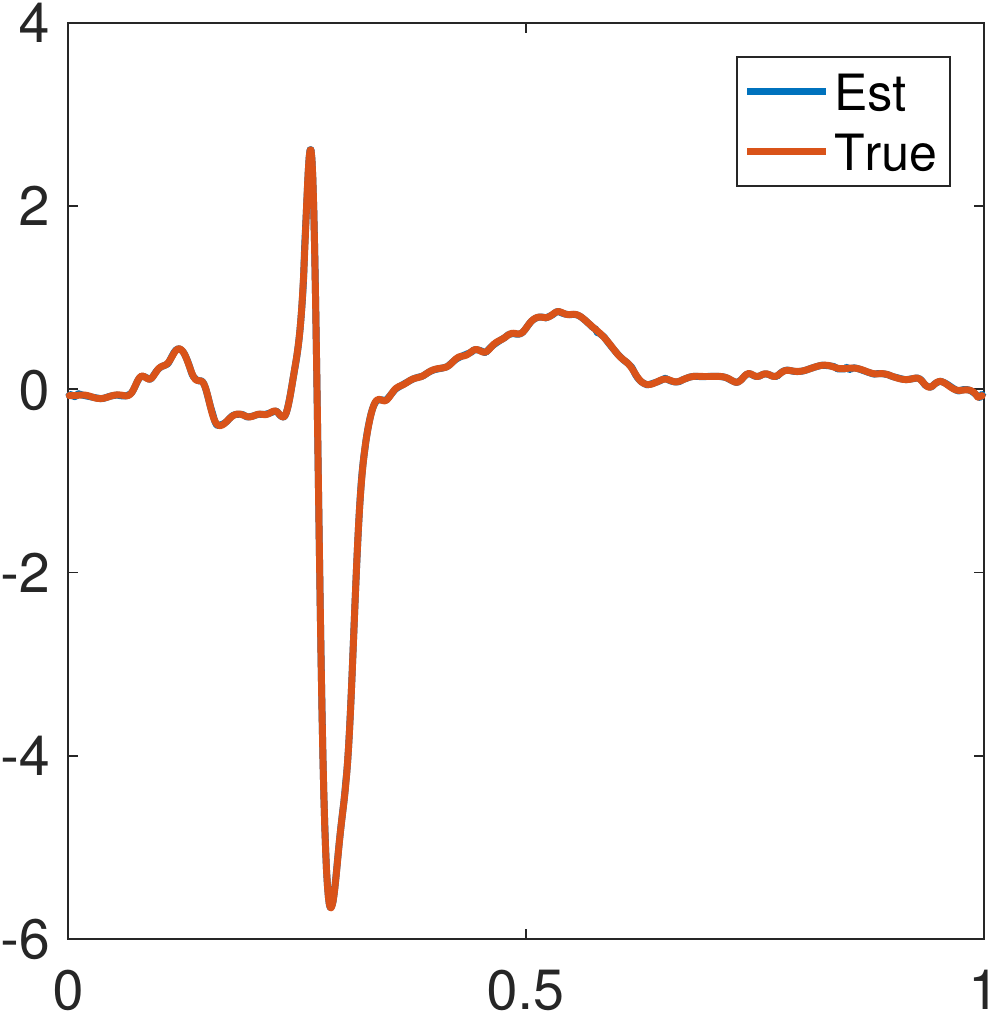}  
   \includegraphics[height=1.3in]{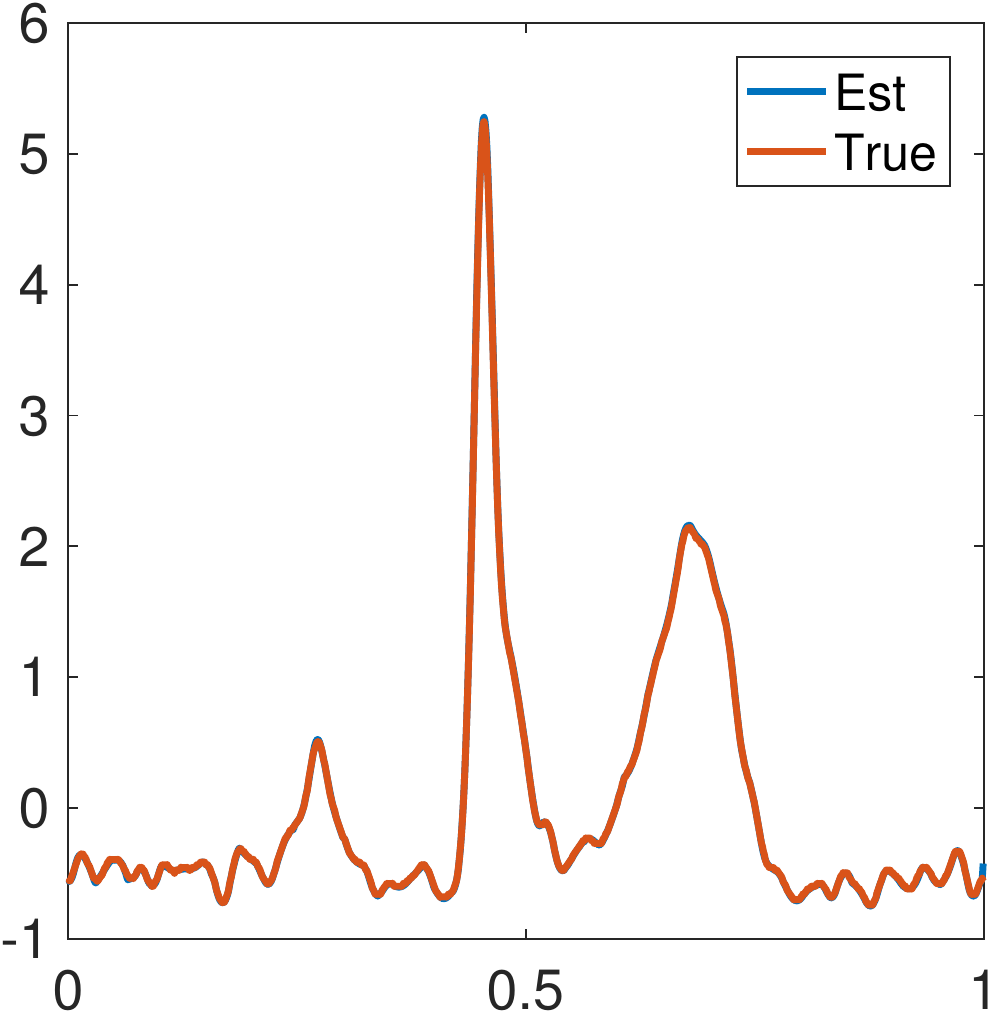}  \hspace{1cm}
   \includegraphics[height=1.3in]{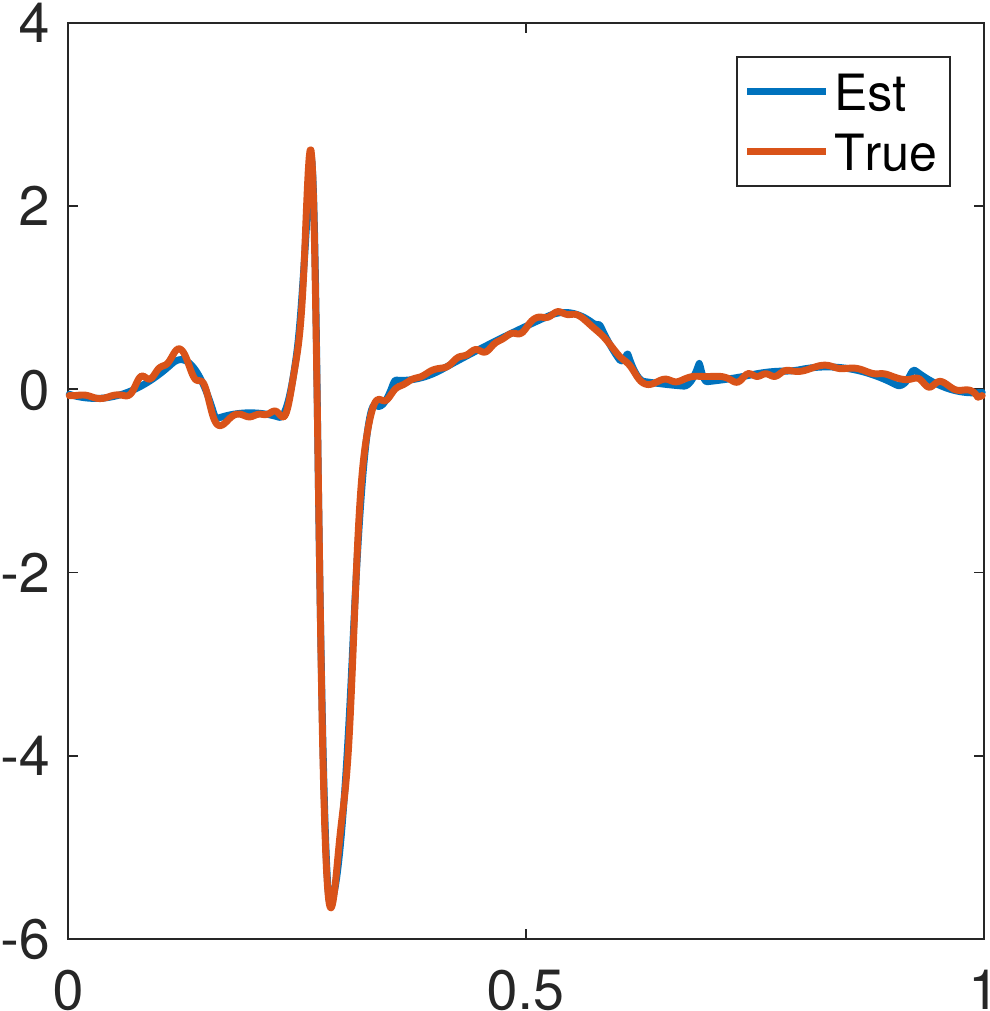}  
      \includegraphics[height=1.3in]{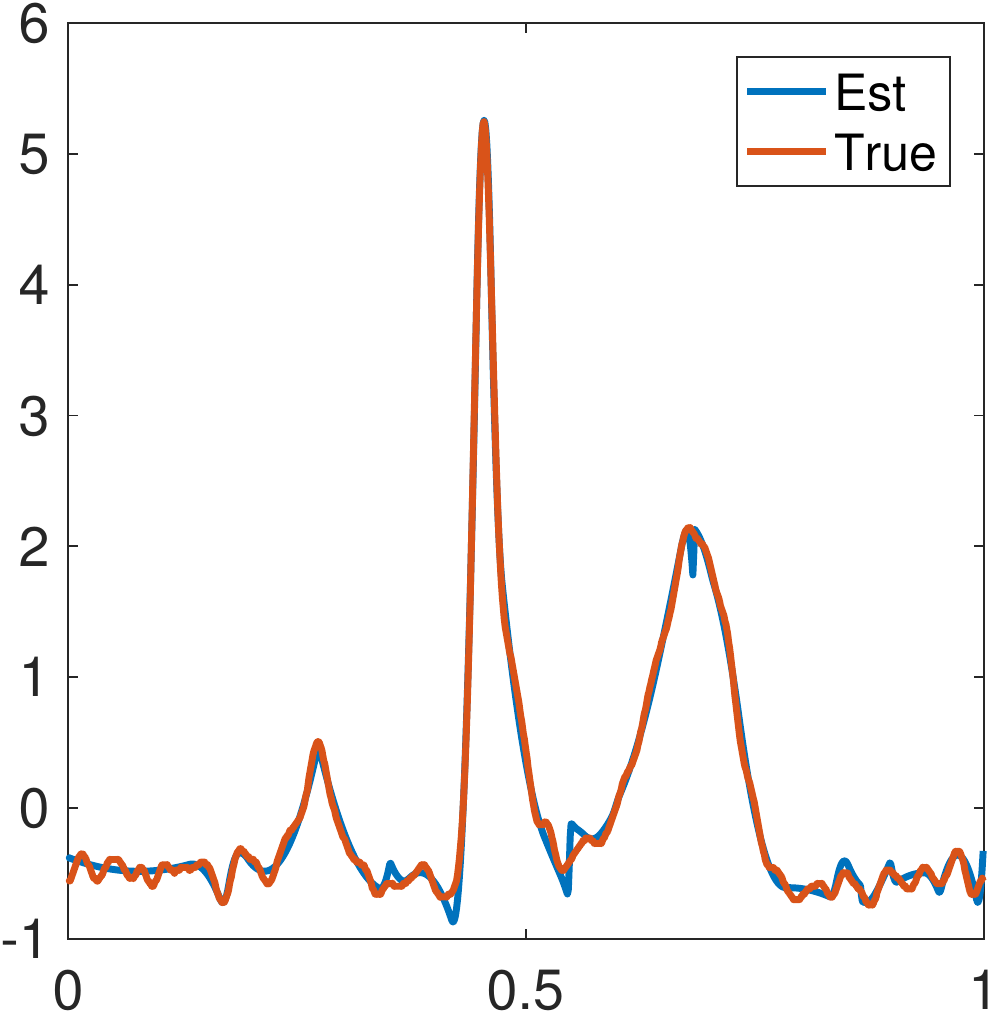}  
    \end{tabular}
  \end{center}
  \caption{Recovered shapes in clean (left) and noisy (right) examples, as compared with ground truth shapes.}
\label{fig:11}
\end{figure}

The second example in this section, corresponding to the example in Figure 16 in \cite{1DSSWPT}, contains two generalized modes generated with two piecewise constant shape functions, which is shown in Figure \ref{fig:12}. Figure \ref{fig:12} (left two graphs) shows the recovered shape functions, as compared with the ground truth shape functions, when the synthetic data is clean. Figure \ref{fig:12} (right two graphs) shows the recovered shape functions when $\SNR=-3\text{ dB}$. Again, the $\SNR$ in the example in Figure 16 in \cite{1DSSWPT} is $0\text{ dB}$, meaning that our example here is noisier than that in \cite{1DSSWPT}, and we obtained better results, even though the results in \cite{1DSSWPT} are improved by additional TV-norm minimization to remove noise and the Gibbs phenomenon around discontinuous points in shape functions. Comparing Figure 16 in \cite{1DSSWPT} and Figure \ref{fig:12}, we see that the RDBR, even without post-processing, is competitive with the DSA.

\begin{figure}[ht!]
  \begin{center}
    \begin{tabular}{c}
      \includegraphics[height=1.3in]{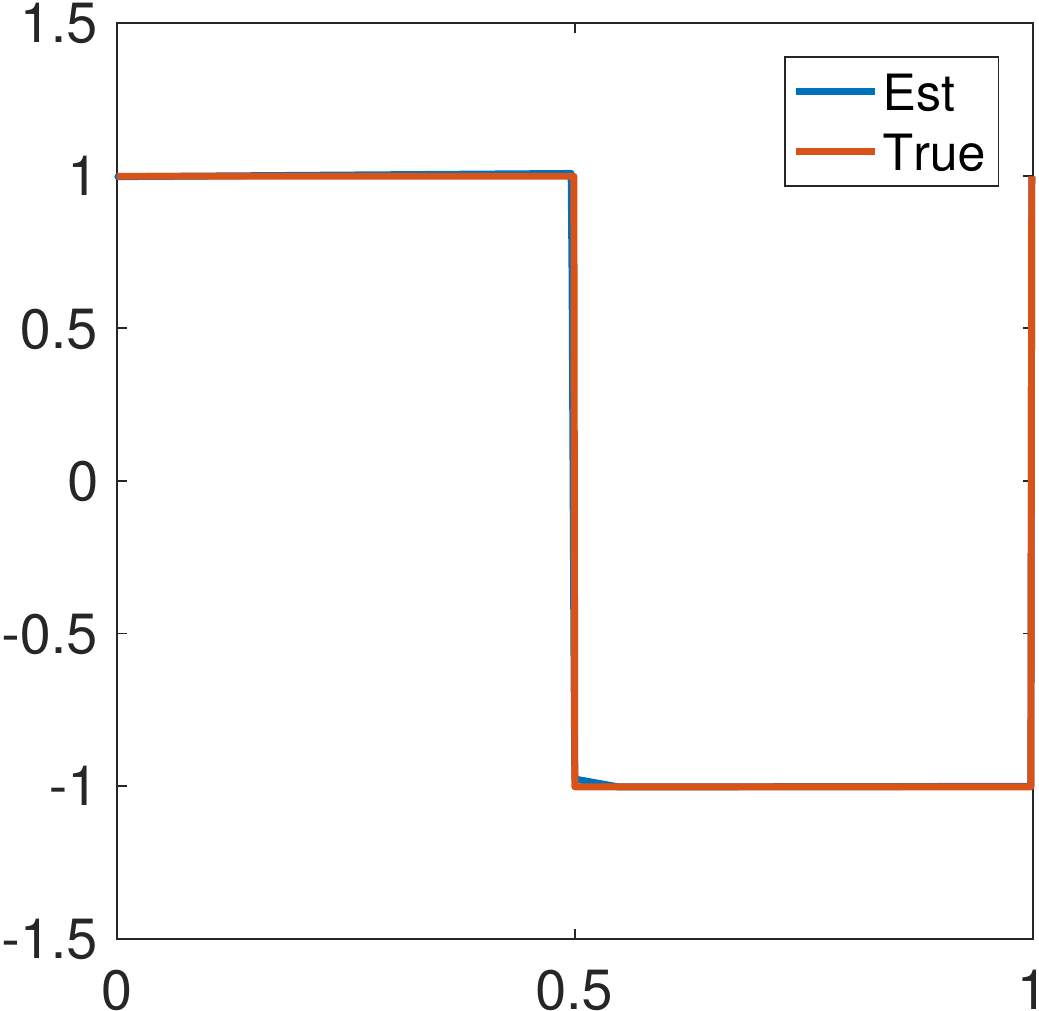}  
   \includegraphics[height=1.3in]{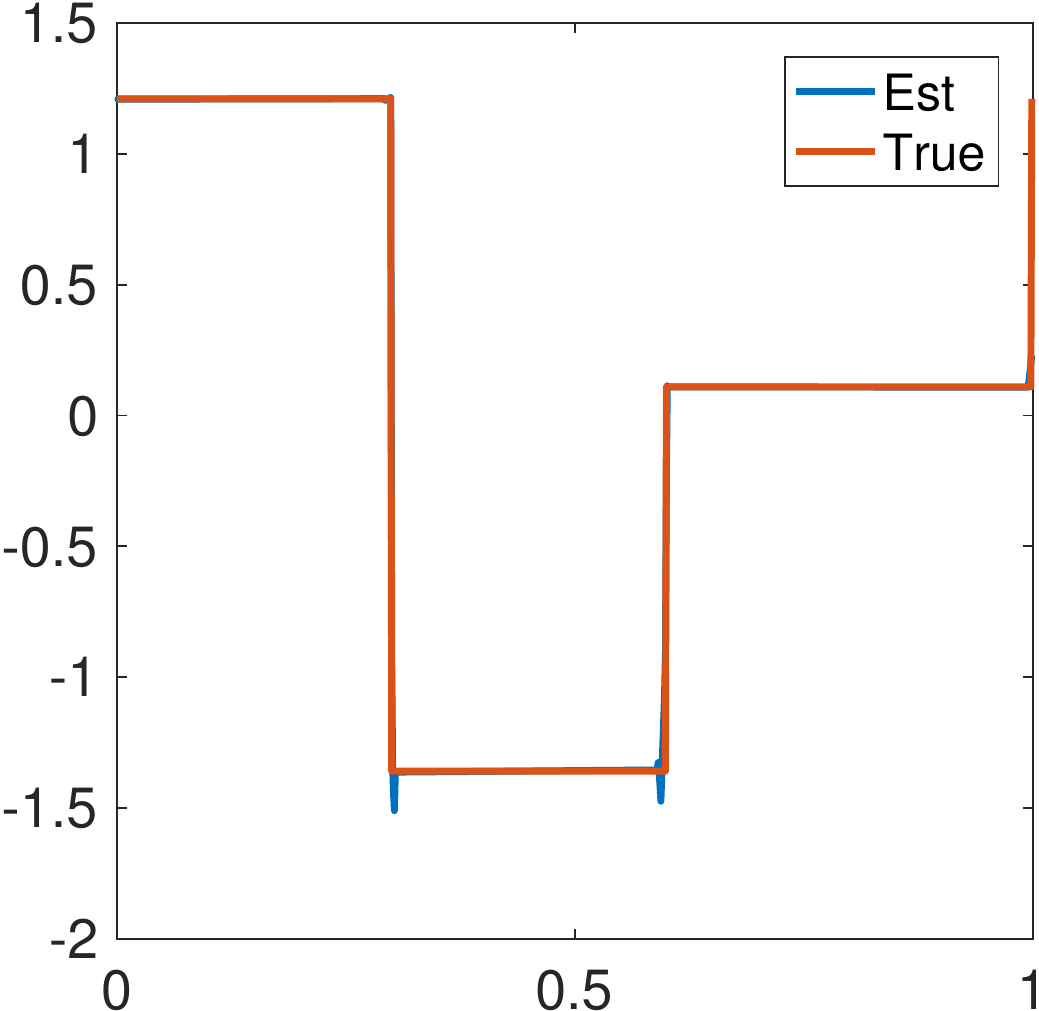}    \hspace{1cm}
   \includegraphics[height=1.3in]{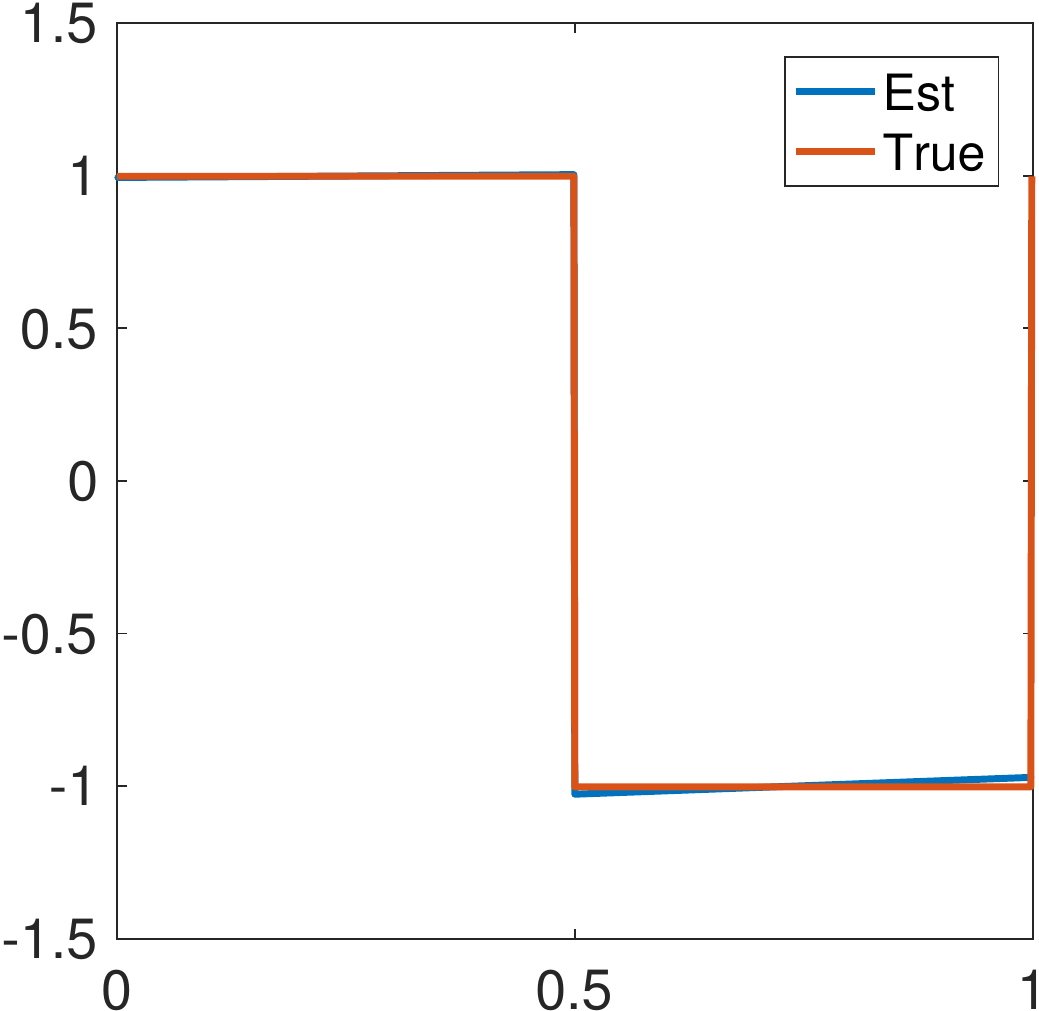}  
      \includegraphics[height=1.3in]{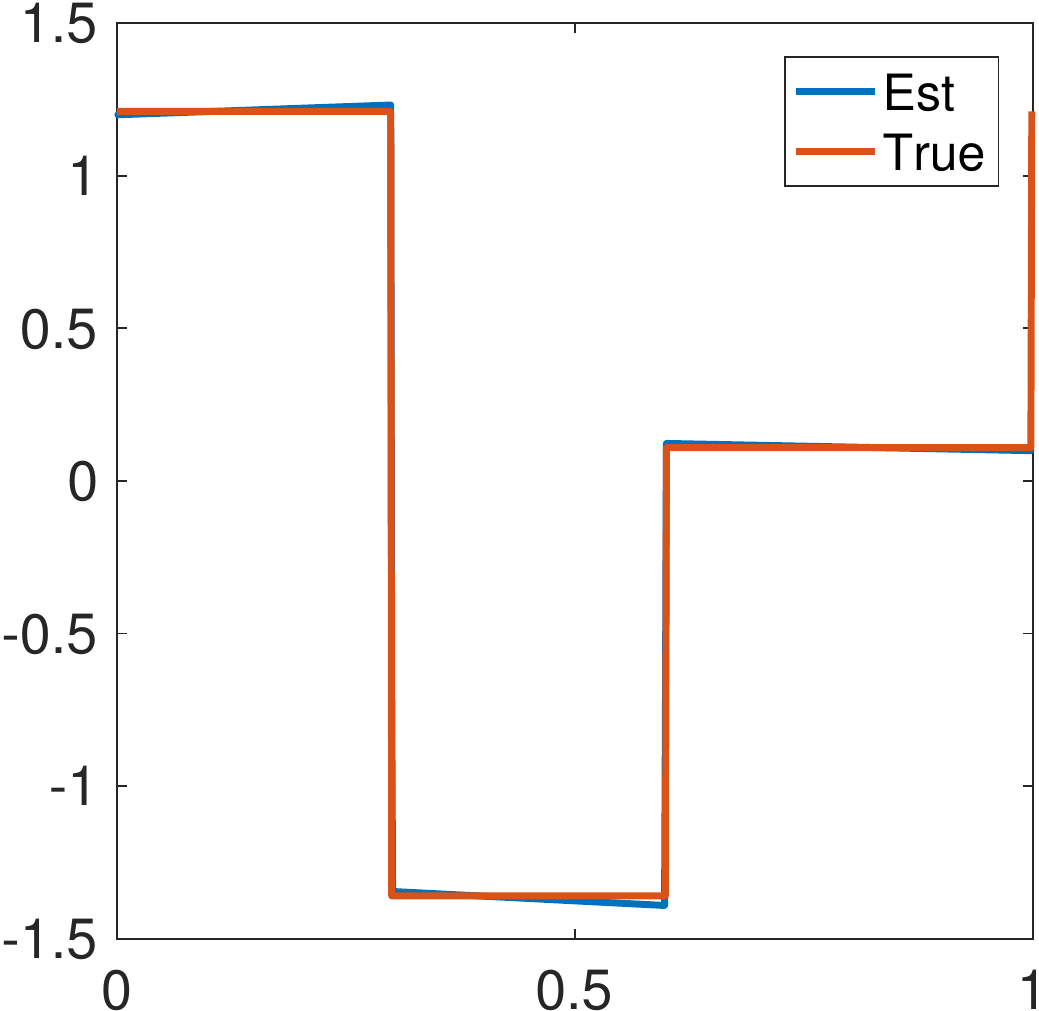}  
    \end{tabular}
  \end{center}
  \caption{Recovered shapes in clean (left) and noisy (right) examples, as compared with ground truth shapes.}
\label{fig:12}
\end{figure}

In the last example of this section, we apply the RDBR to a very challenging case, in which the signal $f(t)$ contains four modes with close instantaneous frequencies (see Figure \ref{fig:13-2}) given below.
\begin{equation}
\label{eqn:ex3}
f(t) = f_1(t)+f_2(t)+f_3(t)+f_4(t),
\end{equation}
where
\[
f_k(t)=s_k(2\pi N\phi_k(t)),
\]
\[
\phi_k(t) = t + 0.05(k-1) + 0.01 \sin(2\pi(t+0.05(k-1))),
\]
for $k=1,\dots,4$, $N=200$, and $\{s_k(t)\}_{k=1,\dots,4}$ are visualized in Figure \ref{fig:13}. As shown in Figure \ref{fig:13}, the RDBR is able to estimate shape functions precisely from clean data. In the case of very noisy data when $\SNR=-3\text{ dB}$, even if the instantaneous frequencies are very close, the RDBR is still able to recover shape functions with a reasonably good accuracy.

\begin{figure}[ht!]
  \begin{center}
    \begin{tabular}{c}
      \includegraphics[height=2in]{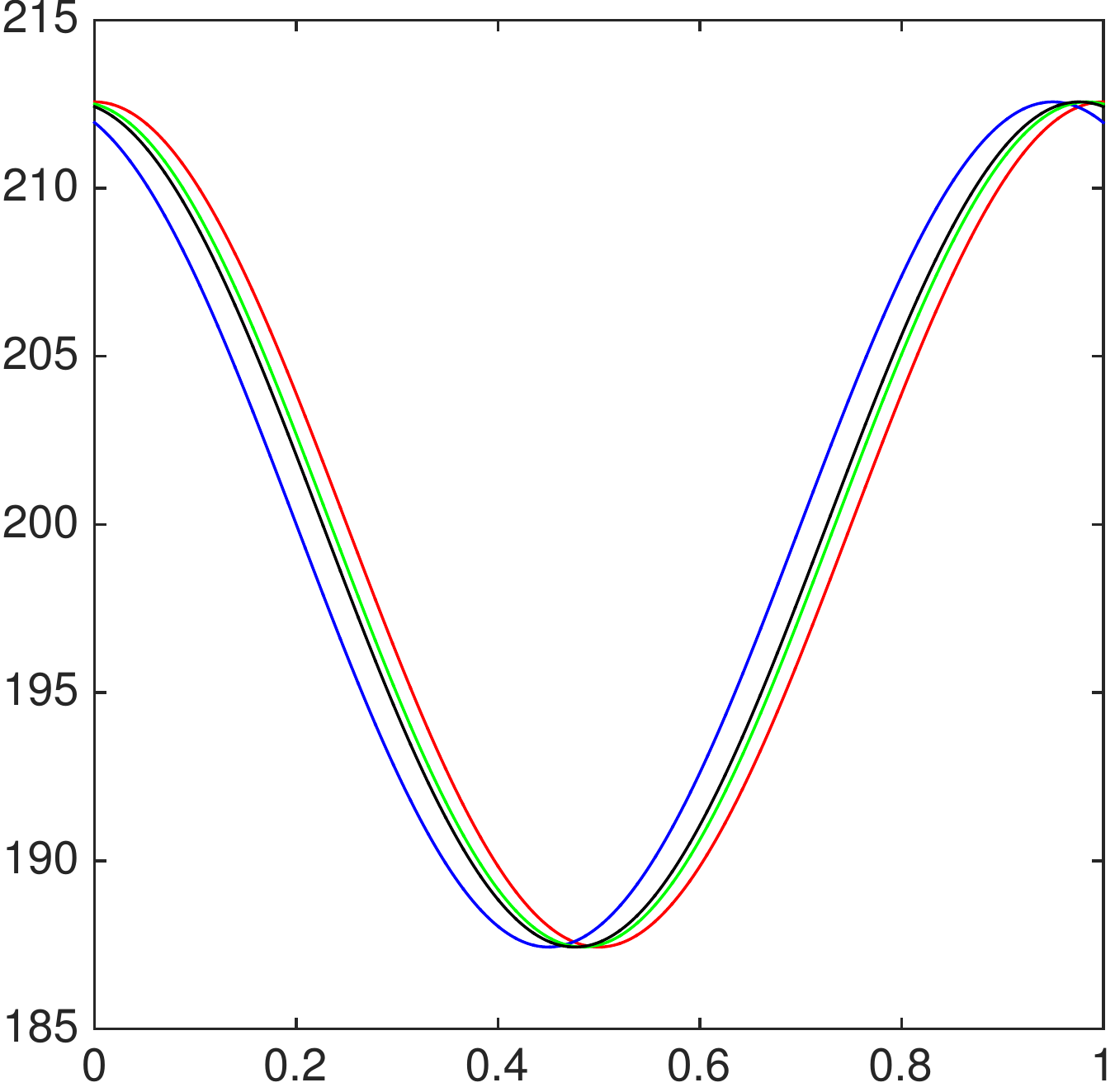}  
    \end{tabular}
  \end{center}
  \caption{Instantaneous frequencies of the example in \eqref{eqn:ex3}.}
\label{fig:13-2}
\end{figure}

\begin{figure}[ht!]
  \begin{center}
    \begin{tabular}{c}
      \includegraphics[height=1.3in]{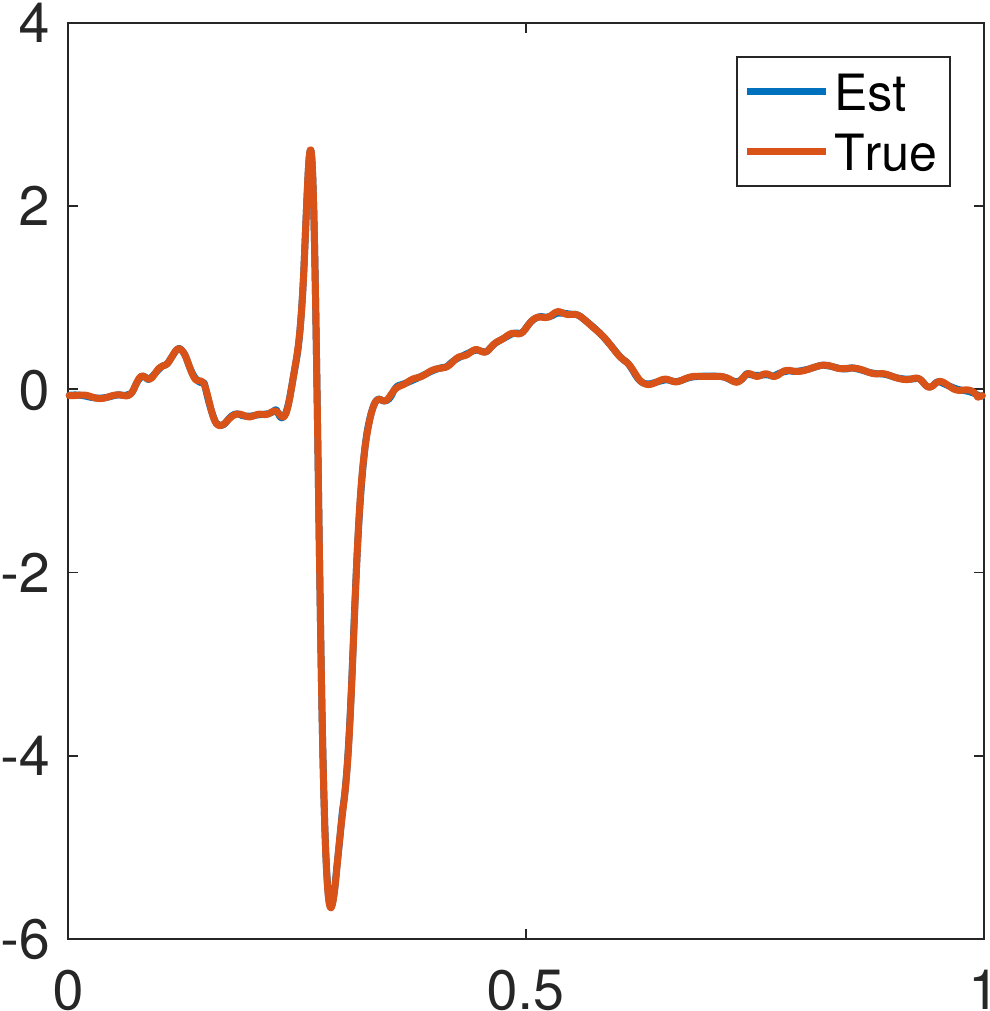}  
   \includegraphics[height=1.3in]{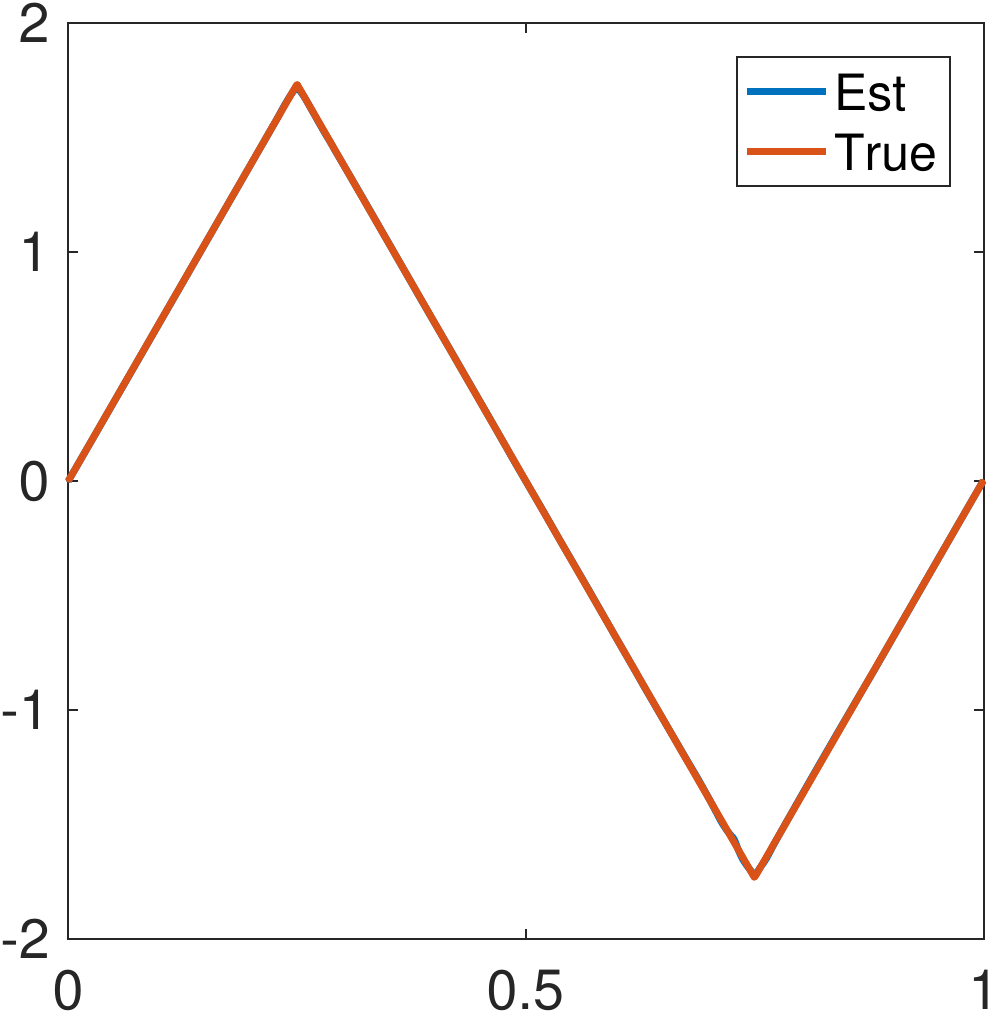}    \hspace{1cm}
   \includegraphics[height=1.3in]{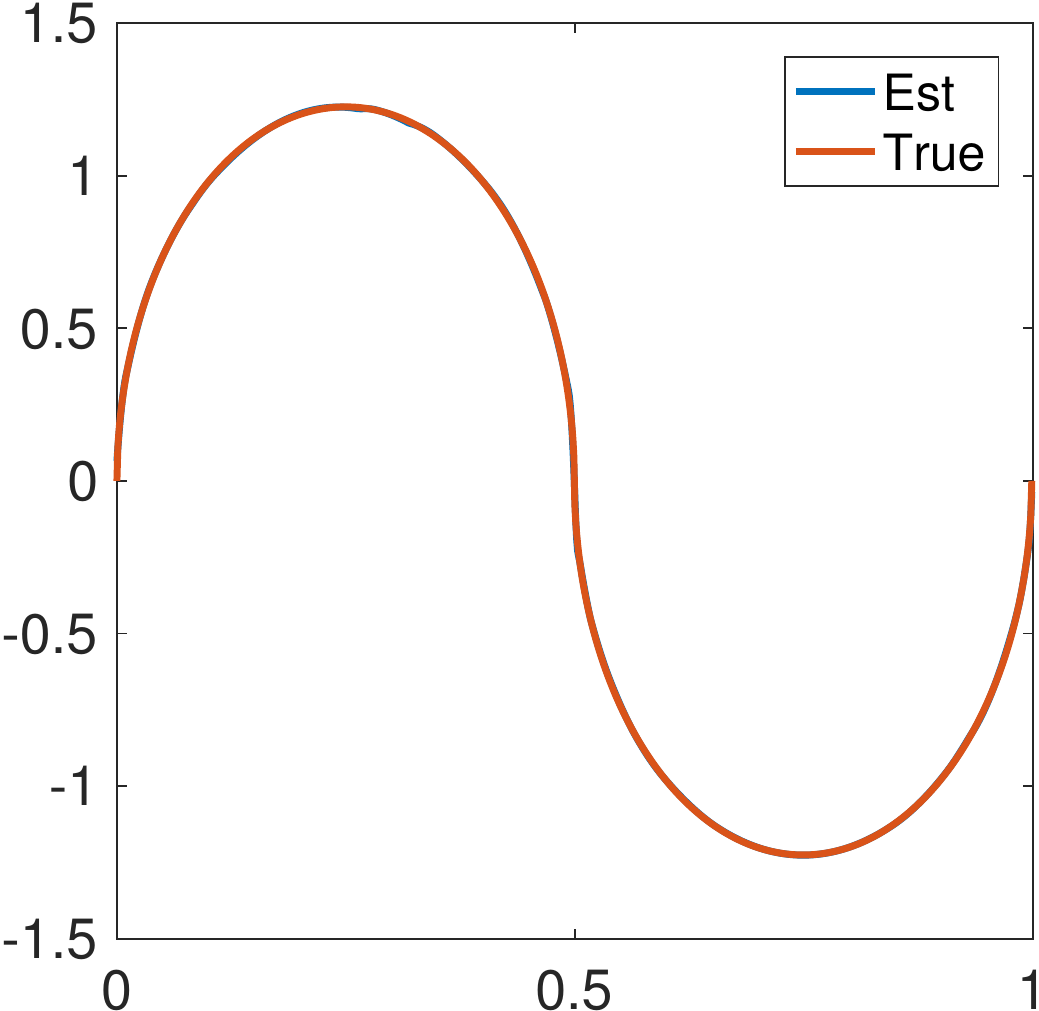}  
      \includegraphics[height=1.3in]{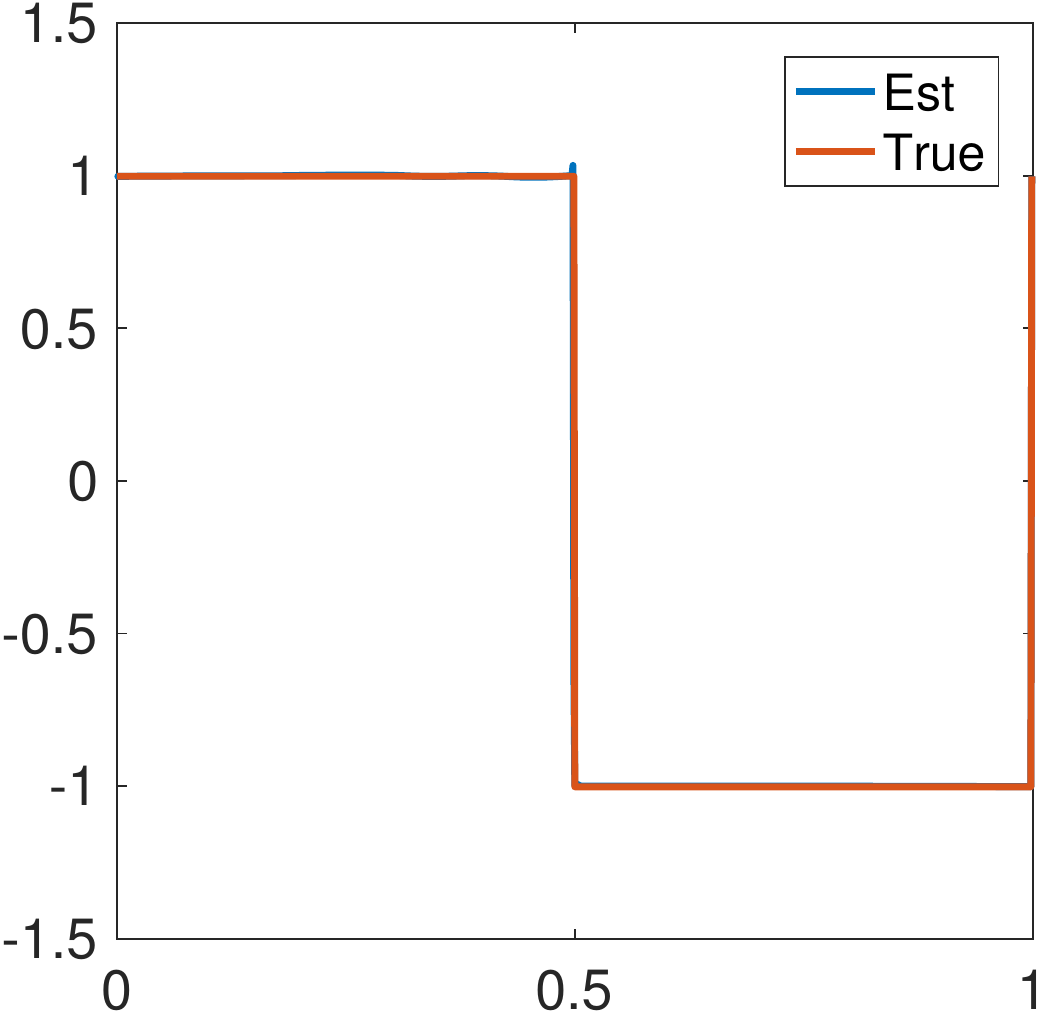}  \\
      \includegraphics[height=1.3in]{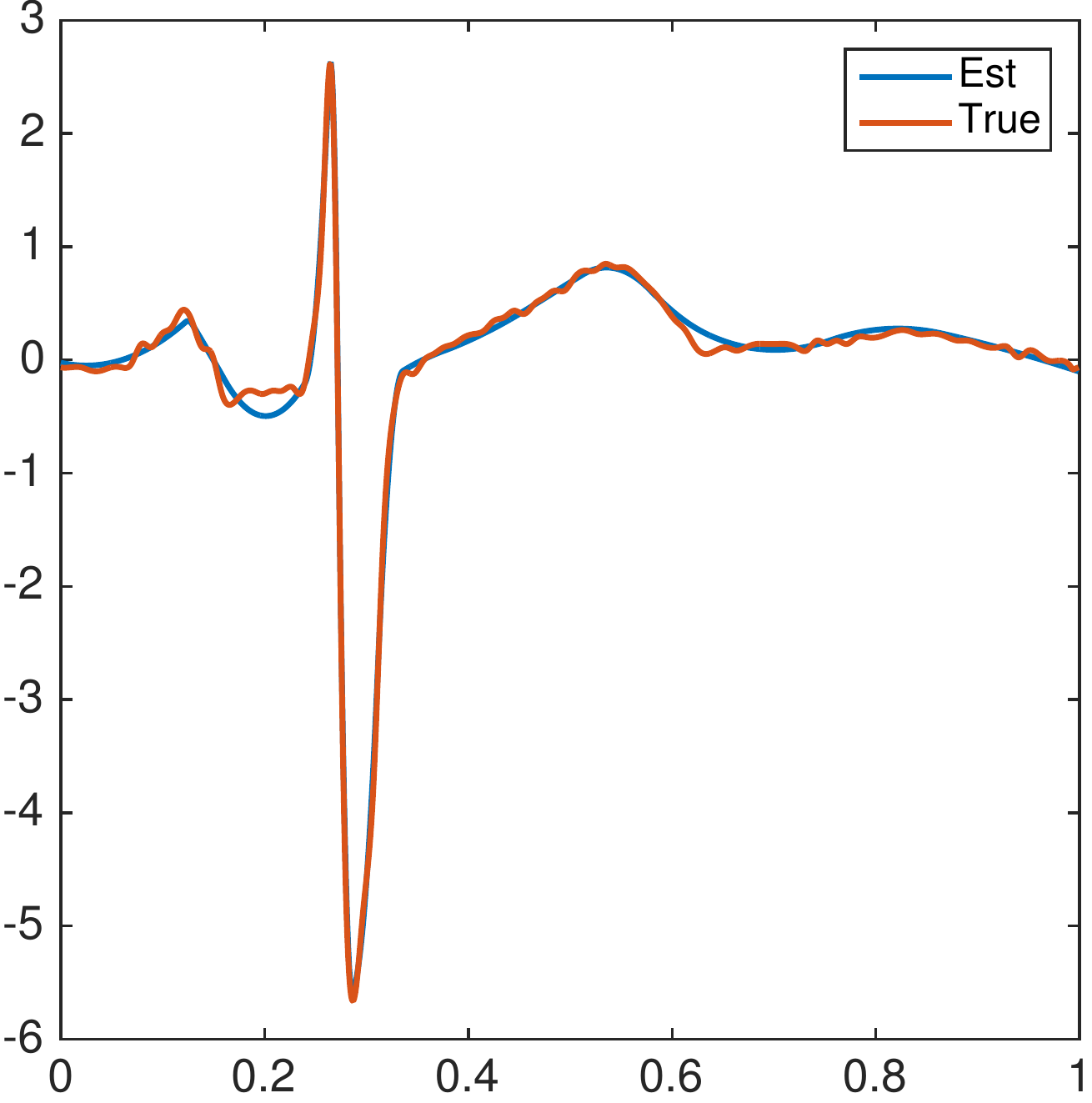}  
   \includegraphics[height=1.3in]{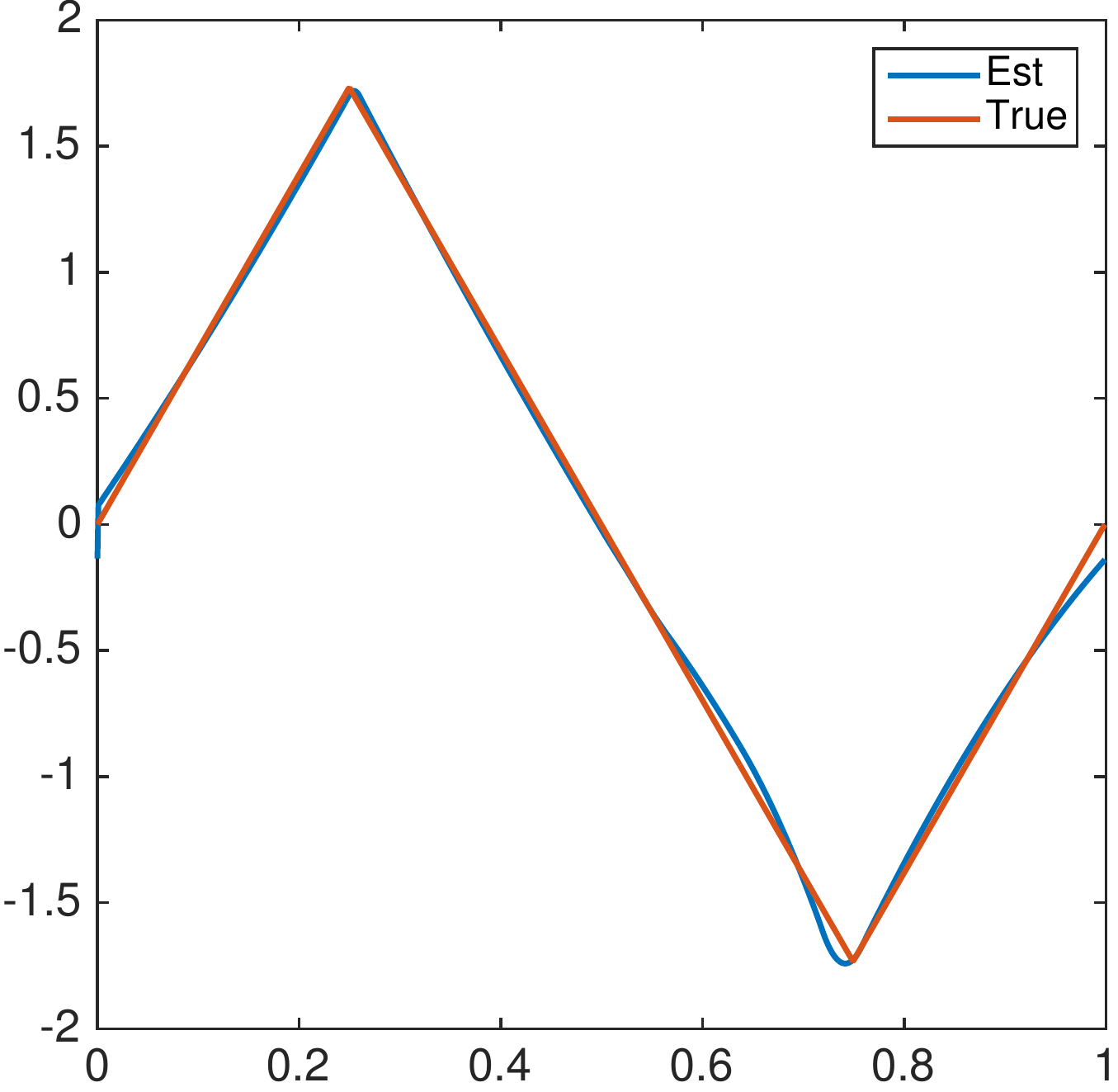}    \hspace{1cm}
   \includegraphics[height=1.3in]{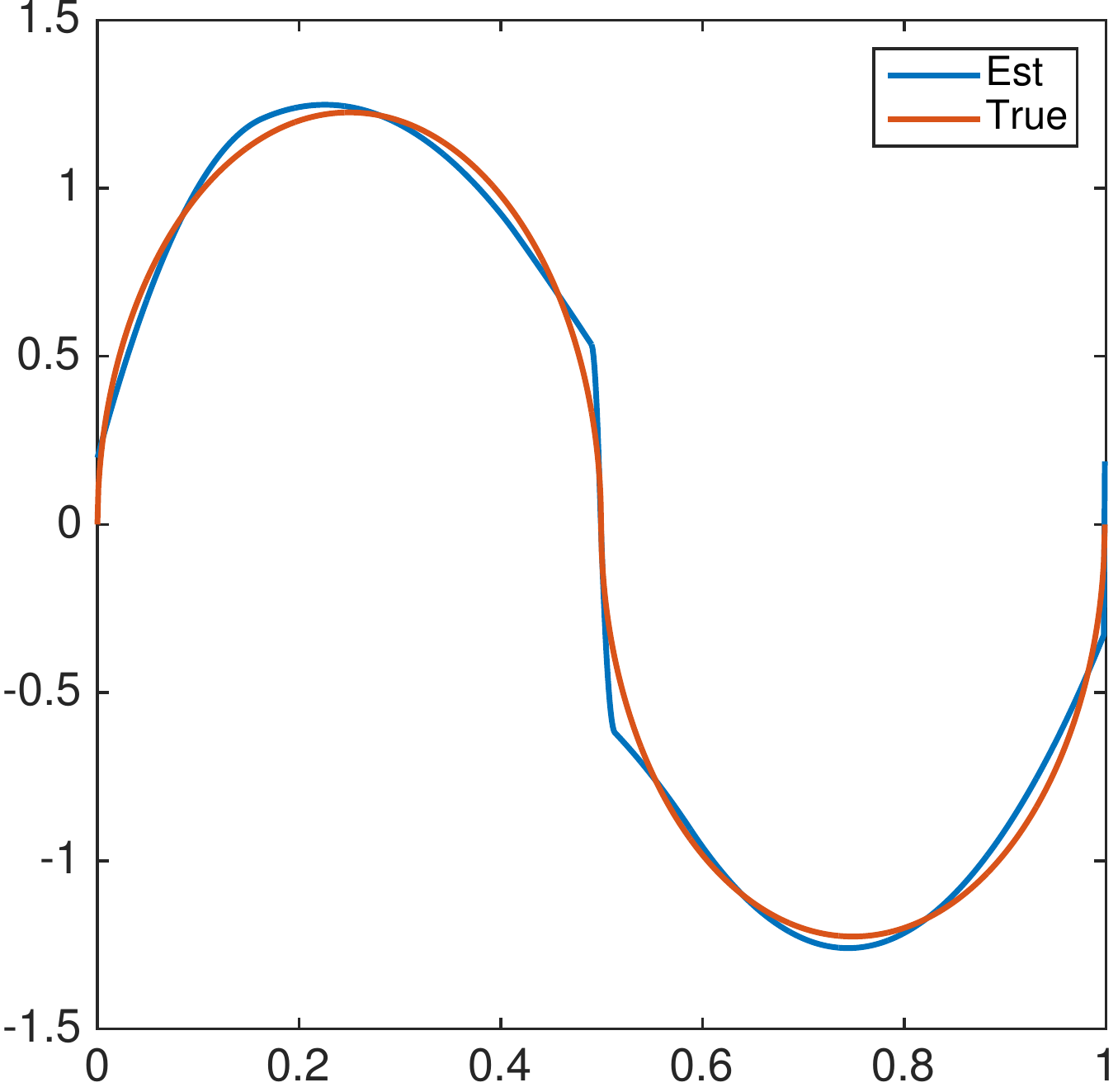}  
      \includegraphics[height=1.3in]{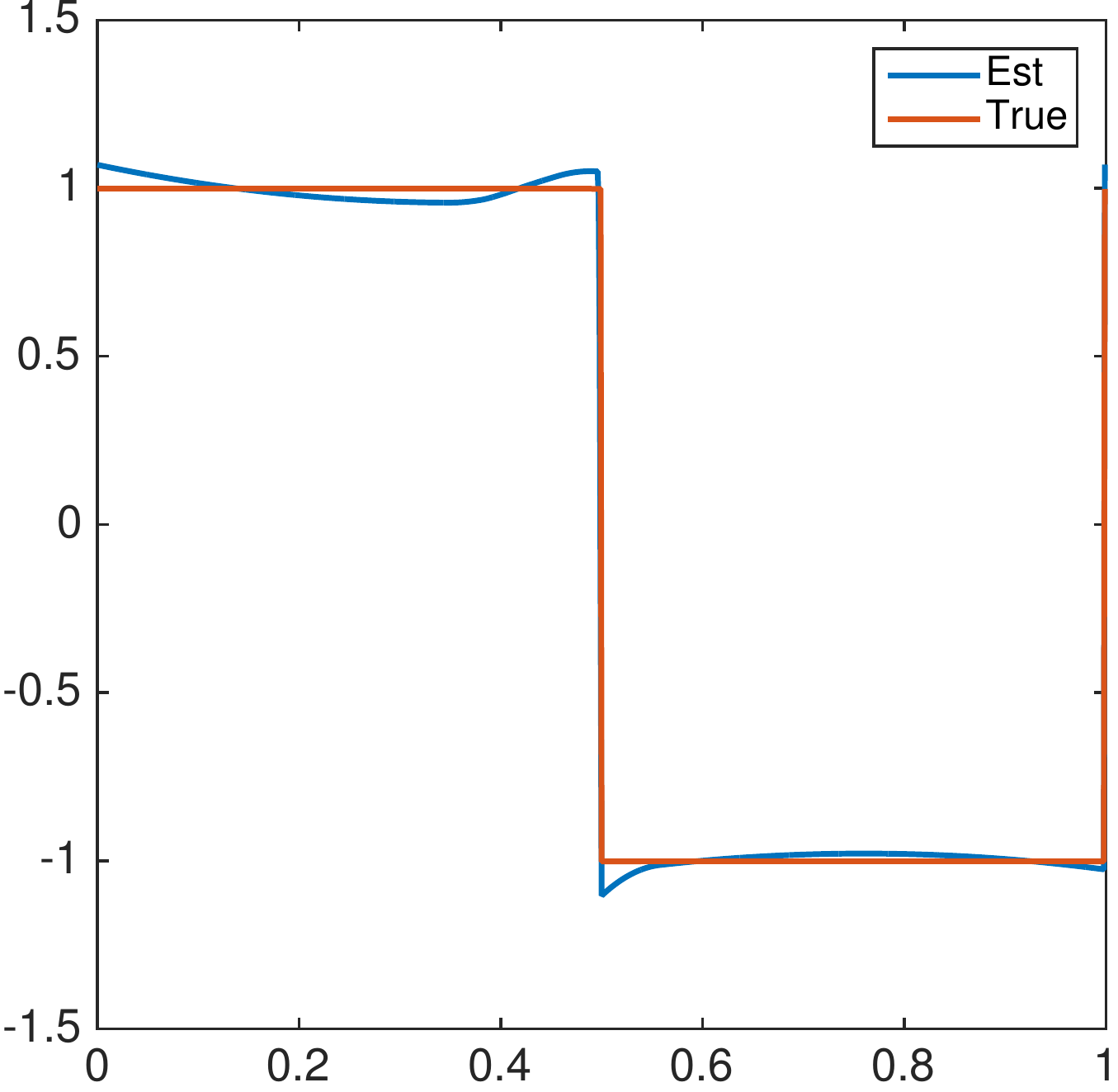}  
    \end{tabular}
  \end{center}
  \caption{From left to right, the ground truth shape functions $s_k(2\pi t)$ for $k=1,\dots,4$ (in red) and their estimations by the RDBR (in blue). From top to buttom, the results recovered from a clean signal and a noisy signal with $\SNR=-3\text{ dB}$.}
\label{fig:13}
\end{figure}

\subsection{Practical applications}

In this section, we provide two examples to demonstrate the capability of the RDBR in practical applications when instantaneous properties are not known. The synchrosqueezed transform as implemented in \cite{1DSSWPT} is applied to estimate these properties as inputs of the RDBR. These inputs may contain systematic error due the synchrosqueezed transform, but the RDBR is still able to estimate the shape functions precisely. The first example is a generalized mode decomposition similar to the example of Figure \ref{fig:11} that used ECG shape functions. To make the problem more challenging, instantaneous frequencies are much smaller and more similar in this example. Let us define
\begin{eqnarray*}
f(t)=\alpha_1(t)s_1(2\pi N_1\phi_1(t))+\alpha_2(t)s_2(2\pi N_2\phi_2(t)),
\end{eqnarray*}
where $\alpha_1(t) = 1+0.05\sin(2\pi t)$, $\alpha_2(t) = 1+0.05\cos(2\pi t)$, $N_1=32$, $N_2=48$, $\phi_1(t) = t+0.001\sin(2\pi t)$, and $\phi_2(t)= t+0.001\cos(2\pi t)$. 

A clean signal and a noisy signal with $\SNR =-3\text{ dB}$ were analyzed. The SSWPT in \cite{1DSSWPT} was applied to estimate fundamental instantaneous properties and results are shown in Figure \ref{fig:14_1} and \ref{fig:14_2}. Inputing these properties in the RDBR, we obtained estimated shape functions contained in $f(t)$ as shown in Figure \ref{fig:14}. Note that even though the estimated instantaneous properties have large errors, especially in the noisy case, the RDBR is still able to give reasonably good shape estimation. 

\begin{figure}[ht!]
  \begin{center}
    \begin{tabular}{c}
      \includegraphics[height=1.3in]{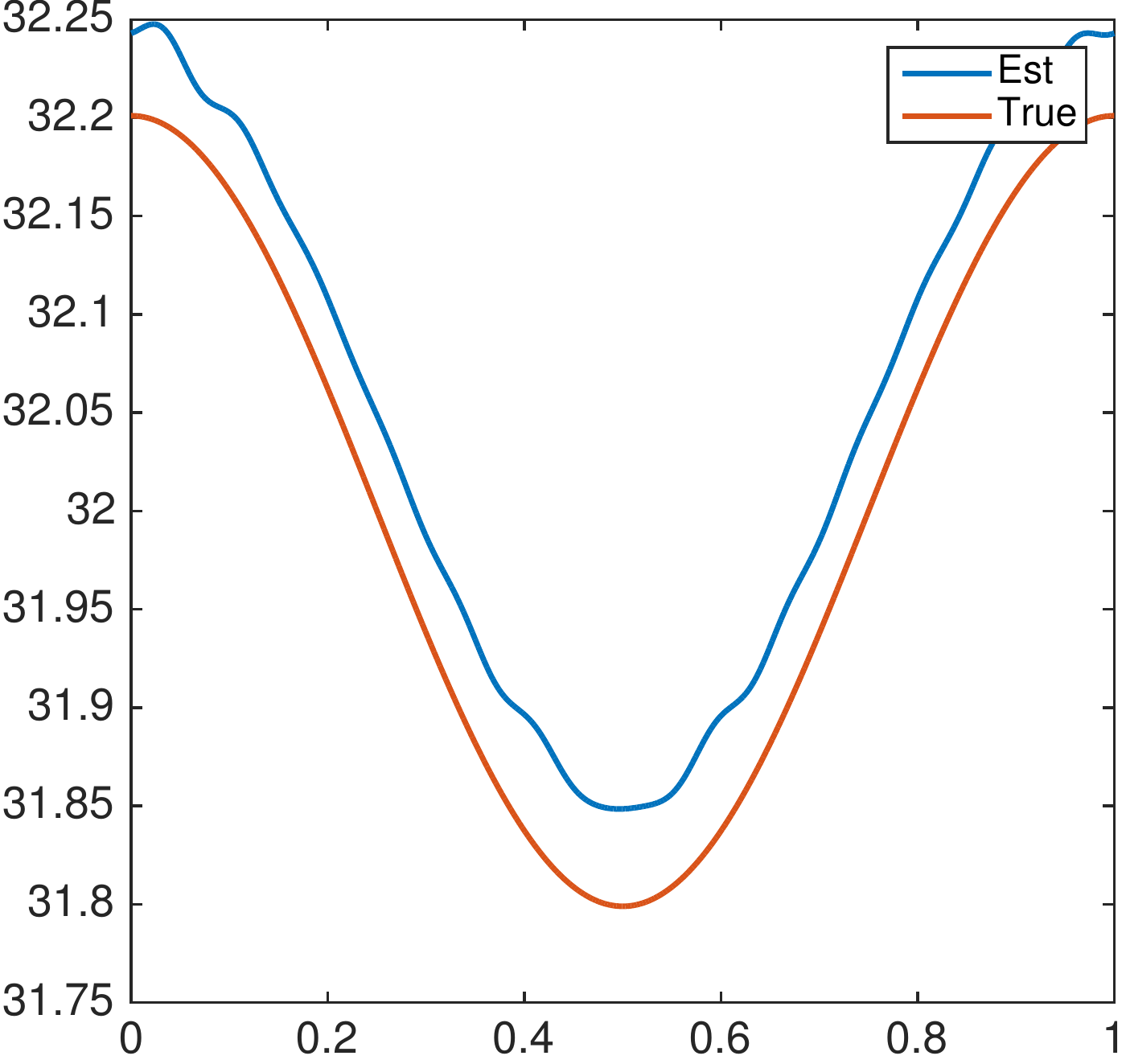}  
   \includegraphics[height=1.3in]{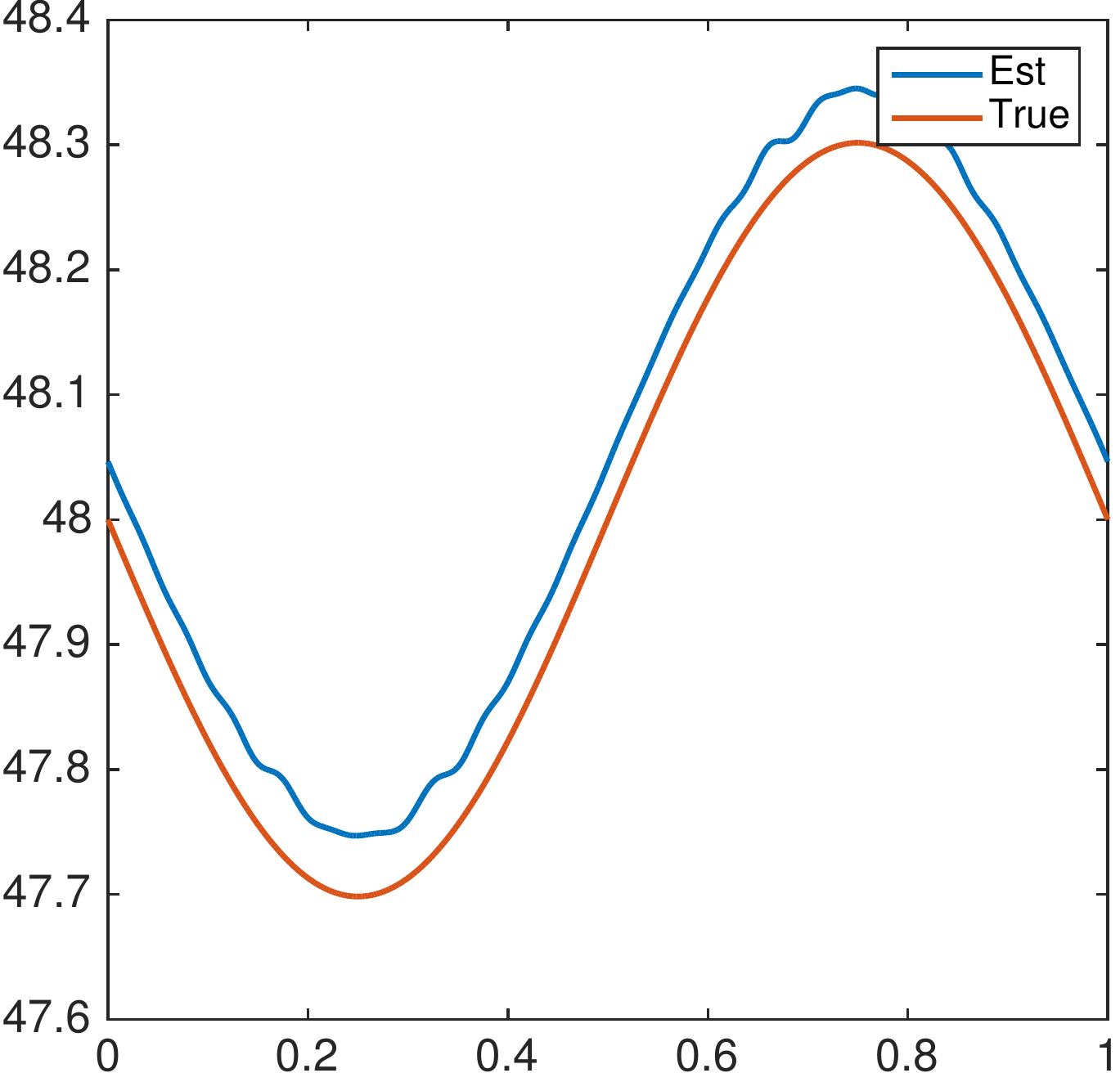}    \hspace{1cm}
   \includegraphics[height=1.3in]{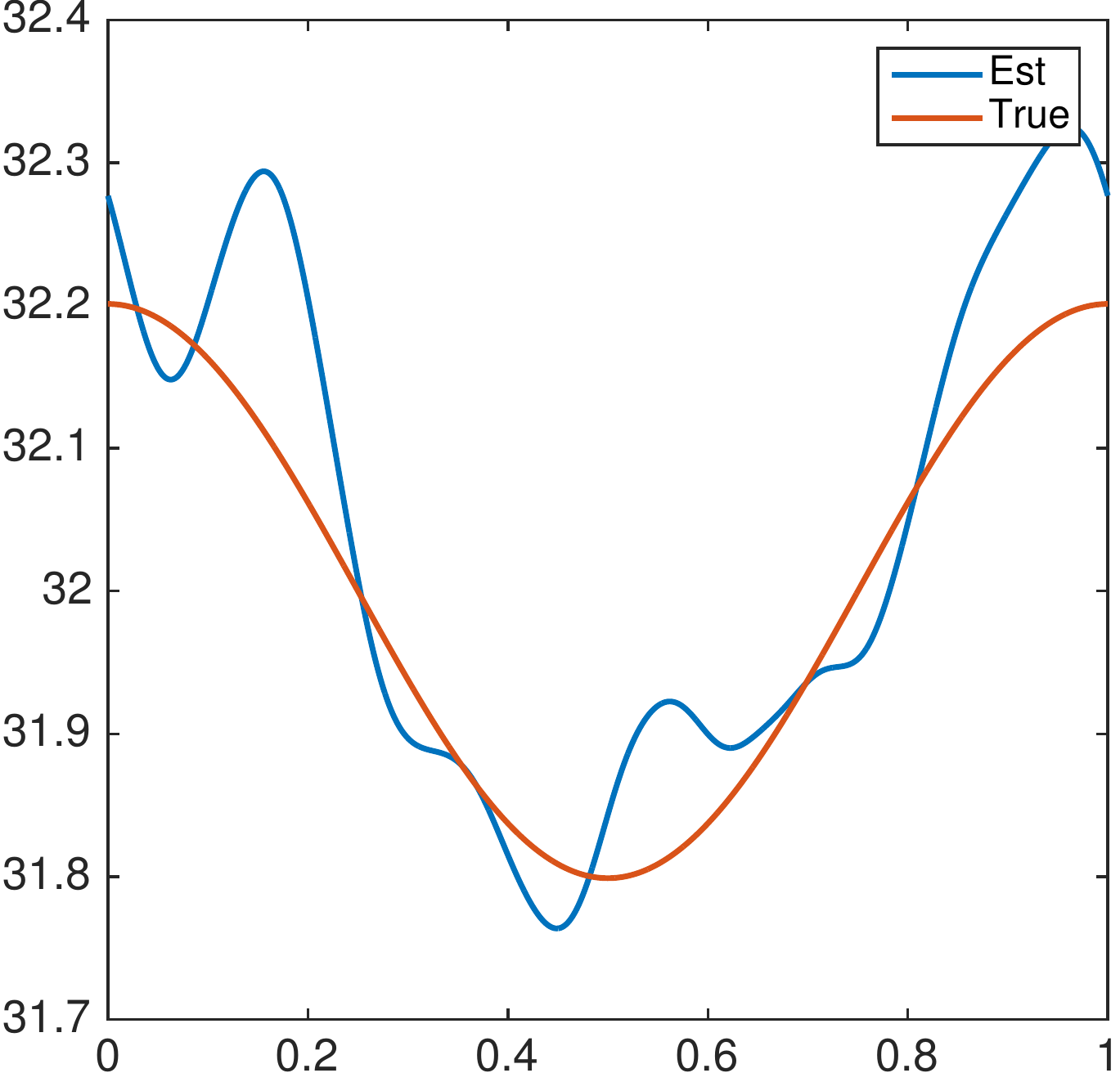}  
      \includegraphics[height=1.3in]{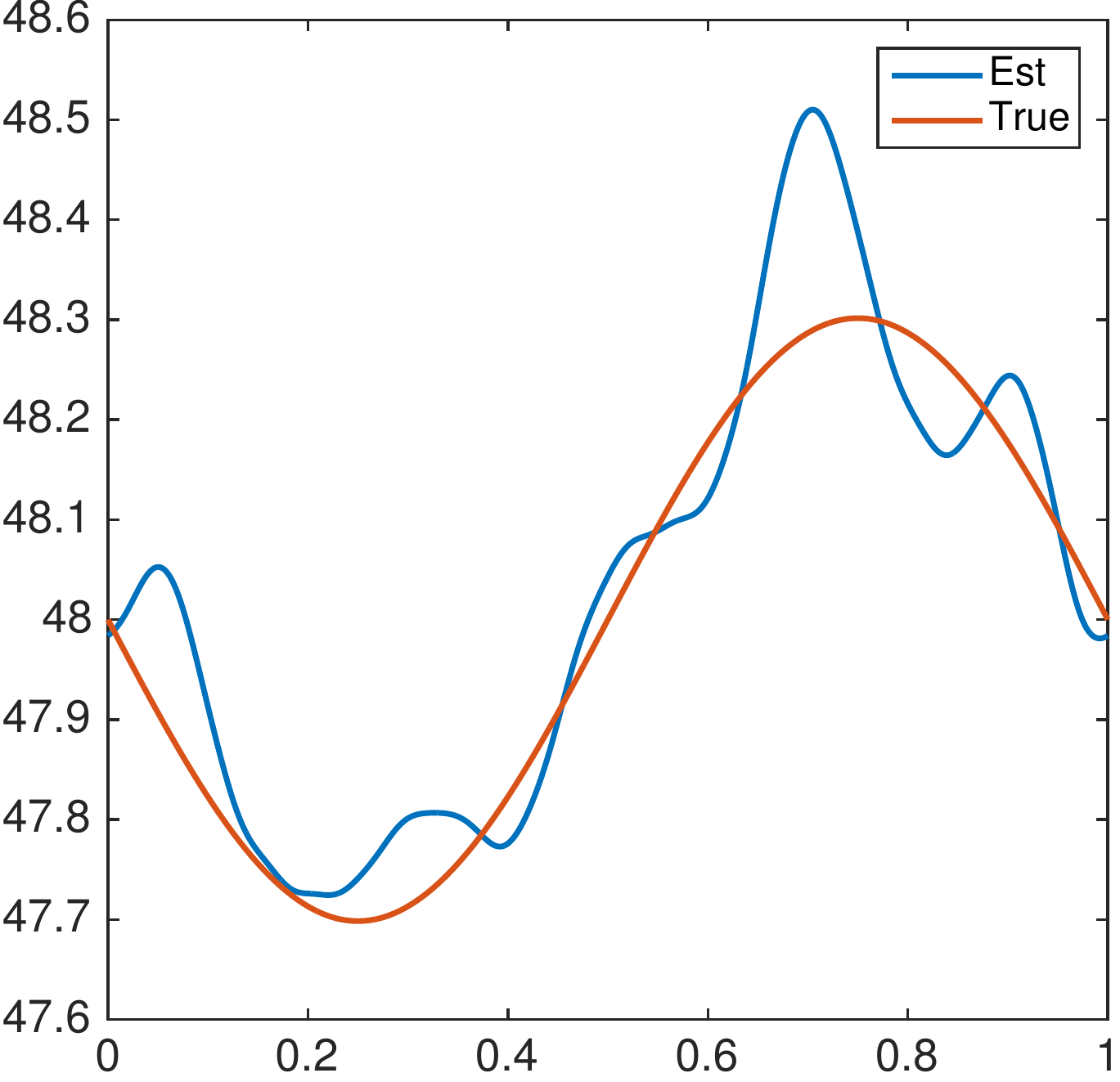}  
    \end{tabular}
  \end{center}
  \caption{Recovered fundamental instantaneous frequencies in clean (left) and noisy (right) examples by the synchrosqueezed transform in \cite{1DSSWPT}, as compared with ground truth shapes.}
\label{fig:14_1}
\end{figure}

\begin{figure}[ht!]
  \begin{center}
    \begin{tabular}{c}
      \includegraphics[height=1.3in]{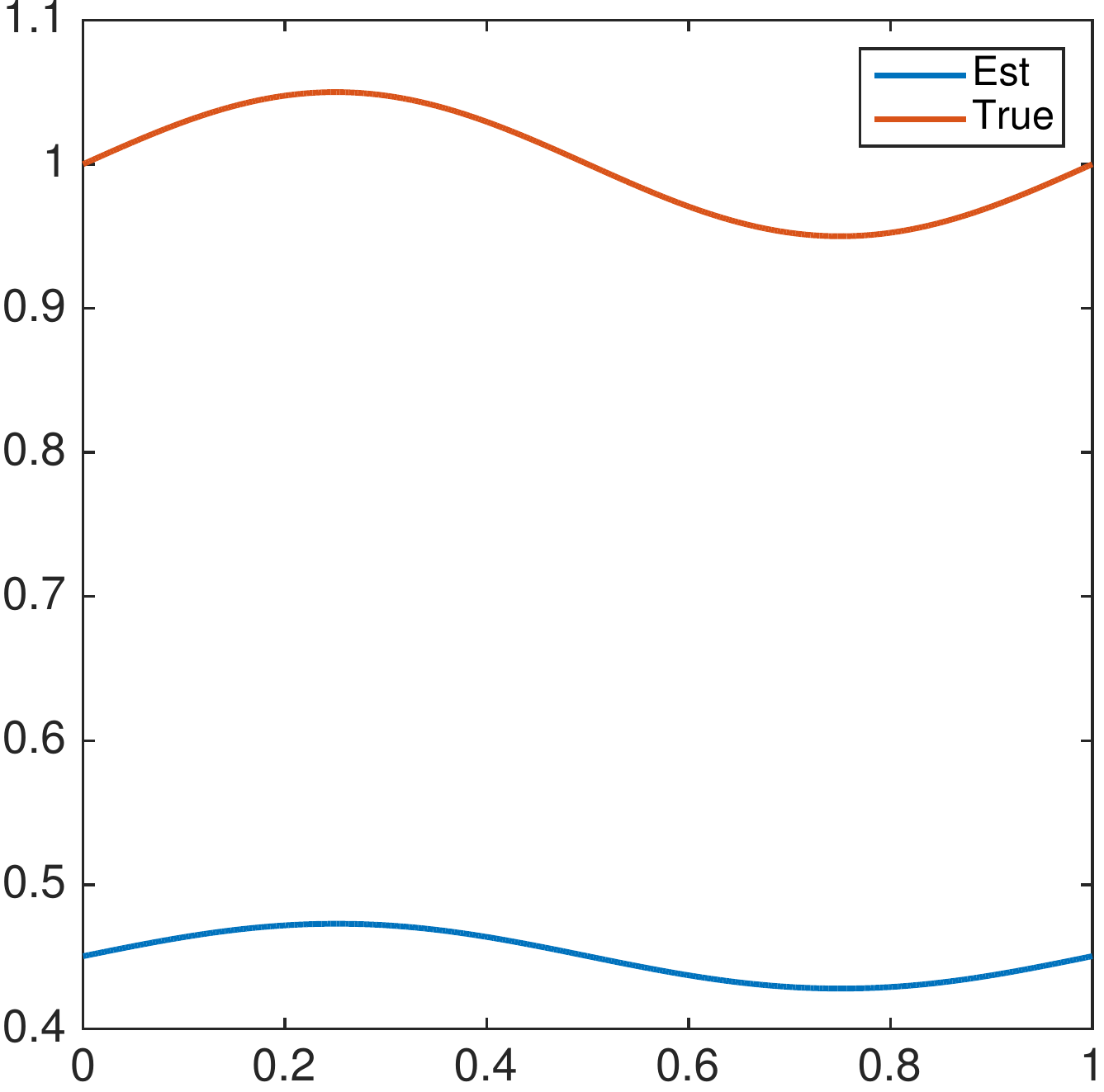}  
   \includegraphics[height=1.3in]{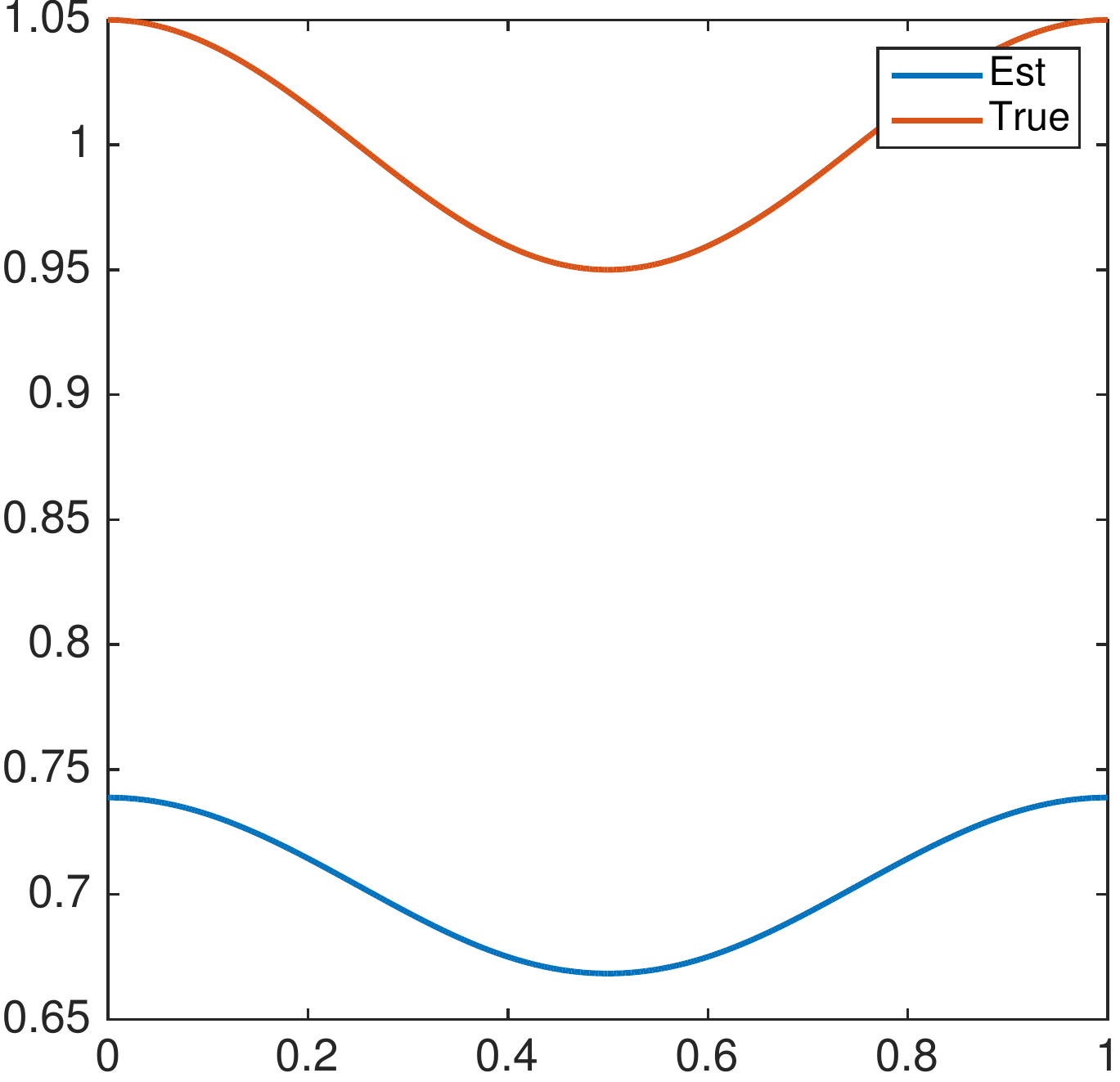}    \hspace{1cm}
   \includegraphics[height=1.3in]{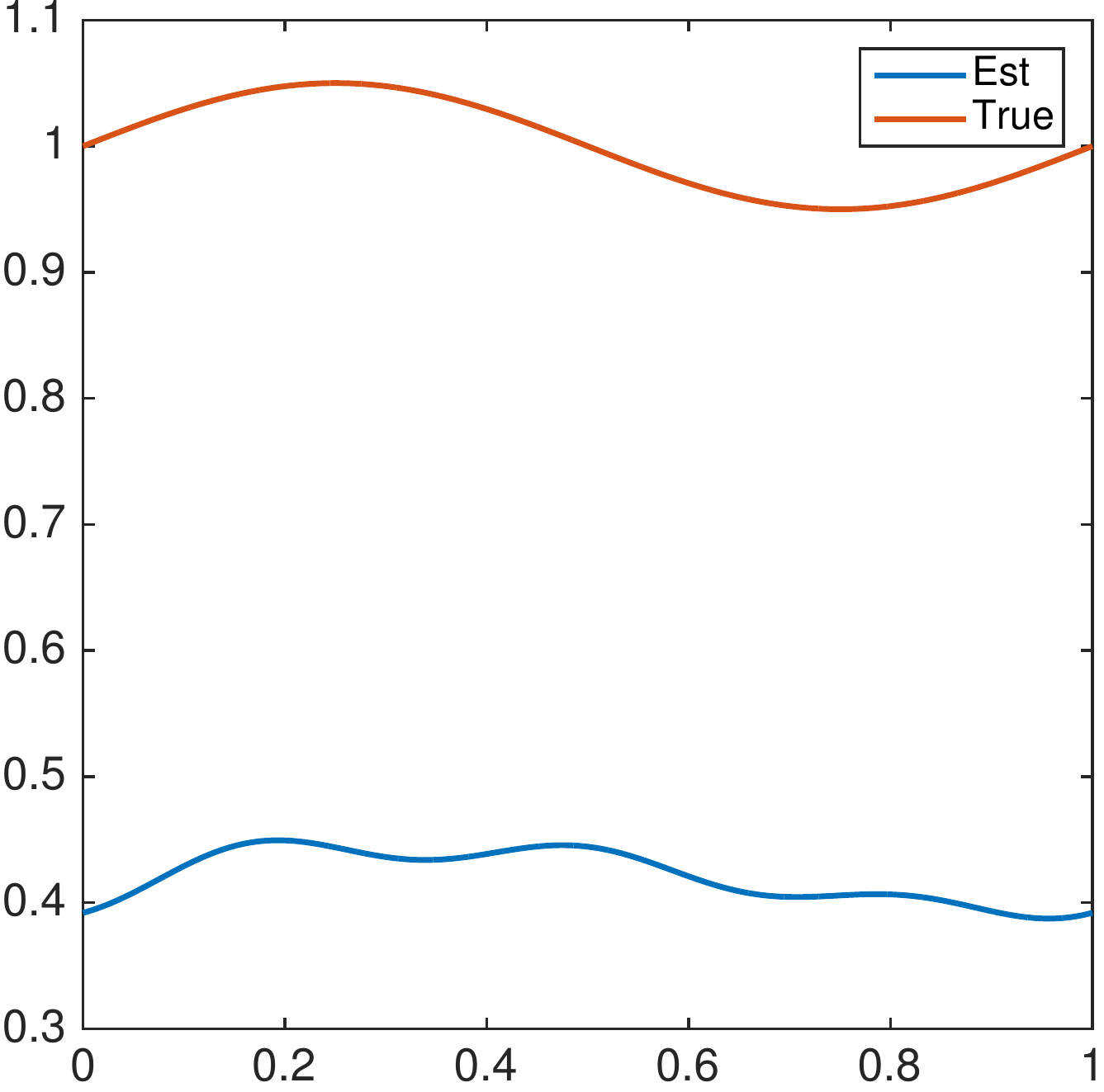}  
      \includegraphics[height=1.3in]{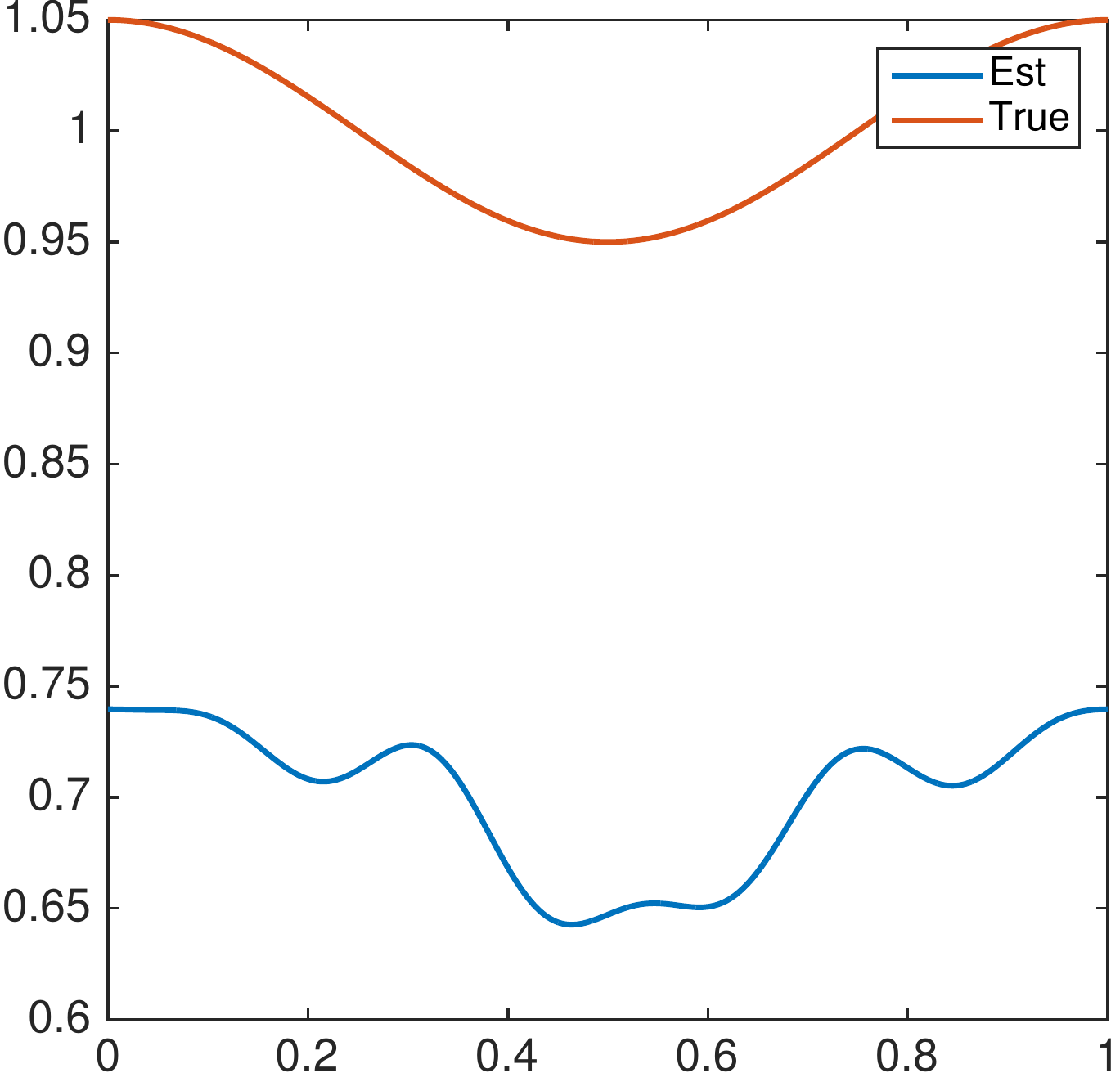}  
    \end{tabular}
  \end{center}
  \caption{Recovered fundamental instantaneous amplitudes (up to an unknown factor) in clean (left) and noisy (right) examples by the synchrosqueezed transform in \cite{1DSSWPT}, as compared with ground truth shapes.}
\label{fig:14_2}
\end{figure}

\begin{figure}[ht!]
  \begin{center}
    \begin{tabular}{c}
      \includegraphics[height=1.3in]{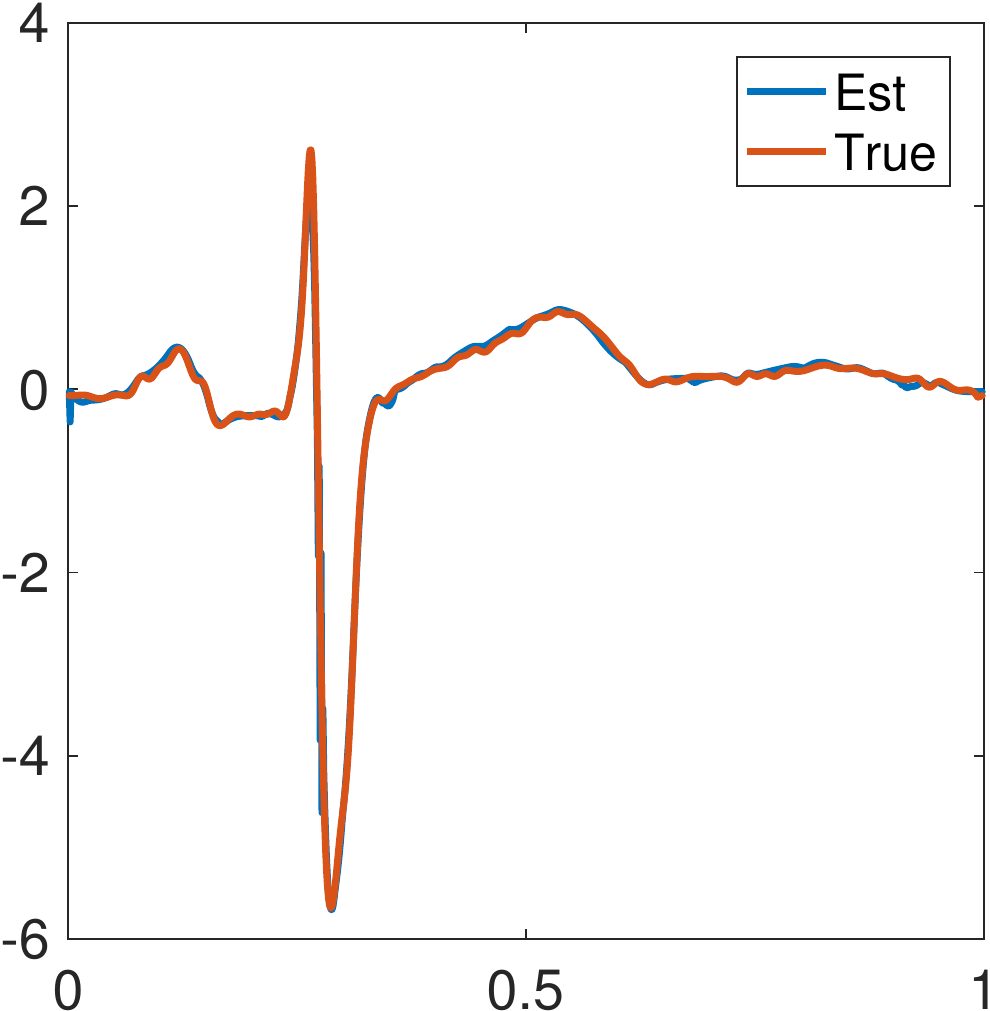}  
   \includegraphics[height=1.3in]{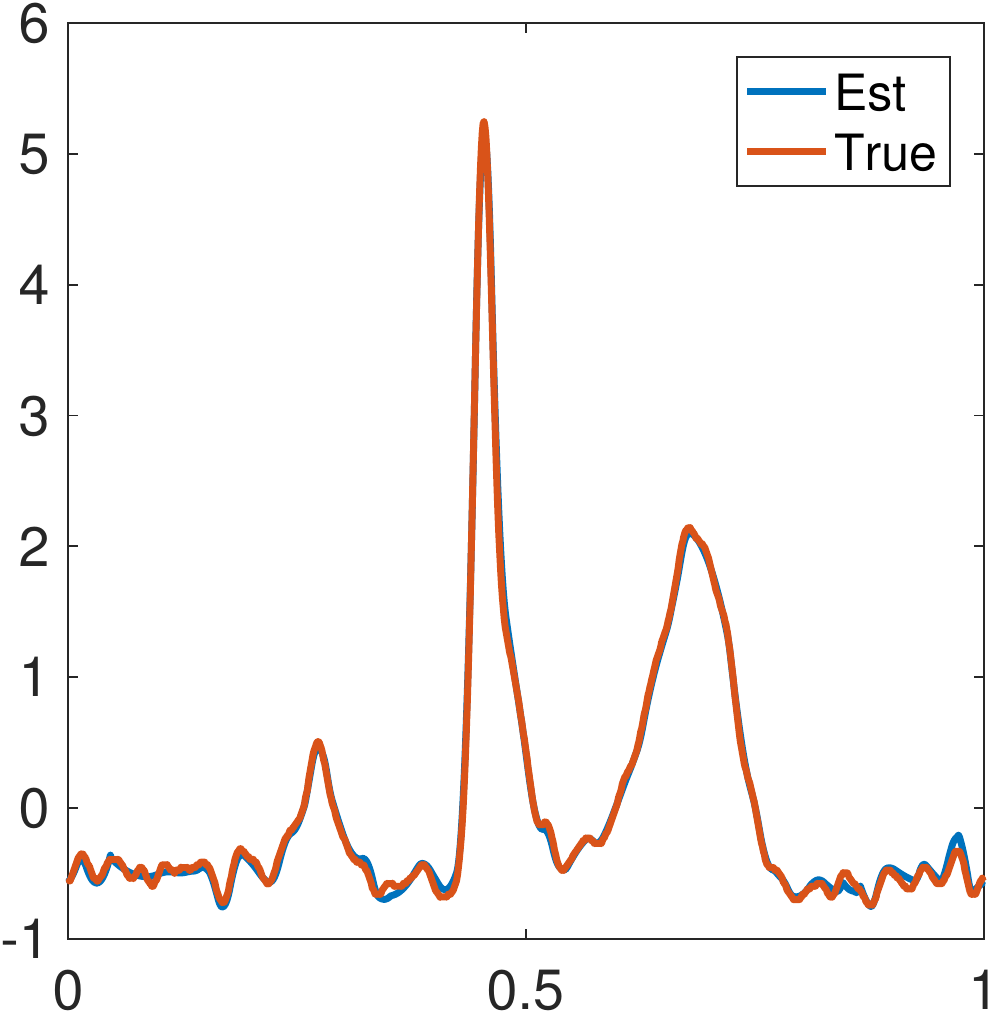}    \hspace{1cm}
   \includegraphics[height=1.3in]{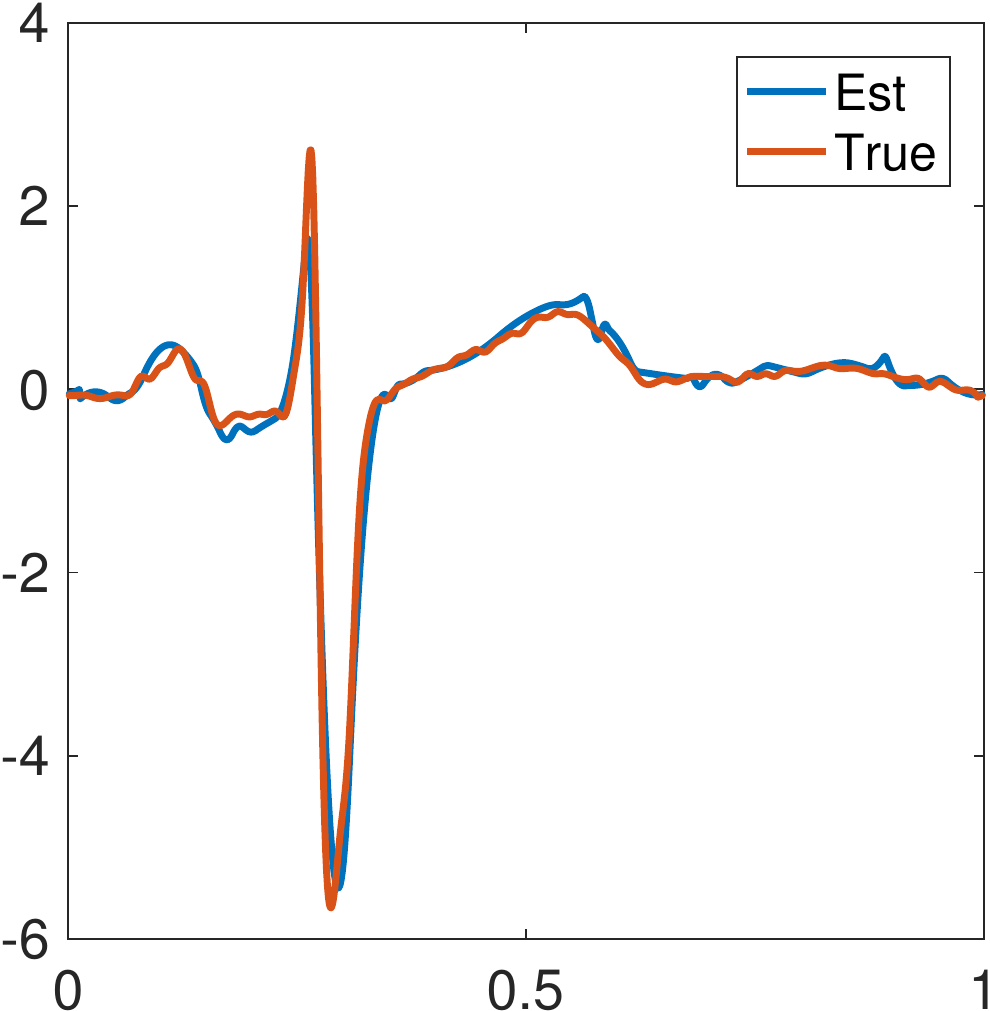}  
      \includegraphics[height=1.3in]{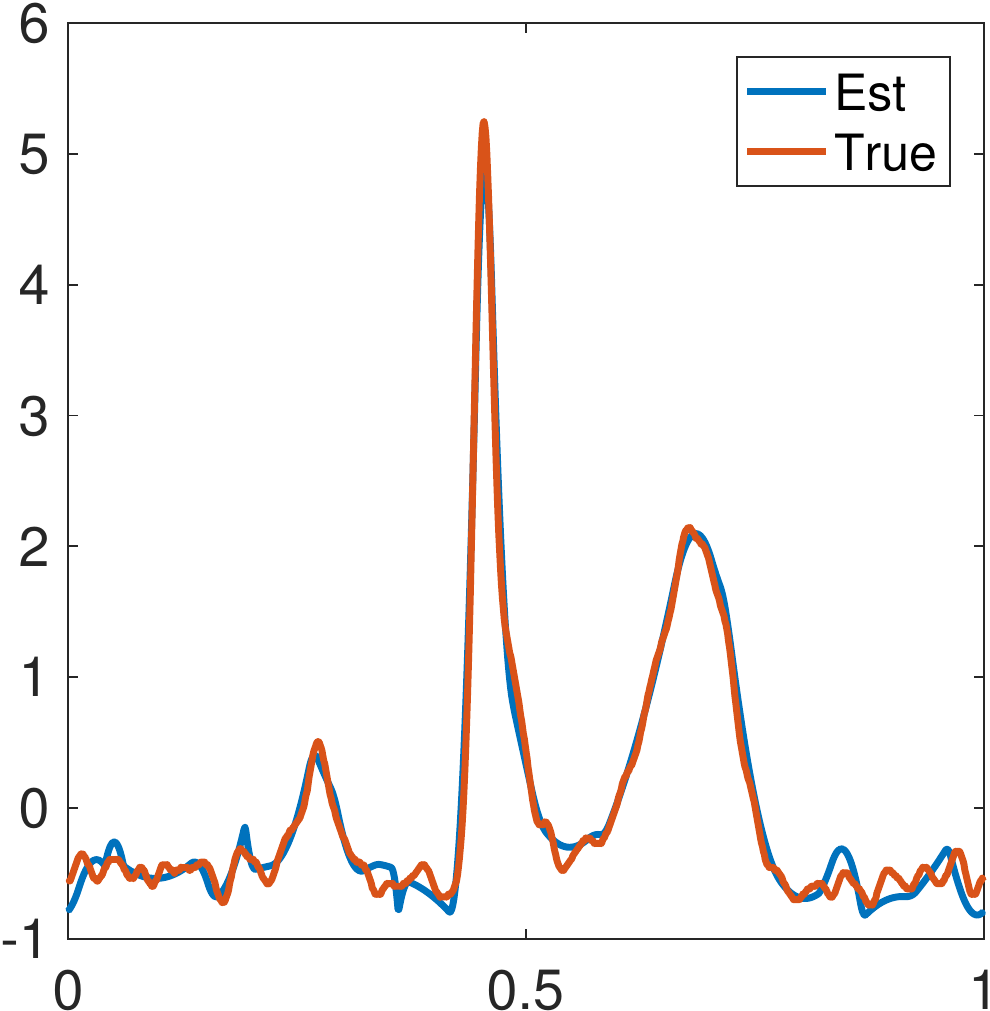}  
    \end{tabular}
  \end{center}
  \caption{Recovered shapes in clean (left) and noisy (right) examples, as compared with ground truth shapes.}
\label{fig:14}
\end{figure}

In the last example, we apply the RDBR to analyze daily atmospheric CO$_2$ concentration data in \cite{PhysicalAnal}. The data were observed by National Oceanic and Atmospheric Administration at Mauna Loa (MLO) in recent 31 years (1981-2011). As shown in Figure \ref{fig:15_1} (top), there is a smooth growing trend in the original data. To focus on the oscillatory pattern, this trend is approximated by a linear function and removed from the original data. The SSWPT is applied to estimate fundamental instantaneous properties of the residual data. As shown in Figure \ref{fig:15} (left), the synchrosqueezed energy distribution indicates only one fundamental component. The semiannual component has a instantaneous frequency that is nearly twice of the one of the annual component. The RDBR is applied to the residual data with the estimated fundamental properties by the SSWPT. Figure \ref{fig:15} (right) shows the estimated shape function contained in the residual data. This shape function reflects a nonlinear evolution pattern in a year: the CO$_2$ concentration usually increases in a longer period and decreases in a shorter period. As explained in \cite{PhysicalAnal}, this special pattern comes from seasonal photosynthetic drawdown and respiratory release of CO2 by terrestrial ecosystems.

\begin{figure}[ht!]
  \begin{center}
    \begin{tabular}{c}
      \includegraphics[height=1.6in]{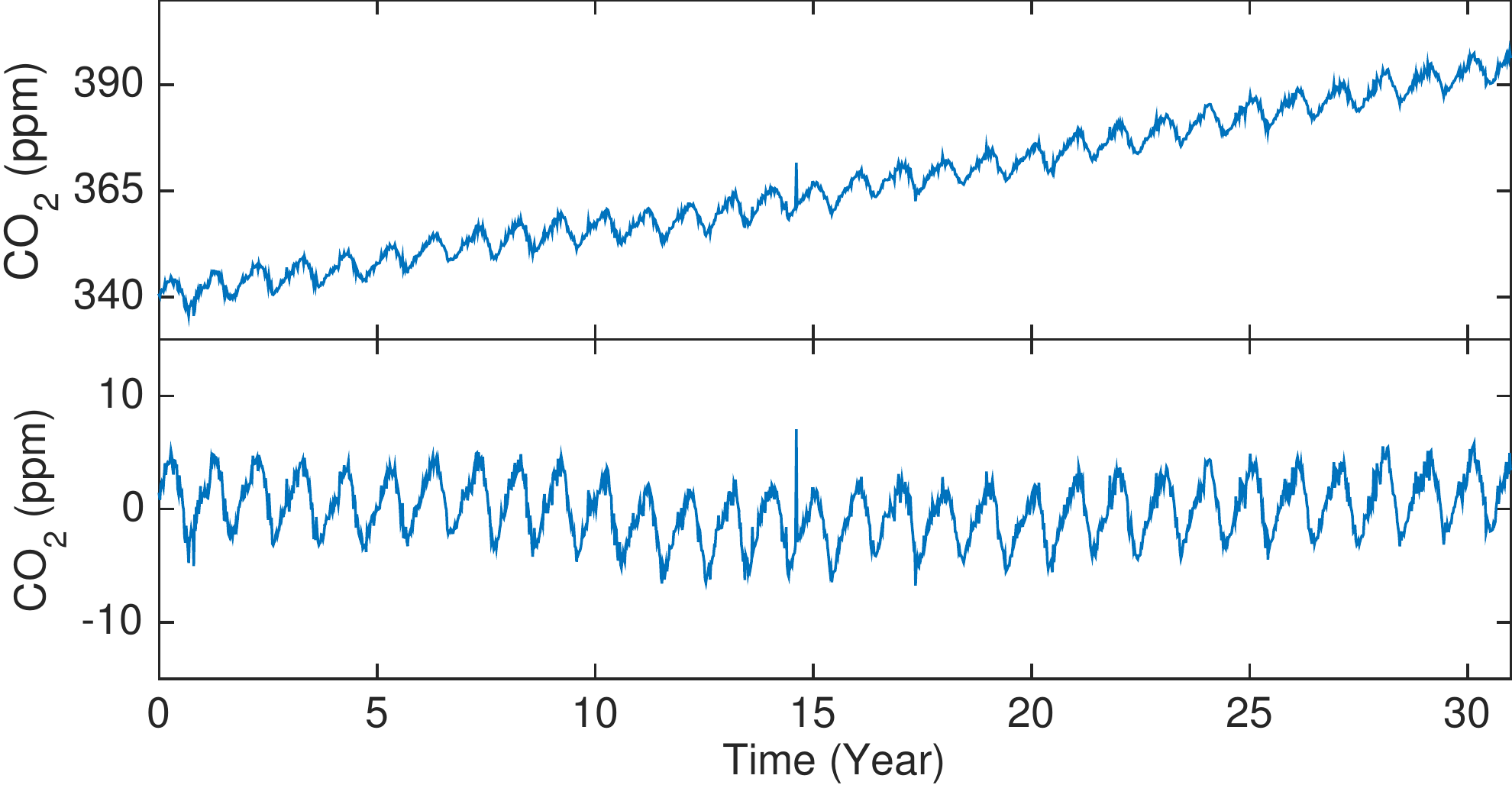} 
    \end{tabular}
  \end{center}
  \caption{Top: original CO$_2$ concentration data. Bottom: the residual CO$_2$ concentration data after removing a smooth trend.}
\label{fig:15_1}
\end{figure}

\begin{figure}[ht!]
  \begin{center}
    \begin{tabular}{cc}
      \includegraphics[height=1.6in]{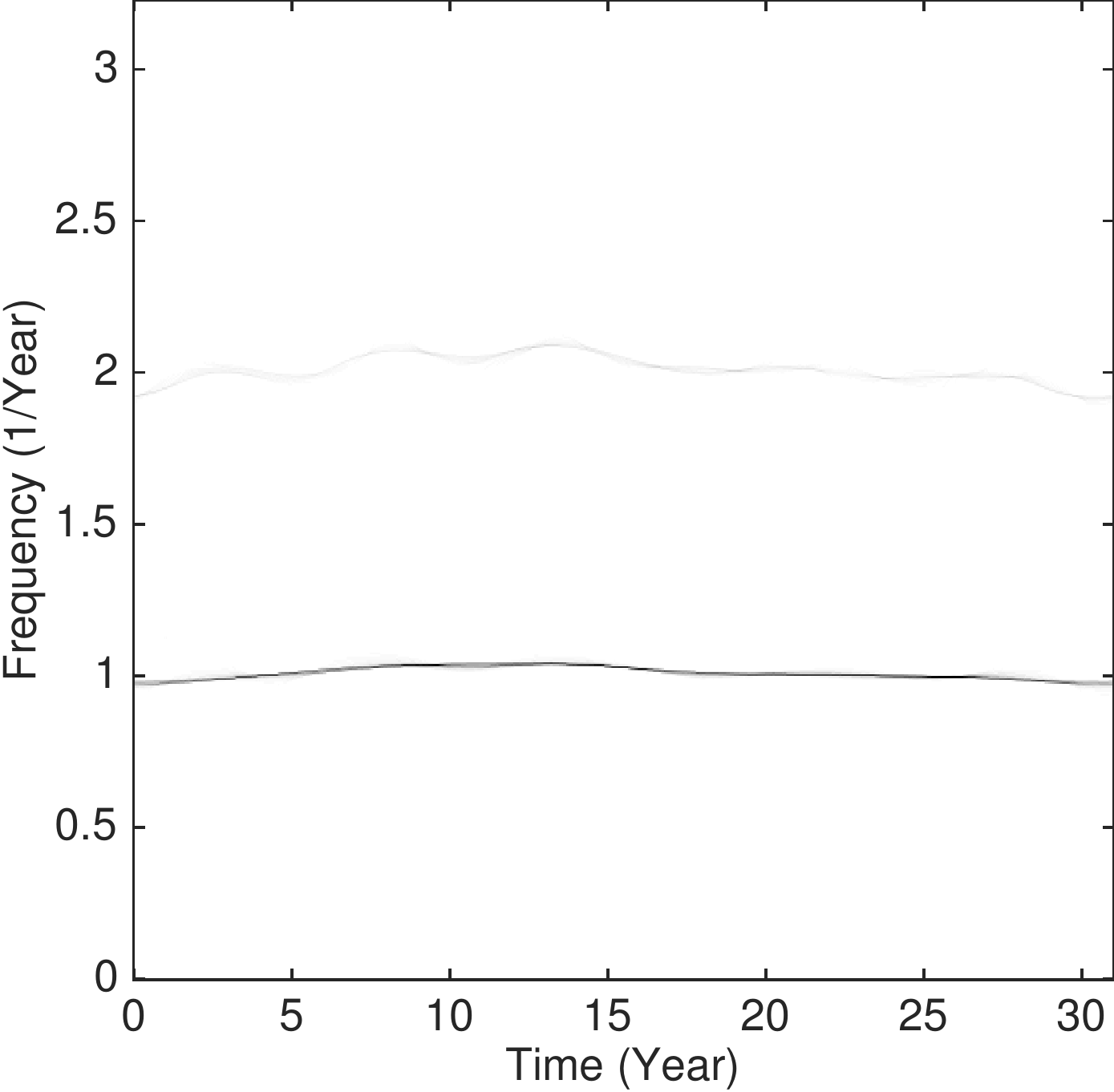}  &
   \includegraphics[height=1.6in]{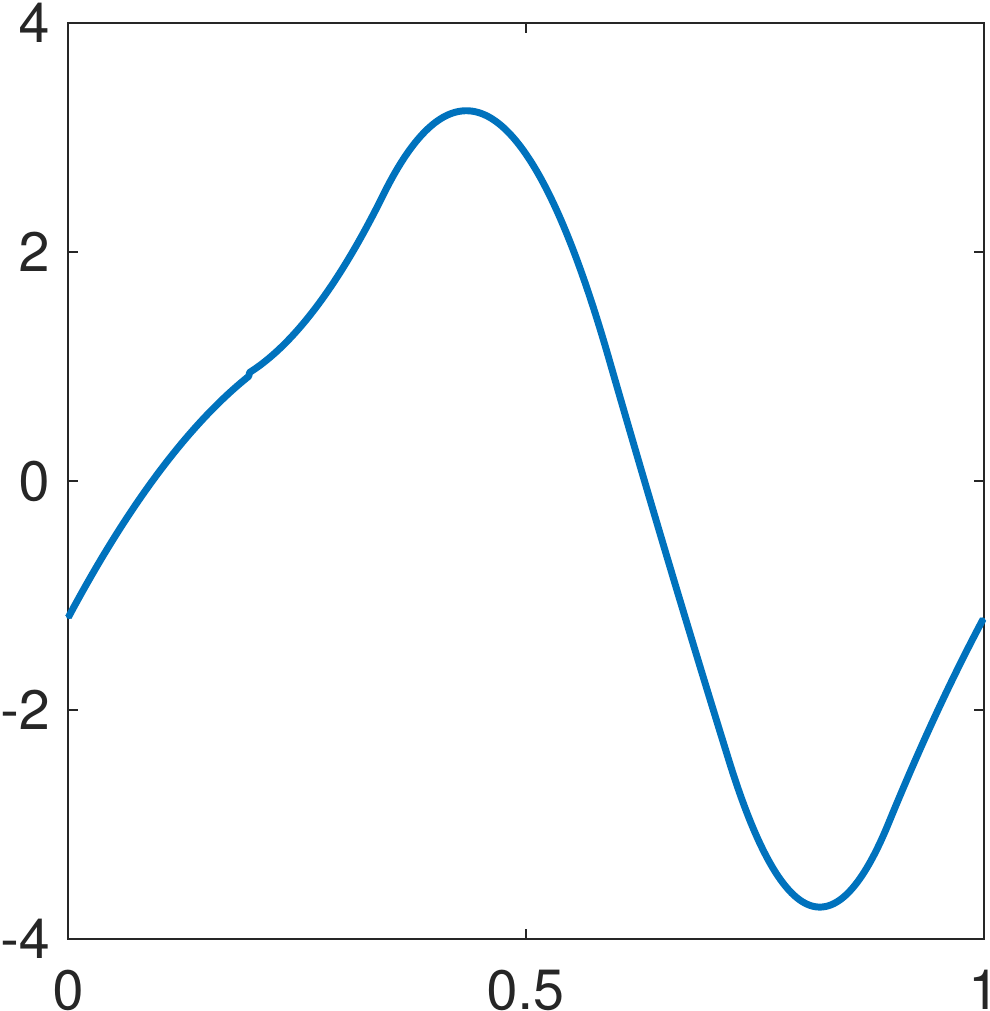} 
    \end{tabular}
  \end{center}
  \caption{Left: the synchrosqueezed energy distribution of the CO$_2$ concentration data in the low frequency part. Right: estimated shape function by the RDBR.}
\label{fig:15}
\end{figure}

\section{Conclusion}
\label{sec:conclusion}

This paper introduced a recursive diffeomorphism-based regression method (RDBR) for estimating shape functions from a superposition of generalized intrinsic mode type functions (GIMT). Combining the RDBR with other methods for estimating instantaneous properties of GIMT's, namely, synchrosqueezed transforms \cite{Daubechies2011}, adaptive optimization \cite{VMD,Hou2012}, recursive filtering \cite{iterativeFilter2,iterativeFilter1}, we provide an alternative solution to the generalized mode decomposition problem. As we have shown theoretically and numerically, once the instantaneous properties are accurate, the RDBR is a precise and robust method to estimate shape function, as long as instantaneous phases of these oscillatory modes are well-differentiated. The convergence of the RDBR is linear if instantaneous frequencies are sufficiently large. Numerical observation suggests that the RDBR converges sublinearly if instantaneous frequencies are small.

{\bf Acknowledgments.} 
H.Y. thanks the support of the AMS-Simons Travel Award and the National Science Foundation under grants ACI-1450280, and the startup grant from the Department of Mathematrics at the National University of Singapore.

\bibliographystyle{unsrt} 
\bibliography{ref}

\end{document}